\documentclass[12pt]{amsart}

\usepackage[hmargin=2.3cm,bmargin=2.3cm,tmargin=2.9cm]{geometry}
\usepackage[usenames]{color}

\usepackage{amsmath,amssymb,amsthm,amsfonts,enumerate,url,mathbbol,mathrsfs}
\usepackage{esint,tikz,graphicx,pifont,float}
\usepackage{hyperref}
\usepackage{cancel}

\usepackage[normalem]{ulem}
\usepackage[utf8]{inputenc}

\theoremstyle{plain}
\newtheorem{theorem}{Theorem}[section]
\newtheorem{lemma}[theorem]{Lemma}
\newtheorem{proposition}[theorem]{Proposition}
\newtheorem{corollary}[theorem]{Corollary}

\theoremstyle{definition}
\newtheorem{definition}[theorem]{Definition}
\newtheorem{example}[theorem]{Example}
\newtheorem{conjecture}[theorem]{Conjecture}

\newtheorem{assumption}[theorem]{Assumption}
\newtheorem{remark}[theorem]{Remark}

\numberwithin{equation}{section}
\numberwithin{figure}{section}

\DeclareMathOperator{\interior}{int}
\DeclareMathOperator{\dist}{dist}
\DeclareMathOperator{\supp}{supp}

\newcommand{\R}{\mathbb{R}}
\newcommand{\N}{\mathbb{N}}

\newcommand{\Graph}{\mathcal{G}} 
\newcommand{\NewGraph}{\widetilde{\mathcal{G}}} 
\newcommand{\HGraph}{\mathcal{H}} 
\newcommand{\NewHGraph}{\widetilde{\mathcal{H}}} 
\newcommand{\EdgeSet}{\mathcal{E}} 
\newcommand{\cardE}{M} 
\newcommand{\VertexSet}{\mathcal{V}} 
\newcommand{\cardV}{N} 
\newcommand{\DG}{\mathsf{G}} 
\newcommand{\DE}{\mathsf{E}} 
\newcommand{\DV}{\mathsf{V}} 
\newcommand{\De}{\mathsf{e}} 

\newcommand{\mG}{\mathsf{G}}

\newcommand{\Tree}{{\mathcal T}} 

\newcommand{\parti}{{\mathcal P}} 
\newcommand{\nenergy[1]}{{\Lambda}^N_{#1}} 
\newcommand{\denergy[1]}{{\Lambda}^D_{#1}} 
\newcommand{\noptenergy[1]}{{\mathcal L}^N_{#1}} 
\newcommand{\doptenergy[1]}{{\mathcal L}^D_{#1}} 
\newcommand{\nminenergy}{{\Xi}^N} 
\newcommand{\dminenergy}{{\Xi}^D} 
\newcommand{\nmaxmin[1]}{{\mathcal M}^N_{#1}} 
\newcommand{\dmaxmin[1]}{{\mathcal M}^D_{#1}} 

\newcommand{\noptenergyloose[1]}{{\widetilde{\mathcal L}}^N_{#1}} 
\newcommand{\doptenergyloose[1]}{{\widetilde{\mathcal L}}^D_{#1}} 

\newcommand{\cutset}{\mathcal{C}} 
\newcommand{\colourset}{\mathfrak{T}} 

\def\:{\thinspace:\thinspace}

\title{A theory of spectral partitions of metric graphs} 



\author[J.~B.~Kennedy]{James B.~Kennedy}
\author[P.~Kurasov]{Pavel Kurasov}
\author[C.~L{\'e}na]{Corentin L{\'e}na}
\author[D.~Mugnolo]{Delio Mugnolo}

\address{James B.~Kennedy, Grupo de F\'isica Matem\'atica, Faculdade de Ci\^encias, Universidade de Lisboa, Campo Grande, Edif\'icio C6, P-1749-016 Lisboa, Portugal}
\email{jbkennedy@fc.ul.pt}

\address{Pavel Kurasov, Department of Mathematics, Stockholm University, SE-106 91 Stockholm, Sweden}
\email{kurasov@math.su.se}

\address{Corentin L{\'e}na, Department of Mathematics, Stockholm University, SE-106 91 Stockholm, Sweden}
\email{corentin@math.su.se}

\address{Delio Mugnolo, Lehrgebiet Analysis, Fakult\"at Mathematik und Informatik, Fern\-Universit\"at in Hagen, D-58084 Hagen, Germany}
\email{delio.mugnolo@fernuni-hagen.de}

\thanks{The authors were partially supported by the Center for
  Interdisciplinary Research (ZiF) in Bielefeld within the framework of
  the cooperation group on ``Discrete and continuous models in the
  theory of networks''. 
  All the authors would like to acknowledge networking support by the COST Action CA18232. 
  The work of J.B.K.~was supported by the 
  Funda\c{c}\~ao para a Ci\^encia e a Tecnologia, Portugal, 
  via the program ``Investigador FCT'', reference IF/01461/2015, and 
  via project PTDC/MAT-CAL/4334/2014. P.K. was also supported by 
  the Swedish Research Council grant D0497301.
  C.L. was partially supported by the Funda\c{c}\~ao para a Ci\^encia e a Tecnologia, Portugal (under the project OPTFORMA, IF/00177/2013) and the Swedish Research Council (under the grant D0497301). D.M. was also supported by the Deutsche Forschungsgemeinschaft (Grant 397230547).
  The authors would like to thank Matthias Hofmann (Lisbon) and Marvin Pl\"umer (Hagen) for a number of 
  helpful comments and suggestions, including M.H. for the idea for 
  Example~\ref{ex:matthias-exhaustive-example}.}

\begin{document}

\begin{abstract}
We introduce an abstract framework for the study of clustering in metric graphs: after suitably metrising the space of graph partitions, we restrict Laplacians to the clusters thus arising and use their spectral gaps to define several notions of partition energies; this is the graph counterpart of the well-known theory of spectral minimal partitions on planar domains and includes the setting in [Band \textit{et al}, Comm.\ Math.\ Phys.\ \textbf{311} (2012), 815--838] as a special case. We focus on the existence of optimisers for a large class of functionals defined on such partitions, but also study their qualitative properties, including stability, regularity, and parameter dependence. We also discuss in detail their interplay with the theory of nodal partitions.
Unlike in the case of domains, the one-dimensional setting of metric graphs allows for explicit computation and analytic -- rather than numerical -- results. Not only do we recover the main assertions in the theory of spectral minimal partitions on domains, as studied in [Conti \textit{et al}, Calc.\ Var.\ \textbf{22} (2005), 45--72; Helffer \textit{et al}, Ann.\ Inst.\ Henri Poincar\'e Anal.\ Non Lin\'eaire \textbf{26} (2009), 101--138], but we can also generalise some of them and answer (the graph counterparts of) a few open questions.
\end{abstract}

\maketitle

\tableofcontents

\section{Introduction} 
\label{sec:intro}

How to cut a connected object -- be it a manifold, a domain, or a graph -- into $k$ mutually disjoint, connected parts?
Partitioning of objects is a natural topic in geometry and has important consequences in applied sciences, like data analysis or image segmentation. This article is devoted to the issue of graph partitioning via properties of eigenvalues.

Consider a connected, compact metric graph~\cite{BerKuc13,Mug14,Kur20}: the lowest eigenvalue of the Laplacian with natural vertex conditions is 0, with the constant functions as associated eigenfunctions. Thus, all further eigenfunctions are orthogonal to the constants and hence sign-changing: it is conceivable to use the support of their positive and negative parts as a natural splitting of the graph, similar to the classical Cheeger approach, and indeed this has been discussed by several authors~\cite{Nic87,Pos12,Kur13,DPR16_clustering,KenMug16}. 

More generally, one can consider the eigenfunctions associated with the $k$-th eigenvalue and, inspired by Sturm's Oscillation Theorem (or in higher dimensions by Courant's Nodal Domain Theorem), hope that they deliver a reasonable splitting into $k$ subsets, which we then interpret as \textit{clusters}. However, {\it a priori} there is no reason for this splitting to result in $k$ pieces: the zeros of the  eigenfunction may divide the graph into fewer -- or more -- than $k$ pieces, or nodal domains. (Similar issues have been observed in the case of discrete graphs by Davies et al.\ in~\cite{DavGlaLey01}, and recently extended to the case of general quadratic forms generating positive semigroups in~\cite{KelSch19}.)
Indeed, accurate estimates on the number $\nu_k$ of nodal domains of a $k$-th eigenfunction involve the topology of the metric graph (via its first Betti number) and have been proved in~\cite{GnuSmiWeb04,Ber08,Ban14}, see~\cite[\S~5.2]{BerKuc13} for an overview.

A different approach to the task of splitting a \emph{domain} $\Omega\subset\R^2$ into precisely $k$ connected pieces was introduced in~\cite{ConTerVer05} and amounts to seeking those $k$-partitions $(\Omega_i)_{1\le i\le k}$ of $\Omega$ such that, upon imposing Dirichlet conditions at the boundary of each $\Omega_i$ and considering the lowest eigenvalue $\lambda_1(\Omega_i)$ of the Dirichlet Laplacian on $\Omega_i$, a certain function -- typically the maximum -- of the vector $(\lambda_1(\Omega_i))_{1\le i\le k}$ is minimised. The attention devoted to such \emph{spectral minimal partitions} was greatly boosted when a connection to nodal domains \`a la Courant was established in \cite{HelHofTer09}. In another direction, related to Cheeger partitions and free discontinuity problems, the authors of \cite{BuFrGi18} proved existence and some regularity for minimal partitions associated with the Robin Laplacian.

The aim of the present article is to develop a comprehensive theory of spectral partitions on metric graphs. In the case of planar domains, whenever defining a spectral partition the only issue is to make sure that each subdomain can be associated with a Dirichlet eigenvalue; hence, it is sufficient and also natural to consider those partitionings consisting of open, mutually disjoint sets $\Omega_i$ whose closures are contained in the closure of $\Omega$; the boundaries of nodal domains turn out to be smooth curves -- along which Dirichlet conditions are imposed -- meeting at a finite number of points. But when it comes to metric graphs, there is no such natural definition of partitioning. A first major issue is how to proceed when cutting through vertices, more specifically: what transformations of the connectivity of a piece, or \emph{cluster}, $\Graph_i$ should be allowed for, should $\Graph_i$ contain a vertex that (as a vertex in the given metric graph $\Graph$) lies at the boundary between it and a further cluster $\Graph_j$. While the few existing works that have considered metric graph partitioning, in particular~\cite{BanBerRaz12}, avoided this problem by only allowing for cuts through the interiors of edges, or equivalently vertices of degree 2, we will define several ``partitioning'' rules of increasing generality, based on which types of cuts may be allowed: the most important ones give rise to what we call \textit{rigid} and \textit{loose} partitions. This requires a whole new theory, which we illustrate in Section~\ref{sec:part} with the help of a few elementary examples the most ubiquitous of which is a lasso graph, demonstrating that loose and rigid partitions may be considered natural objects.

The main existence result in the theory of spectral minimal partitions on domains was obtained in~\cite[\S~2]{ConTerVer05} by variational methods; a similar -- but technically much more involved -- approach was recently used in~\cite{BuFrGi18} to deal with the eigenvalues of Robin Laplacians. The natural constraints studied there are however too rough for our setting, and we have to proceed differently.
Given a compact metric graph $\mathcal G$, we introduce in Section~\ref{sec:abstract-existence} a Polish metric space $\mathfrak P(\mathcal G)$ of its partitions, study general lower semi-continuous functionals with respect to the induced topology, and discuss qualitative properties of their minima. Our abstract approach pays off: among other things, loose and rigid partitions are indeed natural simply because -- unlike many other, ostensibly natural, classes -- they define closed subsets of the metric space $\mathfrak P(\mathcal G)$ and are hence particularly suitable for minimisation purposes.

This theoretical toolbox is then applied in Section~\ref{sec:ex-min} to several classes of optimal partition problems, including those appearing in earlier studies on nodal partitions, like~\cite{BanBerRaz12}. By checking that the relevant energy functionals actually satisfy the (rather mild) sufficient conditions introduced in Section~\ref{sec:exis}, we can finally prove existence of optimal partitions for Laplacians with either Dirichlet or natural vertex conditions. (Needless to say, minimising among rigid or loose partitions will generally yield different optima, an issue we will touch upon in Section~\ref{sec:lndmnd}.) Our investigations show that both optimisation probelms are well-motivated. Dirichlet conditions at the cut points -- the classical choice in the earlier literature, both on domains and metric graphs -- are naturally related to the issue of nodal domains, a connection that led to the very birth of this field in~\cite{ConTerVer05}. Imposing natural conditions at the cut points, on the other hand, will be shown to lead to well-posed spectral problems whose minimizing partitions consist of clusters that are \textit{connected} in a more straightforward, apparent sense.
These results can be further generalised considering graph counterparts of the energy functionals first introduced in~\cite{ConTerVer05} -- essentially, the mean value of $p$-th powers of spectral gaps of a suitable Laplacian, defined clusterwise; this amounts to studying minima of functionals $\Lambda^D_p$ and $\Lambda^N_p$, $p\in (0,\infty]$, defined on the partition space $\mathfrak P(\mathcal G)$. 
Similar ideas may also be developed if more general conditions -- say, $\delta$-couplings -- should imposed at the cut points, analogously to what was done in \cite{BuFrGi18} for the case of domains, although we will not develop such ideas here. 
Neumann domains, a Neumann-type analogue of nodal domains, have been studied recently on quantum graphs~\cite{AloBan19,AlBaBeEg18}, and it is natural to ask whether there is a similar link between these and Neumann-type partitions as there is between Dirichlet partitions and nodal domains. We also leave this question to future work.
Also, one could in principle study other spectral quantities, for example by considering higher eigenvalues. 
We leave these as open problems to be discussed in later investigations.

Beginning with~\cite{HelHofTer09}, much research has been devoted not just to the issue of regularity of spectral minimal partitions of planar domains, but also to the shape of the partition elements. In view of the Faber--Krahn inequality, all subdomains of $\Omega$ would try to get as close as possible to a disc in order to minimise the lowest Dirichlet eigenvalue. A spectral minimal $k$-partition is such that all these subdomains can, roughly speaking, find an optimal compromise: this is conjectured to lead asymptotically (for large $k$) to a hexagonal tiling of $\Omega$, see~\cite[\S~10.9]{BNHe17}.
What do optimal \textit{graph} partitions with respect to the energy functionals $\Lambda^D_p$ and $\Lambda^N_p$ look like? Section~\ref{sec:nodal} is devoted to this question.
Some properties of spectral minimal partitions for the Dirichlet Laplacian have been already discussed by Band \textit{et al} \cite{BanBerRaz12} \emph{under the assumption that they actually exist}, and provided there is a ground state that does not vanish at any vertex. We can generalise some of their findings by dropping this structural assumption, which in fact typically fails if the graph is highly symmetric. 
While one simple optimisation problem does actually deliver a $2$-partition that agrees precisely with the nodal domains of a second eigenfunction of the Laplacian with natural transmission conditions, an exact characterisation of $k$-partitions minimising any reasonable energy functional for $k\ge 3$ is still missing. This problem is not trivial in the case of spectral partitioning of a domain $\Omega$, either. In the case of general metric graphs the nice qualitative description provided in~\cite{HelHofTer09} for the domain case does not hold anymore; however, there is some evidence that, for $k$ large enough, any optimal partition will consist of a collection of stars (which in view of their minimising properties~\cite{Fri05} can be regarded as graph analogues of disks, or hexagons). We content ourselves with showing that the interplay between Dirichlet minimal partitions and nodal domains does hold for trees, at least; we briefly summarise similarities and differences between the qualitative properties of spectral minimal partitions on planar domains and metric graphs in Section~\ref{sec:domgra}.

To show the flexibility of our approach, analogous spectral partitions that \textit{maximise} two different energy functionals $\Xi^D_p$ and $\Xi^N_p$  are discussed in Section~\ref{sec:ex-max} by showing that they also satisfy our basic topological assumptions.

In Section~\ref{sec:examples} we turn to the issue of the dependence of minimisers of $\Lambda^D_p$ and $\Lambda^N_p$ on $p$, and also on the edge lengths of the underlying metric graph $\Graph$, for a fixed topology. Here the simplicity of the 1-dimensional setting of metric graphs is highly advantageous: while such a $p$-dependence is only numerically observed in the case of domains, we study in detail spectral minimal partitions of a $3$-star and show their dependence of $p$ in an analytic way.

Finally, in Section~\ref{sec:comparison} we present examples comparing the different optimisation problems (for $\Lambda^D_p$, $\Lambda^N_p$, $\Xi^D_p$ and $\Xi^N_p$) and the corresponding optimal energies on a fixed graph: these different problems tend, naturally, to split the graph in different ways. Here we present a couple of heuristic conjectures based on our examples; in future work we intend to return to the question of how these different problems behave in a more rigorous and complete way.

To enhance the readability of the article and for ease of reference, the following table collects a number of new notions and symbols used throughout the paper.

\begin{center}
\begin{tabular}{c|l|l}
Symbol & Description/name & See\\
\hline
&\\[-8pt]
$\Graph$, $\NewGraph$, $\Graph'$ & metric graph, ur-graph & Sec.~\ref{subsec:def}, Def.~\ref{def:urg}\\
& canonical representative & Def.~\ref{def:urg}\\
$\DG$ & underlying discrete (ur-)graph & Def.~\ref{def:discrete-graph}, Def.~\ref{def:urg}\\
$\lambda_1(\Graph)$, $\lambda_1(\Graph;\VertexSet_D)$ & first Dirichlet eigenvalue & Eq.~\eqref{eq:lambda-1}\\
$\mu_2(\Graph)$ & first nontrivial natural eigenvalue & Eq.~\eqref{eq:mu-2}\\
& cut of a graph & Def.~\ref{def:graph-cut-new1}\\
& (nontrivial) cut through a vertex & Def.~\ref{def:graph-cut-new2}\\
$\parti = \{\Graph_1,\ldots,\Graph_k\}$ & ($k$-)partition & Def.~\ref{def:partition}\\
$\Graph_i$ & cluster of a partition & Def.~\ref{def:partition}\\
$\Omega_i$ & cluster support corresponding to $\Graph_i$ & Def.~\ref{def:cluster-support}\\
$\Omega$ & partition support & Def.~\ref{def:cluster-support}\\
$\mathcal{C}(\parti)$ & cut set (cut points) of $\parti$ & Def.~\ref{def:separation-set}\\
$\VertexSet_D(\Graph_i)$ & set of cut points in $\Graph_i$ & Def.~\ref{def:separation-set}\\
$\partial\parti$ & separation set (points) of $\parti$ & Def.~\ref{def:separation-set}\\
& neighbour, neighbouring cluster & Def.~\ref{def:neighbours}\\
& loose, rigid, faithful, internally & \\
& \qquad connected, proper partition & Def.~\ref{def:classification}\\
$\mathfrak{P}$, $\mathfrak{P}_k$ & set of loose ($k$-)partitions & Eq.~\eqref{eq:loose-and-rigid}\\
$\mathfrak{R}$, $\mathfrak{R}_k$ & set of rigid ($k$-)partitions & Eq.~\eqref{eq:loose-and-rigid}\\
$\rho_{\Omega_i}$ & set of possible rigid clusters for $\Omega_i$ & Eq.~\eqref{eq:rigid-cluster-set}\\
$\mathcal C$ & cut pattern, similar partition & Def.~\ref{defi:simcut}\\
$\colourset$, $\colourset_{\cutset}$ & primitive partition (associated with cut pattern $\mathcal C$) & Def.~\ref{def:primitive-partition}\\
$\Gamma_{\DG}$ & set of ur-graphs for $\DG$ & Sec.~\ref{subsec:part-convergence}\\
$d_{\Gamma_{\DG}}(\Graph,\NewGraph)$ & distance between metric graphs\\
& \qquad with same discrete ur-graph & Sec.~\ref{subsec:part-convergence}\\
$[v]$ & equivalence class of vertices converging to $v$ & Sec.~\ref{subsec:part-convergence}\\
& (strongly) lower semi-continuous & Def.~\ref{def:lsc}\\
$\denergy[p](\parti)$, $\nenergy[p] (\parti)$ & Dirichlet, natural partition energy & Eq.~\eqref{eq:nenergy}, Eq.~\eqref{eq:denergy}\\
$\doptenergy[k,p](\Graph)$, $\noptenergy[k,p](\Graph)$ & rigid Dirichlet, natural minimal energy & Eq.~\eqref{eq:minimal-energies}\\
$\doptenergyloose[k,p](\Graph)$, $\noptenergyloose[k,p](\Graph)$ & loose Dirichlet, natural minimal energy & Eq.~\eqref{eq:minimal-energies}\\
& Dirichlet, natural ($k$-)equipartition & Def.~\ref{def:equipartition}\\
& nodal, generalised nodal, bipartite partition & Def.~\ref{def:nodal-partition}, Def.~\ref{def:generalised-nodal}, Def.~\ref{def:proximity}\\
$\nu$, $\nu(\psi)$ & number of nodal domains (of $\psi$) & Prop.~\ref{prop:weak-courant}\\
$\dminenergy(\parti)$, $\nminenergy(\parti)$ & min.\ Dirichlet, natural partition energy & Eq.~\eqref{eq:dminenergy}, Eq.~\eqref{eq:nminenergy}\\
$\dmaxmin[k](\Graph)$, $\nmaxmin[k](\Graph)$ & Dirichlet, natural max-min energy & Eq.~\eqref{max-min-energies}\\
\end{tabular}
\end{center}

\section{Graphs and partitions}
\label{sec:partitions}

\subsection{Basic definitions}
\label{subsec:def}

We start with the metric graphs we shall be considering; it will be necessary to consider the formalism we will be using in some detail, which mirrors the one used in \cite{KuSt02}. By a metric graph $\Graph = (\VertexSet,\EdgeSet)$ we understand a pair consisting of a \emph{vertex set} $\VertexSet = \VertexSet (\Graph) = \{v_1,\ldots,v_{\cardV}\}$ and an \emph{edge set} $\EdgeSet = \EdgeSet (\Graph) = \{e_1,\ldots,e_{\cardE}\}$; throughout the paper we will always assume these to be finite sets.

Each edge $e_m = e_m (\Graph)$, $m=1,\ldots,M$, is identified with a compact interval $[x_{2m-1},x_{2m}] \subset \R$ of length $|e_m| = x_{2m}-x_{2m-1}$ belonging to a separate copy of $\R$. Each edge should connect two vertices: formally, we introduce an equivalence relation on the set of endpoints $\{x_j\}_{j=1}^{2{\cardE}}$, thus partitioning it into nonempty, mutually disjoint sets
\begin{equation} \label{eqrel}
	\{x_j\}_{j=1}^{2{\cardE}} = V_1 \cup \ldots \cup V_{\cardV} .
\end{equation}
The vertex $v_n \in \VertexSet (\Graph)$ is identified with the set $V_n = V_n (\Graph)$. 

If $x_{2m-1},x_{2m} \in V_{m_1} \cup V_{m_2}$ for some $m$, then we write $e_m \equiv v_{m_1}v_{m_2}$ and  in this case we say that $e_m$ is \textit{incident} with the vertices $v_{m_1},v_{m_2}$. We refer to the cardinality of $V_n$ as the \emph{degree} of $v_n$, written $\deg v_n$.

Vertices of degree two are allowed, but are called \emph{dummy vertices}; these can be introduced and removed at will without altering any of the properties of the graph (in particular the spectral quantities) in which we will be interested, as we shall discuss below.
\emph{Loops}, that is, edges incident with only one vertex, and multiple edges, that is, distinct edges incident with the same pair of vertices, are also allowed.

Any metric graph $\Graph = (\VertexSet,\EdgeSet)$ will be identified with a \emph{set} of equivalence classes of points
by extending the equivalence relation  \eqref{eqrel} to all points in the interior of each edge (interval); this will be done by associating with any $x \in \interior e_m = (x_{2m-1},x_{2m})$
equivalence class $\{x\}$ formed by one element. 
With this in mind, in future we will take points $x \in \Graph$, and in particular regard the vertices $v_n$ as points in the set, without further comment. \emph{However}, for some purposes it is important to remember that $v_n$ and $V_n$ are different objects; indeed, our theory relies essentially upon the possibility to \textit{cut through a vertex $v_n$} by subdividing $V_n$ into two or more nonempty, mutually disjoint subsets.

We will write $|\Graph| = \sum_{e \in \EdgeSet} |e| = \sum_{m=1}^M |e_m|$ for the finite total length of the graph, the sum of the lengths of the edges. We refer to the monographs \cite{BerKuc13,Mug14} for more information on metric graphs in general.

A metric graph has both an underlying discrete structure and a notion of distance defined on it, and both will be important to us. The interested reader may take~\cite{Die05} as a standard reference for classical notions and results in combinatorial graph theory (e.g., regarding \textit{subdivisions} as will appear).

\begin{definition}
\label{def:discrete-graph}
Given a metric graph $\Graph = (\VertexSet,\EdgeSet)$ the \emph{underlying discrete graph} (or \emph{associated discrete graph}) is the discrete graph $\DG = (\DV, \DE)$ for which there are bijections $\Phi: \DV \to \VertexSet$ and $\Psi: \DE \to \EdgeSet$ such that for all $\De \in \DE$ and all $e \in \EdgeSet$, $\Psi(\De)=e$ implies the vertices incident with $\De$ in $\DG$ are mapped by $\Phi$ to the vertices incident with $e$ in $\Graph$. If $\Psi(\De)=e$, then we say that $\De$ and $e$ \emph{correspond} to each other.
\end{definition}

We next recall how a canonical distance is introduced on metric graphs. Given $x,y\in \mathcal G$, we take $\dist(x,y)$ to be the minimal length among all paths connecting $x$ with $y$, see~\cite[Def.~3.14]{Mug14} for details. If the graph is not connected then we set the distance between points belonging to different connected components to be infinity.
Throughout this paper we always consider on $\Graph$ the topology induced by this distance: in particular, given a subset $\Omega$ of $\Graph$ we can consider its interior
\[
	\interior \Omega :=\left\{ x \in \Omega : \{ y \in \Graph: \dist (x,y) < \varepsilon \} \subset \Omega \text{ for some } \varepsilon > 0 \right\},
\]
and its boundary
\begin{equation}
\label{eq:elt-boundary-set}
	\partial \Omega := \overline{\Omega}\setminus\interior\Omega.
\end{equation}
(Regarding terminology, we consider $\partial \Omega$ to consist of those points that \textit{separate} $\Omega$ from $\Graph\setminus \Omega$.) Equipped with the distance function, each metric graph is a metric space.

Summarising, we shall assume throughout that:

\begin{assumption}
\label{assumption}
The metric graph $\Graph$ is finite compact and  connected, {\it i.e.} the vertex set is finite, 
the edge set is finite, each edge has finite length and there is a continuous path connecting any
two points on the graph. 
\end{assumption}

The only assumption we will occasionally have cause to drop is of connectedness, but in such cases we will always state this explicitly.

\emph{Isometric isomorphisms}, i.e., bijective mappings between metric graphs (even those with different edge sets!) that preserve distances, define an equivalence relation $\approx$ on the class of all metric graphs which satisfy Assumption~\ref{assumption}. If two graphs are isometrically isomorphic to each other, then we are in one or both of the following situations:
\begin{enumerate}
\item the edge and vertex sets of one are a permutation (i.e. relabelling) of the edge and vertex sets of the other;
\item the graphs differ by the presence of dummy vertices.
\end{enumerate}
See also \cite[Definition~5]{KuSt02} and the discussion around it. We will need to make the following definition for technical purposes, which will be necessary for the constructions in the coming sections up to and including Section~\ref{sec:abstract-existence}.

\begin{definition}\label{def:urg}
\begin{enumerate}
\item We call any equivalence class of metric graphs satisfying Assumption~\ref{assumption} with respect to $\approx$, an \emph{ur-graph}.
\item If $\Graph$ is an ur-graph, then its \emph{canonical representative} is the metric graph representative of $\Graph$ which has no vertices of degree two (or, if $\Graph$ is a loop, then its canonical representative is any representative with exactly one vertex of degree two).
\item We will call the underlying discrete graph of the canonical representative of an ur-graph $\Graph$ the \emph{underlying discrete ur-graph} of $\Graph$ (or \emph{discrete ur-graph associated with $\Graph$}).
\end{enumerate}
\end{definition}

In practice, we will not distinguish between different representatives of the same ur-graph; indeed, for spectral analysis different representatives of an ur-graph are indistinguishable (see Remark~\ref{rem:dummy}). We will tacitly tend to identify an ur-graph $\Graph$ with any of its representatives as convenient, most commonly (but not always) its canonical representative. We will thus also speak of ur-graphs as being compact metric spaces, and as satisfying Assumption~\ref{assumption}, etc.

There is a canonical notion of (scalar-valued) continuous functions over $\Graph$ with respect to the distance defined above, and we stress that this notion is invariant under taking different representatives of the same ur-graph.

Similarly, the Lebesgue measure, defined edgewise, induces in a canonical way a measure on $\Graph$, allowing us to define square integrable functions on $\Graph$:
\begin{displaymath}
\begin{aligned}
	L^2 (\Graph) &:= \bigoplus_{e_m \in \EdgeSet} L^2(e_m) \simeq \bigoplus_{m=1}^{\cardE} L^2 ([x_{2m-1},x_{2m}]),\\
	C(\Graph) &:= \left\{f \in \bigoplus_{e_m \in \EdgeSet} C(e_m): f(x_j)=f(x_k) =:f(v_n) \hbox{ if } x_j,x_k \in v_n\text{ for some } v_n\in \VertexSet \right\}.
\end{aligned}
\end{displaymath}

In order to define Laplacian-type operators we will require the Sobolev spaces
\begin{displaymath}
	H^1(\Graph) := \{f \in C(\Graph)\cap \bigoplus_{e_m\in\EdgeSet}  H^1(e_m)\}
\end{displaymath}
and, for a given distinguished set $\VertexSet_D \subset \VertexSet$ of vertices,
\begin{displaymath}
	H^1_0 (\Graph) \equiv H^1_0(\Graph;\VertexSet_D): = \{f \in H^1(\Graph): f(v_n)=0 \text{ for all } v_n \in \VertexSet_D\}.
\end{displaymath}
Given the sesquilinear form
\begin{equation}
\label{eq:form}
	a(f,g) = \int_\Graph f'\cdot \bar{g}'\,\textrm{d}x \simeq \sum\limits_{m=1}^{\cardE} \int_{x_{2m-1}}^{x_{2m}} f'\cdot \bar{g}'\,\textrm{d}x,\qquad f,g\in H^1(\Graph),
\end{equation}
the associated self-adjoint operator on $L^2(\Graph)$ is the Laplacian $ -\Delta = - \frac{d^2}{dx^2} $ defined on the domain of functions from $ \bigoplus_{e_m\in\EdgeSet}  H^2(e_m)
$ satisfying continuity and Kirchhoff conditions (sum of inward-pointing derivatives is zero) at every vertex. Such vertex conditions, which we will call \emph{natural}, are also known as standard, free, or sometimes Neumann--Kirchhoff conditions. The Laplacian with Dirichlet conditions on a subset $\VertexSet_D$ and natural conditions at all other vertices is the operator on $L^2(\Graph)$ which is associated with the form $a$ restricted to $H^1_0(\Graph;\VertexSet_D)$.

Due to the positivity of $a$ and the compact embedding of $H^1(\Graph)$ in $L^2(\Graph)$, the Laplacian on the connected, compact graph $\Graph$ with natural vertex conditions has a sequence of non-negative eigenvalues, which we will denote by
\begin{displaymath}
	0 = \mu_1 (\Graph) < \mu_2 (\Graph) \leq \mu_3 (\Graph) \leq \ldots \to \infty,
\end{displaymath}
repeating them according to their (finite) multiplicities; the eigenfunction corresponding to $\mu_1 = \mu_1 (\Graph)$ is just the constant function. 

We shall do likewise for the eigenvalues of the Laplacian on $\Graph$ with some Dirichlet conditions:
\begin{displaymath}
	0 < \lambda_1 (\Graph;\VertexSet_D) < \lambda_2 (\Graph;\VertexSet_D) \leq \ldots \to \infty.
\end{displaymath}
In practice we will abbreviate these to $\lambda_k (\Graph)$ or even just $\lambda_k$, $k\geq 1$, if the vertex set and the graph are clear from the context. These eigenvalues admit the usual minimax and maximin characterisations; in particular, we have
\begin{equation}
\label{eq:lambda-1}
	\lambda_1 (\Graph)= \lambda_1 (\Graph;{\mathcal V}_D) = \inf \left\{ \frac{a(f,f)}{\|f\|_{L^2(\Graph)}^2}: 0 \not\equiv f \in H^1_0 (\Graph;\mathcal V_D) \right\},
\end{equation}
while
\begin{equation}
\label{eq:mu-2}
	\mu_2 (\Graph) = \inf \left\{ \frac{a(f,f)}{\|f\|_{L^2(\Graph)}^2}: 0 \not\equiv  f \in H^1 (\Graph),\,\int_\Graph f\,\textrm{d}x=0 \right\}.
\end{equation}
In both \eqref{eq:lambda-1} and \eqref{eq:mu-2}, the infima are achieved only by the respective eigenfunctions, which are sign-changing and may be multiple in \eqref{eq:mu-2}, but are unique up to scalar multiples and  non-zero everywhere in \eqref{eq:lambda-1}.

\begin{remark}
\label{rem:dummy}
Suppose $\Graph$ and $\Graph'$ are isometrically isomorphic to each other in the sense described above. Then the respective spaces $L^2$, $C$ and $H^1$ on the two graphs are also isometrically isomorphic to each other. It follows in particular that the corresponding Laplacians with natural vertex conditions are unitarily equivalent to each other, and the respective eigenvalues are equal: $\mu_k (\Graph) = \mu_k (\Graph')$ for all $k \geq 1$. Likewise, if we fix a set $\VertexSet_D (\Graph) \subset \VertexSet (\Graph)$ of vertices of $\Graph$, and choose $\Graph'$ in such a way that the image of each point in $\VertexSet_D (\Graph)$ under the isomorphism is also a vertex of $\Graph'$, so that we may write $\VertexSet_D (\Graph) \simeq \VertexSet_D (\Graph')$, then $H^1_0 (\Graph;\VertexSet_D (\Graph))$ and $H^1_0 (\Graph'; \VertexSet_D (\Graph'))$ are also isometrically isomorphic. Thus the corresponding Dirichlet Laplacians are likewise unitarily equivalent, and $\lambda_k (\Graph; \VertexSet_D (\Graph)) = \lambda_k (\Graph'; \VertexSet_D (\Graph'))$ for all $k \geq 1$. In other words, the eigenvalues and eigenfunctions may be associated with the corresponding ur-graph; and for our purposes, within an ur-graph, i.e., an equivalence class of isometrically isomorphic graphs, we may at any time pick any representative, as convenient.
\end{remark}

Finally, we mention in passing fundamental inequalities for the eigenvalues $\lambda_1 (\Graph)$ and $\mu_2 (\Graph)$ originally due to Nicaise \cite{Nic87}, which we will require on several occasions throughout the paper.

\begin{theorem}[Nicaise' inequalities]
\label{thm:nicaise}
Let $\Graph$ be any finite, compact connected (ur-) graph. Then
\begin{equation}
\label{eq:nicaise}
	\lambda_1 (\Graph) \geq \frac{\pi^2}{4|\Graph|^2} \qquad \text{and} \qquad \mu_2 (\Graph) \geq \frac{\pi^2}{|\Graph|^2},
\end{equation}
where in the first case $\Graph$ is equipped with at least one Dirichlet vertex. Equality in either inequality implies that $\Graph$ is a path graph (interval) of length $|\Graph|$, with a Dirichlet vertex at exactly one endpoint and a natural (Neumann) condition at the other in the first case, and natural conditions at both endpoints in the second case.
\end{theorem}

\begin{proof}
The inequalities may be found in \cite[Th\'eor\`eme~3.1]{Nic87}. For the characterisation of equality, see for example \cite{Fri05} (or also \cite[Theorem~3]{KuNa14} in the case of natural conditions).
\end{proof}

We refer to \cite{BerKuc13,Kur20,Mug14} for more background details on the properties of metric graphs and Laplacian-type differential operators on them; we also refer to \cite{BeKeKuMu18} and the references therein for more details on the eigenvalues $\lambda_k(\Graph)$, $\mu_k(\Graph)$ and their dependence on properties of the 
graph $\Graph$.

\subsection{A motivating example}
\label{sec:motivation}

Our goal is to study cutting metric graphs into pieces forming a partition; we shall call these pieces \textit{clusters}. We
shall require later on  that partitions have certain  ``good'' properties, but we need to discuss first what kinds of splittings are possible at all.

This subject is not new. Both Cheeger-like splittings as introduced in~\cite{Nic87} and the investigations in~\cite{BanBerRaz12} restrict to the case 
of cuts performed in the interior of edges. These kinds of partitions are referred to as \textit{proper} in \cite{BanBerRaz12}, where their interplay with 
nodal domains (studied e.g.\ in~\cite{GnuSmiWeb04,Ber08}) is discussed.

Here we wish to consider essentially all possibilities for making cuts, in particular when cuts are made at vertices of degree at least $3$, using a concrete example to motivate what we will introduce subsequently. More precisely, we will consider the \emph{lasso graph} $\Graph$ depicted in Figure~\ref{fig:basic-example}, formed by three edges $e_1,e_2,e_3$, as shown.
\begin{figure}[H]
\begin{tikzpicture}[scale=1.2]
\coordinate (a) at (0,0);
\coordinate (b) at (2,0);
\coordinate (c) at (4,0);
\draw[thick] (a) -- (b);
\draw[thick,bend left=90]  (b) edge (c);
\draw[thick,bend right=90]  (b) edge (c);
\draw[fill] (0,0) circle (1.75pt);
\draw[fill] (2,0) circle (1.75pt);
\draw[fill] (4,0) circle (1.75pt);
\node at (b) [anchor=west] {$v$};
\node at (a) [anchor=north] {$w$};
\node at (c) [anchor=east] {$z$};
\node at (1,0) [anchor=south] {$e_1$};
\node at (3,0.6) [anchor=south west] {$e_2$};
\node at (3,-0.6) [anchor=north west] {$e_3$};
\end{tikzpicture}
\caption{The lasso $\Graph$.}\label{fig:basic-example}
\end{figure}

As done in \cite{BanBerRaz12}, we will refer to any partitions where the cuts are made only at interior points of edges as proper. Figure~\ref{fig:basic-example-proper} illustrates two different ways to split $\Graph$ into two clusters, i.e., to create a proper $2$-partition.
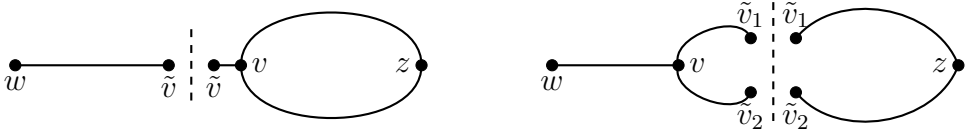
\begin{figure}[H]
\begin{minipage}[l]{7cm}
\begin{tikzpicture}[scale=1.2]
\coordinate (a) at (-0.5,0);
\coordinate (b1) at (1.2,0);
\coordinate (b2) at (1.7,0);
\coordinate (b) at (2,0);
\coordinate (c) at (4,0);
\draw[thick] (a) -- (b1);
\draw[thick] (b2) -- (b);
\draw[thick,bend left=90]  (b) edge (c);
\draw[thick,bend right=90]  (b) edge (c);
\draw[fill] (a) circle (1.75pt);
\draw[fill] (b) circle (1.75pt);
\draw[fill] (b1) circle (1.75pt);
\draw[fill] (b2) circle (1.75pt);
\draw[fill] (c) circle (1.75pt);
\node at (b) [anchor=west] {$v$};
\node at (b1) [anchor=north] {$\tilde{v}$};
\draw[thick,dashed] (1.45,0.4) -- (1.45,-0.4);
\node at (b2) [anchor=north] {$\tilde{v}$};
\node at (a) [anchor=north] {$w$};
\node at (c) [anchor=east] {$z$};
\end{tikzpicture}
\end{minipage}
\begin{minipage}[r]{7cm}
\begin{tikzpicture}[scale=1.2]
\coordinate (a) at (-0.5,0);
\coordinate (b0) at (.9,0);
\coordinate (b1+) at (1.7,.3);
\coordinate (b1-) at (1.7,-.3);
\coordinate (b2+) at (2.2,.3);
\coordinate (b2-) at (2.2,-.3);
\coordinate (c) at (4,0);
\draw[thick] (a) -- (b0);
\draw[thick,bend left=90]  (b0) edge (b1+);
\draw[thick,bend right=90]  (b0) edge (b1-);
\draw[thick,bend left=60]  (b2+) edge (c);
\draw[thick,bend right=60]  (b2-) edge (c);
\draw[fill] (a) circle (1.75pt);
\draw[fill] (b0) circle (1.75pt);
\draw[fill] (b1+) circle (1.75pt);
\draw[fill] (b1-) circle (1.75pt);
\draw[fill] (b2+) circle (1.75pt);
\draw[fill] (b2-) circle (1.75pt);
\draw[fill] (c) circle (1.75pt);
\node at (b0) [anchor=west] {$v$};
\node at (b1+) [anchor=south] {$\tilde{v}_1$};
\node at (b1-) [anchor=north] {$\tilde{v}_2$};
\node at (b2+) [anchor=south] {$\tilde{v}_1$};
\node at (b2-) [anchor=north] {$\tilde{v}_2$};
\node at (a) [anchor=north] {$w$};
\node at (c) [anchor=east] {$z$};
\draw[thick,dashed] (1.95,0.7) -- (1.95,-0.7);
\end{tikzpicture}
\end{minipage}
\caption{Two proper 2-partitions of the lasso $\Graph$: with \emph{separating points} at the dummy vertices $\tilde v$ and $\tilde v_1$, $\tilde v_2$, respectively.}\label{fig:basic-example-proper}
\end{figure}

Note that using our convention to consider ur-graphs, any cut at an interior point of an edge can be
considered as a vertex cut: every such point on the original graph can be seen as a dummy
degree two vertex, before cutting one should choose a representative (from the equivalence class) 
with degree two vertex at the point one wish to cut through.

At any rate, we will refer to the points at which cuts are made as \emph{separating points}; these will be introduced more formally in Section~\ref{sec:part}.

Proper partitions are relatively easy to deal with, but are rather restrictive. 
Any sort of functional we define on partitions should depend continuously on the points at which we are cutting, 
therefore we necessarily have to consider the limits that may arise when the cut reaches a vertex of degree $\geq 3$
in the original graph. 
The authors of \cite{BaLe17} faced a similar problem in a somewhat different context and considered supremisers and minimizers; since we wish to obtain existence results 
we shall consider vertex cuts in full generality.

Let us study what happens to partitions as the cutting points approach the degree $3$ vertex in the lasso graph above. Our intuition
tells us:
\begin{itemize}
\item starting with the partition on the left of Figure~\ref{fig:basic-example-proper} and letting the separating point $\tilde{v}$ tend towards $v$, the
 limit partition should be the one depicted in Figure~\ref{fig:basic-example-faithful}.
\item starting with the partition on the right Figure~\ref{fig:basic-example-proper} and letting $\tilde{v}_1,\tilde{v_2}$ tend towards $v$ the
limit partition should coincide with the one depicted in Figure~\ref{fig:basic-example-rigid}.
\end{itemize}

The edge sets within each cluster are the same in both partitions; the way endpoints of these edges are organised into vertices is different. 
These partitions can be obtained by cutting the lasso graph through vertex $ v$ in two different ways: 
separating the corresponding equivalence class of end points into two and three subclasses respectively.

Partitions of the first type (corresponding to Figure~\ref{fig:basic-example-faithful}) inherit all possible connections from the original graph and reflect its topology as closely as possible; for this reason we will refer to them as \emph{faithful}.

We will call any other partition where we are still only altering the connectivity of our clusters at separating points \emph{rigid}; this is, in particular, the case of the partition in Figure~\ref{fig:basic-example-rigid} (though it is also true of the faithful partition from Figure~\ref{fig:basic-example-faithful}.

\begin{figure}[H]
\begin{tikzpicture}[scale=1.2]
\coordinate (a) at (-.5,0);f
\coordinate (b1) at (1.5,0);
\coordinate (b) at (2,0);
\coordinate (c) at (4,0);
\draw[thick] (a) -- (b1);
\draw[thick,bend left=90]  (b) edge (c);
\draw[thick,bend right=90]  (b) edge (c);
\draw[fill] (a) circle (1.75pt);
\draw[fill] (b) circle (1.75pt);
\draw[fill] (b1) circle (1.75pt);
\draw[fill] (c) circle (1.75pt);
\node at (b) [anchor=west] {$v$};
\node at (b1) [anchor=north] {$v$};
\node at (a) [anchor=north] {$w$};
\node at (c) [anchor=east] {$z$};
\draw[thick,dashed] (1.75,0.4) -- (1.75,-0.4);
\end{tikzpicture}
\caption{A faithful 2-partition of the lasso $\Graph$; the only separating point is $v$.}\label{fig:basic-example-faithful}
\end{figure}
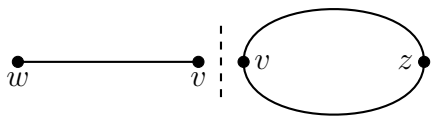

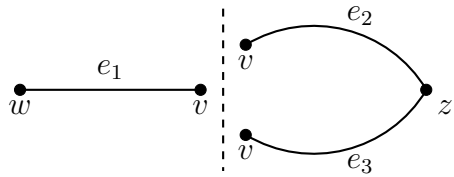
\begin{figure}[H]
\begin{tikzpicture}[scale=1.2]
\coordinate (e) at (6.5,0);
\coordinate (f) at (8.5,0);
\coordinate (g) at (9,0.5);
\coordinate (h) at (9,-0.5);
\coordinate (i) at (11,0);
\draw[thick] (e) -- (f);
\draw[thick,bend right=45] (i) edge (g);
\draw[thick,bend left=45] (i) edge (h);
\draw[fill] (6.5,0) circle (1.75pt);
\draw[fill] (8.5,0) circle (1.75pt);
\draw[fill] (9,0.5) circle (1.75pt);
\draw[fill] (9,-0.5) circle (1.75pt);
\draw[fill] (11,0) circle (1.75pt);
\node at (e) [anchor=north] {$w$};
\node at (f) [anchor=north] {$v$};
\node at (g) [anchor=north] {$v$};
\node at (h) [anchor=north] {$v$};
\node at (i) [anchor=north west] {$z$};
\node at (7.5,0) [anchor=south] {$e_1$};
\node at (10,0.6) [anchor=south west] {$e_2$};
\node at (10,-0.6) [anchor=north west] {$e_3$};
\draw[thick,dashed] (8.75,0.9) -- (8.75,-0.9);
\end{tikzpicture}
\caption{A rigid 2-partition of the lasso $\Graph$; again, the only separating point is $v$.}\label{fig:basic-example-rigid}
\end{figure}

Rigid partitions may appear less natural than faithful ones. As charming the notion of faithful partition  may look at a first glance, it turns out that it is of little use: we will elaborate on this in Section~\ref{sec:abstract-existence}. Indeed, the point of departure of this article involves introducing a suitable, arguably natural metric on the space of graph partitions with respect to which neither the set of proper partitions, nor the set of faithful ones, is closed; however, the set of rigid partitions is. This decisive topological feature is the main reason why we believe it is natural to consider them.

We conclude by considering a further relaxation, which also explains the use of the term \emph{rigid}: namely, we may allow cuts not only at the points separating clusters but also at interior points of clusters (note that these points are not necessarily interior points on some edges: these points could be vertices lying inside clusters), as long as each cluster stays connected: we shall refer to this kind of partition as \textit{loose}.

\begin{figure}[H]
\begin{tikzpicture}[scale=1.2]
\coordinate (a) at (0,0);
\coordinate (b) at (2,0);
\coordinate (c) at (2.5,0);
\coordinate (d1) at (4.5,-.5);
\coordinate (d2) at (4.5,.5);
\draw[thick] (a) -- (b);
\draw[thick,bend right=45] (c) edge (d1);
\draw[thick,bend left=45] (c) edge (d2);
\draw[fill] (a) circle (1.75pt);
\draw[fill] (b) circle (1.75pt);
\draw[fill] (c) circle (1.75pt);
\draw[fill] (d1) circle (1.75pt);
\draw[fill] (d2) circle (1.75pt);
\node at (a) [anchor=north] {$w$};
\node at (b) [anchor=north] {$v$};
\node at (c) [anchor=west] {$v$};
\node at (d1) [anchor=north west] {$z$};
\node at (d2) [anchor=north west] {$z$};
\node at (1,0) [anchor=south] {$e_1$};
\node at (3.7,0.6) [anchor=south west] {$e_2$};
\node at (3.7,-0.6) [anchor=north west] {$e_3$};
\draw[thick,dashed] (2.25,0.4) -- (2.25,-0.4);
\draw[thick,dashed] (4.2,0) -- (4.8,0);
\end{tikzpicture}
\caption{A loose 2-partition of the lasso $\Graph$; in this case, the only separating point is $v$ but we are additionally cutting through $z$.}\label{fig:basic-example-loose}
\end{figure}
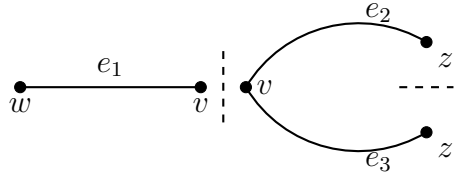
Loose partitions also define a closed set of partitions with respect to the natural metric we are going to introduce in Section~\ref{sec:abstract-existence}.

Finally, we regard partitions consisting of non-connected clusters as invalid.
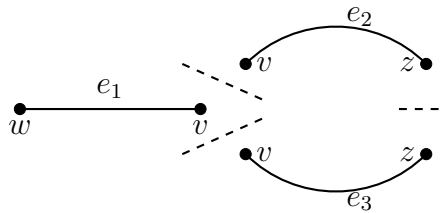
\begin{figure}[H]
\begin{tikzpicture}[scale=1.2]
\coordinate (e) at (6.5,0);
\coordinate (f) at (8.5,0);
\coordinate (g1) at (9,0.5);
\coordinate (g2) at (9,-0.5);
\coordinate (h1) at (11,.5);
\coordinate (h2) at (11,-.5);
\draw[thick] (e) -- (f);
\draw[thick,bend right=45] (g2) edge (h2);
\draw[thick,bend left=45] (g1) edge (h1);
\draw[fill] (e) circle (1.75pt);
\draw[fill] (f) circle (1.75pt);
\draw[fill] (g1) circle (1.75pt);
\draw[fill] (g2) circle (1.75pt);
\draw[fill] (h1) circle (1.75pt);
\draw[fill] (h2) circle (1.75pt);
\node at (e) [anchor=north] {$w$};
\node at (f) [anchor=north] {$v$};
\node at (g1) [anchor=west] {$v$};
\node at (g2) [anchor=west] {$v$};
\node at (h1) [anchor=east] {$z$};
\node at (h2) [anchor=east] {$z$};
\node at (7.5,0) [anchor=south] {$e_1$};
\node at (10,0.8) [anchor=south west] {$e_2$};
\node at (10,-0.8) [anchor=north west] {$e_3$};
\draw[thick,dashed] (8.3,0.5) -- (9.2,0.1);
\draw[thick,dashed] (8.3,-0.5) -- (9.2,-0.1);
\draw[thick,dashed] (10.7,0) -- (11.3,0);
\end{tikzpicture}
\caption{An invalid 2-partition of the lasso $\Graph$; cutting through both $v$ and $z$ has led to a disconnected cluster. This is a faithful $3$-partition, though.}\label{fig:basic-example-invalid}
\end{figure}

\subsection{Graph partitions}
\label{sec:part}

After informally sketching the ideas that motivate our classification, let us now introduce our notion of partition more precisely. We first need to recall an operation transforming a graph into another one by joining or cutting through its vertices (cf., e.g., \cite[Definitions~3.1 and~3.2]{BeKeKuMu18}). We phrase this slightly differently, using the formalism introduced in Section~\ref{subsec:def} and especially Definition~\ref{def:urg}.

\begin{definition}
\label{def:graph-cut-new1}
Let $\Graph,\NewGraph$ be ur-graphs. Then $\NewGraph$ is called a \emph{cut} of $\Graph$ if there exist a representative $\Graph'$ of $\Graph$ and a representative $\NewGraph'$ of $\NewGraph$ with vertex sets $\VertexSet(\Graph') = \{v_1(\Graph),\ldots,v_{\cardV}(\Graph)\}$ and $\VertexSet(\NewGraph') = \{\tilde{v}_1(\Graph'),\ldots,\tilde{v}_{\widetilde{{\cardV}}}(\Graph')\}$ and edge sets $\EdgeSet(\Graph')$ and  $\EdgeSet(\NewGraph')$, respectively, such that
\begin{enumerate}
\item[(a)] $\EdgeSet(\NewGraph') = \EdgeSet(\Graph')$,
\item[(b)] $\widetilde{{\cardV}}\geq {\cardV}$, and
\item[(c)] for all $\tilde{n} = 1,\ldots,\widetilde{{\cardV}}$, in the notation and identification of Section~\ref{subsec:def}, we have
\begin{displaymath}
	\widetilde{V}_{\tilde{n}}(\NewGraph') \subset V_n(\Graph')
\end{displaymath}
for some $n=1,\ldots,{\cardV}$.
\end{enumerate}
\end{definition}

In words, the graph $\NewGraph$ is formed from $\Graph$ by first picking a collection of vertices, in general including dummy vertices in the interior of edges, of $\Graph$ (this is the choice $\Graph'$), and then cutting through each such vertex $v_n$ of $\Graph'$ by removing adjacency relations to create new vertices $\tilde{v}_{\tilde{n}}$ out of $v_n$. In practice, however, we will tend to suppress the tildes wherever feasible. As stated earlier, we will also tend not to distinguish between the ur-graph $\Graph$ and its representative $\Graph'$; in particular, in a slight abuse of notation, we will regard the vertices $v_1,\ldots,v_N$ of $\Graph'$ as being vertices of $\Graph$. 
We also stress that we do \textit{not} require cutting through vertices to produce a \textit{connected} metric graph $\widetilde{\Graph}$.

Let us make this clearer by considering what happens if we only cut $\Graph$ at a single vertex.

\begin{definition}
\label{def:graph-cut-new2}
Given two ur-graphs $\Graph,\NewGraph$, keep the setup and notation of Definition~\ref{def:graph-cut-new1}.
\begin{enumerate}
\item Suppose there exist a representative $\Graph'$ of $\Graph$, a representative $\NewGraph'$ of $\NewGraph$, and 
\begin{itemize}
\item[(a)] vertices $v_{n_0}\in \VertexSet(\Graph')$ and $\tilde{v}_{\tilde{n}_1} ,\ldots,\tilde{v}_{\tilde{n}_k}\in \VertexSet(\NewGraph')$ such that $V_{n_0} = \widetilde{V}_{\tilde{n}_1}\cup\ldots\cup \widetilde{V}_{\tilde{n}_k}$, and
\item[(b)] there is equality $\widetilde{V}_{\tilde{n}}(\NewGraph') = V_n(\Graph')$ in condition (c) of Definition~\ref{def:graph-cut-new1} for all $\tilde{n}$ \emph{except} $\tilde{n}_1,\ldots,\tilde{n}_k$.
\end{itemize}
Then we say that $\NewGraph$ has been obtained from $\Graph$ by \emph{cutting through the vertex} $v_{n_0}$ (to obtain the vertices $\tilde{v}_{\tilde{n}_1},\ldots,\tilde{v}_{\tilde{n}_k}$). We call the vertices $\tilde{v}_{\tilde{n}_1}, \ldots, \tilde{v}_{\tilde{n}_k}$ the \emph{image} of the vertex $v_{n_0}$ under the cut.
\item We also say that the vertex $v_{n_0}$ in $\Graph$ \emph{corresponds to} the vertices $\tilde{v}_{\tilde{n}_1},\ldots,\tilde{v}_{\tilde{n}_k}$ in $\NewGraph$, and that $\Graph$ is obtained from $\NewGraph$ by \emph{gluing the vertices} $\tilde{v}_{\tilde{n}_1},\ldots,\tilde{v}_{\tilde{n}_k}$ to form $v_{n_0}$.
\item We say that the vertex $v_{n_0} (\Graph)$ \emph{has been cut nontrivially} if $k \geq 2$ in (1).
\end{enumerate}
\end{definition}

Definition~\ref{def:graph-cut-new2} may be generalised (or applied inductively) to the situation described in Definition~\ref{def:graph-cut-new1}; in particular, we will use the language of cutting through (possibly multiple) vertices and nontrivial cuts in the context of Definition~\ref{def:graph-cut-new1}.

We are finally in a position to define the central notions of this paper.
\begin{definition}[Partitions of a graph]
\label{def:partition}
Let $k\geq 1$ and let $\Graph$ be an ur-graph.
\begin{enumerate}
\item We call any set of $k$ distinct connected metric graphs
\[
	\parti := \{ \Graph_{1},\ldots,\Graph_{k}\}
\]
a \emph{$k$-partition} of $\Graph$ if there is a  cut $\NewGraph = \bigsqcup_{j=1}^{k_0}\Graph_{i_j}$ of $\Graph$, $k_0 \geq k$, whose connected components include $\Graph_1,\ldots,\Graph_{k}$, i.e., $\Graph_{i_j}=\Graph_j$ for all $j=1,\ldots,k$, where $i_{j_1} \neq i_{j_2}$ for $j_1\neq j_2$. In this case, we refer to the components $\Graph_{1},\ldots,\Graph_{k}$ as the \emph{clusters} of the partition $\parti$ (\textit{arising} from the cut $\NewGraph$).
\item If in (1) there exists a cut $\NewGraph$ of $\Graph$ such that $\NewGraph = \bigsqcup_{i=1}^k \Graph_i$, then we say the partition $\parti = \{ \Graph_1,\ldots, \Graph_k\}$ is \emph{exhaustive}.\footnote{Note that in the case of domains, some sources, such as \cite{HelHofTer09},  refer to exhaustive partitions as \emph{strong}.}
\end{enumerate}
\end{definition}

With this definition, the $k$ clusters are themselves compact metric graphs, which may be identified with subsets of $\mathcal G$. It will however often be useful to consider explicitly the subsets of $\Graph$ which correspond to the clusters; to this end we make the following definition.

\begin{definition}[Cluster supports]
\label{def:cluster-support}
Let $\Graph$ be an ur-graph and let $\parti = \{\Graph_1,\ldots,\Graph_k\}$ be a $k$-partition of $\Graph$, arising from the cut $\NewGraph = \bigsqcup_{i=1}^{k_0}\Graph_i$, $k_0\geq k$, of $\Graph$. 
We identify $\Graph$ and $\NewGraph$ with any respective representatives satisfying the conditions of Definition~\ref{def:graph-cut-new1}, that is, in such a way that $\EdgeSet (\Graph) = \EdgeSet (\NewGraph)$.
\begin{enumerate}
\item For each $i=1,\ldots,k$, we denote by $\Omega_i$ the unique closed subset of $\Graph$ such that
\begin{displaymath}
\{ e \in \EdgeSet (\NewGraph): e \subset \Graph_i\}= \left\{ e \in \EdgeSet (\Graph): e \subset \Omega_i \right\}
\end{displaymath}
and call the set $\Omega_i$ the \emph{cluster support} (associated with the cluster $\Graph_i$), or just \emph{support} for short.
\item We call the set
\begin{equation}
\label{eq:partition-support}
	\Omega := \bigcup\limits_{i=1}^k \Omega_i
\end{equation}
the \emph{support} of the partition $\parti$.
\end{enumerate}
\end{definition}

With this definition, the cluster supports $\Omega_1,\ldots,\Omega_k$ are really a partition of $\Graph$ in the ``classical'' sense; let us elaborate on this point. Indeed, we may think of the $\Omega_i$ as the subsets of $\Graph$ out of which we form new graphs, the clusters $\Graph_i$, by cutting through vertices as desired. Thus, by construction, the $\Omega_i$ are closed, connected subsets of $\Graph$, and their interiors $\interior \Omega_i$, $i=1,\ldots,k$, are pairwise disjoint. Moreover, $\parti$ is exhaustive if and only if the set $\Omega \subset \Graph$ actually equals $\Graph$. (For various practical reasons we are taking the cluster supports to be closed, not open, subsets of $\Graph$.)

Finally, with the right choice of representative of $\Graph$, we may suppose that, for each $i=1,\ldots,k$, we have $\Omega_i = e_{i_1} \cup \ldots \cup e_{i_{{\cardE}_i}}$ for some edges $e_{i_1},\ldots,e_{i_{{\cardE}_i}} \in \EdgeSet (\Graph)$. This means that $\partial \Omega_i \subset \VertexSet (\Graph)$ for all $i=1,\ldots,k$; and for each $e \in \EdgeSet (\Graph)$ there exists at most one $i=1,\ldots,k$ such that $e \subset \Omega_i$, exactly one if $\parti$ is exhaustive. (We emphasise that $\partial\Omega_i$ is always the topological boundary of the closed set $\Omega_i$ in the compact metric space $\Graph$.)

From now on, whenever $\parti = \{\Graph_1,\ldots,\Graph_k\}$ is a $k$-partition of $\Graph$, we will \emph{always} use the notation $\Omega_1,\ldots,\Omega_k$ to denote the corresponding cluster supports, and $\Omega$ for the support of $\parti$ (if distinct from $\Graph$), without further comment.

Observe that if $\parti$ is exhaustive and $x\in\partial \Omega_i$ for some $i$, then there must be at least one $j\ne i$ such that $x\in \partial \Omega_i\cap \partial \Omega_j$.

\begin{definition}
\label{def:separation-set}
Let $\Graph$ be an ur-graph and let $\parti = \{\Graph_1,\ldots,\Graph_k\}$ be a $k$-partition of $\Graph$ for some $k\geq 1$.
\begin{enumerate}
\item We call a vertex $v \in \VertexSet (\Graph) \cap \Omega$ a \emph{cut point} (of $\parti$) if there is no vertex
\begin{displaymath}
	\tilde v \in \bigcup\limits_{i=1}^k \VertexSet (\Graph_i)
\end{displaymath}
such that $v = \tilde v$, that is, if $v$ is nontrivially cut when constructing the partition. We refer to the set
\begin{displaymath}
	\cutset= \cutset (\parti) \subset \Omega
\end{displaymath}
of all cut points of $\parti$ as the \emph{cut set} of the partition $\parti$.
\item We will denote by
\begin{equation}
\label{eq:dirichlet-vertex-definition}
	\VertexSet_D (\Graph_i) \subset \Graph_i
\end{equation}
the set of all vertices in $\Graph_i$ which are obtained by nontrivially cutting through vertices of $\Graph$; in a slight abuse of terminology, we will also refer to its elements as \textit{cut points}.
\item We call the \textit{separation set} (of $\parti$) the set
\begin{equation}
\label{eq:part-boundary-set}
	\partial \parti := \bigcup\limits_{i=1}^k \partial \Omega_i \subset \Omega.
\end{equation}
We refer to its elements as \emph{separating points}.
\end{enumerate}
\end{definition}

It follows from our definition of partitions that every separating point is a cut point, although the converse need not be true; we also reiterate that we are assuming without loss of generality (by taking the right representative of the ur-graph) that each cut point, and in particular each separating point, is a vertex. Both the cut and the separation sets are clearly always finite.

\begin{definition}
\label{def:neighbours}
Let $\parti = \{ \Graph_1, \ldots, \Graph_k \}$ be a $k$-partition of  $\Graph$, and denote by $\Omega_1,\ldots,\Omega_k$ the respective cluster supports. We say that $\Omega_i,\Omega_j$, $i,j=1,\ldots,k$, $i \neq j$, are  \emph{neighbours} if $\partial\Omega_i \cap \partial\Omega_j \neq \emptyset$. In this case, we will also loosely refer to the corresponding clusters $\Graph_i$ and $\Graph_j$ as neighbours. Similarly, given a cut point $v \in \partial\parti$, we will refer to each $\Omega_i$ such that $v \in \partial\Omega_i$ as a \emph{neighbouring support} of $v$.
\end{definition}

It turns out that there are several different, reasonably natural possibilities for defining classes of partitions of a metric graphs, as we intimated in Section~\ref{sec:motivation}. We stress that exhaustivity of a partition, in the sense of Definition~\ref{def:partition}(2), is not related to the following classification: exhaustivity does not imply, nor is it implied by, any of the following properties.

\begin{definition}[Classification of partitions]
\label{def:classification}
Let $\Graph$ be an ur-graph.
\begin{enumerate}
\item Any partition $\parti$ of $\Graph$ satisfying Definition~\ref{def:partition} will be called \textit{loose}.
\item A loose partition $\parti$ of $\Graph$ will be called \emph{rigid} if its cut  and separation sets agree, that is,  we only cut vertices on the boundary of $\Omega_i$ to create the graph $\Graph_i$.
\item A partition $\parti$ of $\Graph$ will be called \emph{faithful} if it is rigid and additionally whenever a separating point $v$ lies in the cluster support $\Omega_i$, then in the corresponding cluster $\Graph_i$ the image of $v$ under the cut $\NewGraph$ is incident with all edges $e$ that were incident with $v$ in $\Graph$, such that $e$ also lies in $\Graph_i$.

\item A partition $\parti$ of $\Graph$ will be called \emph{internally connected} if it is rigid and $\interior \Omega_i = \Omega_i \setminus \partial\Omega_i$ is connected, equivalently, if $\Graph_i \setminus \VertexSet_D (\Graph_i)$ is connected, for all $i=1,\ldots,k$.
\item A partition $\parti$ of $\Graph$ will be called \emph{proper} if it is rigid and all separating points are vertices of degree two in $\Graph$.
\end{enumerate}
\end{definition}

By definition, the cut and separation sets are allowed to be different only in a loose partition. It is clear from the definitions that every proper partition is faithful and internally connected, every faithful and every internally connected partition is rigid, and every rigid partition is loose, but the converse statements do not hold. For example, if $\Graph$ is a graph divided into cluster supports $\Omega_1,\ldots,\Omega_k$, then \emph{any} choice of spanning metric trees $\Graph_1,\ldots,\Graph_k$ of these cluster supports determines a further loose partition.\footnote{A \emph{spanning metric tree} of a metric graph $\Graph$ is, by definition, a tree $\NewGraph$ which is a cut of $\Graph$.}

\begin{example}\label{exa:inthecase}
In the case of the lasso graph discussed in Section~\ref{sec:motivation}, we may choose to split $\Graph$ into the cluster supports $\Omega_1 = e_1$ (interval) and $\Omega_2 = e_2 \cup e_3$ (loop), so that $\partial \parti = \partial \Omega_1 = \partial \Omega_2 = \{v\}$. Suppose we wish $\parti$ to be exhaustive: in order to determine it, we need to specify the clusters $\Graph_1$ and $\Graph_2$: while a cluster $\Graph_1$ is uniquely determined by $\Omega_1$, namely, it is the edge $e_1$, for the cluster $\Graph_2$ there are two possible choices, depicted in Figures~\ref{fig:basic-example-faithful} and~\ref{fig:basic-example-rigid}, which lead to a faithful and a non-faithful but rigid 2-partition, respectively; both are internally connected. The third choice, of Figure~\ref{fig:basic-example-loose}, where to produce $\Graph_2$ we also cut through $z$, gives rise to a loose 2-partition.

If we allow $\parti$ to be non-exhaustive and, say, take $\parti = \{\Graph_1\}$, then $\parti$ is faithful and internally connected (but still not proper). In this case, what happens to the set $\Omega_2$ under any cut giving rise to $\parti$ is irrelevant for the classification of $\parti$.
\end{example}

\begin{example}\label{ex:tree}
In the case of metric trees, our classification of partitions from Definition~\ref{def:classification} boils down to three cases.

Cutting through a vertex of degree two creates by definition a proper $2$-partition.

Cutting through a single vertex $v$ of degree $\deg v >2$ may produce $k$ connected components for any $2\le k\le \deg v $; the associated $k$-partition $\parti$ that arises in this way is necessarily exhaustive. More interestingly, if $k=\deg v$, then $\parti$ is both internally connected and faithful.
If on the other hand $k<\deg v$, then $\parti$ is not internally connected (for there is some cluster such that at least two different edges lie ``on different sides'' of the separating point $v$); it is faithful though, because by definition each cluster must be a connected metric graph in its own right, hence no further cut can be made through $v$ in any of the clusters.

In particular, all loose partitions of metric trees are necessarily faithful, but there are rigid partitions that are not internally connected.
\end{example}

We will be primarily interested in the classes of loose and rigid partitions, 
and in exhaustive partitions. For a fixed ur-graph $\Graph$ and $k\geq 1$, we denote the class of all exhaustive loose $k$-partitions of $\Graph$ by $\mathfrak{P}_k(\Graph)$, or simply by $\mathfrak{P}_k$ if the graph $\Graph$ is clear from the context, the set of all exhaustive rigid $k$-partitions of $\Graph$ by $\mathfrak{R}_k (\Graph)$ or $\mathfrak{R}_k$, and
\begin{equation}
\label{eq:loose-and-rigid}
	\mathfrak{P} = \mathfrak{P} (\Graph) := \bigcup\limits_{k=1}^\infty \mathfrak{P}_k, \qquad \mathfrak{R} = \mathfrak{R} (\Graph)
	:= \bigcup\limits_{k=1}^\infty \mathfrak{R}_k,
\end{equation}
the set of all loose exhaustive, and all rigid exhaustive, partitions of $\Graph$, respectively.

Finally, if $\Omega_1,\ldots,\Omega_k \subset \Graph$ are closed subsets of $\Graph$ with pairwise disjoint interiors, then for each $i=1,\ldots,k$, we set
\begin{equation}
\label{eq:rigid-cluster-set}
	\rho_{\Omega_i}
\end{equation}
to be the finite set of all possible clusters $\Graph_i$ that have $\Omega_i$ as a cluster support and such that the partition $\parti = \{\Graph_1, \ldots, \Graph_k\}$ is \emph{rigid}. Note that $\rho_{\Omega_i} \neq \emptyset$; indeed, $\rho_{\Omega_i}$ always contains exactly one cluster corresponding to a faithful partition of $\parti$. We may also loosely refer to the clusters of such a partition as rigid clusters; we will do likewise for loose, faithful, internally connected and proper clusters. Observe that as long as \emph{proper} partitions are considered, there is no such ambiguity: each cluster support uniquely determines a cluster; in particular, the set $\rho_{\Omega_i}$ always contains a single element.

\begin{example}\label{exa:inthecase-2}
Returning again to the lasso graph discussed in Section~\ref{sec:motivation}, given the cluster supports $\Omega_1 = e_1$ (interval) and $\Omega_2 = e_2 \cup e_3$ (loop), we have that $\rho_{\Omega_1}$ consists of a single element, the graph given by the edge $e_1$, while the set $\rho_{\Omega_2}$ constains two graphs: an interval and a loop, see Figures~\ref{fig:basic-example-faithful} and~\ref{fig:basic-example-rigid}.
\end{example}

\section{Topological issues of graph partitions}
\label{sec:abstract-existence}

Here we wish to construct a suitable topology on spaces of partitions, which will allow us to give existence results for minimisers of suitable functionals. Throughout this section, we will \emph{only} work with exhaustive partitions, as these are more suited to topologisation and they will be of primary interest in the sequel.

\subsection{Primitive partitions}
\label{sec:colour-type}

\begin{definition}\label{defi:simcut}
Let $\Graph$ be an ur-graph and let $\parti_1 = \{\Graph_{1}^{(1)},\ldots, \Graph_{k}^{(1)} \}$ and $\parti_2 = \{\Graph_{1}^{(2)}, \ldots, \Graph_{k}^{(2)} \}$ be two exhaustive, loose $k$-partitions of $\Graph$. Then we say that $\parti_1$ and $\parti_2$ are \emph{similar}, or \emph{share a common cut pattern} (of $\Graph$), if, up to the correct choice of representative of the ur-graph $\Graph$ and numbering of the clusters, for each $i=1,\ldots,k$ the clusters $\Graph_{i}^{(1)}$ and $\Graph_{i}^{(2)}$ have the same underlying discrete graph (see Definition~\ref{def:discrete-graph}), and there is a bijection between the cut sets $\cutset (\parti_1),\cutset (\parti_2)$ (see Definition~\ref{def:separation-set}).
\end{definition}

\begin{proposition}
\label{prop:colour-equivalence}
Suppose $\Graph$ is a fixed ur-graph and let $k\geq 1$.
\begin{enumerate}[(1)]
\item Similarity between $k$-partitions of $\Graph$, as in Definition~\ref{defi:simcut}, is an equivalence relation. It divides $\mathfrak{P}_k (\Graph)$ into a finite number of cells. 
\item If two exhaustive, loose $k$-partitions $\parti,\widetilde\parti$ of $\Graph$ are similar, then after renumbering the clusters if necessary, for each $i=1,\ldots,k$, $\Graph_i^{(2)}$ can be obtained from $\Graph_i^{(1)}$ by lengthening or shortening edges of $\Graph_i^{(1)}$.
\item If two exhaustive, loose $k$-partitions $\parti,\widetilde\parti$ of $\Graph$ are similar and $\parti$ is rigid (respectively, faithful, internally connected or proper), then so too is $\widetilde\parti$.
\end{enumerate}
\end{proposition}

\begin{proof}
(1) is immediate, since all properties of similarity may be characterised in terms of bijections.

(2) It suffices to prove that the same is true of any two metric graphs $\Graph_1$ and $\Graph_2$ which have the same underlying discrete graph $\DG$. But this, in turn, is an immediate consequence of the definition (Definition~\ref{def:discrete-graph}): the edges $e_1^{(1)},\ldots,e_M^{(1)}$ of $\Graph_1$ and $e_1^{(2)},\ldots,e_M^{(2)}$ of $\Graph_2$ are in a canonical bijection to each other, both being in bijective correspondence with the edges $\De_1,\ldots,\De_M$ of $\DG$; moreover, this bijective correspondence preserves all adjacency and incidence relations. Hence, if for each $i=1,\ldots,k$ we replace the edge $e_i^{(1)}$ with an edge of length $|e_i^{(2)}|$, then the resulting graph is isometrically isomorphic to $\Graph_2$.

(3) follows since cut patterns completely describe the connectivity of the resulting clusters in the neighbourhood of any cut point.
\end{proof}

\begin{definition}
\label{def:primitive-partition}
We call the equivalence classes with respect to the above equivalence relation \emph{primitive $k$-partitions}.
We will denote the primitive partition associated with the cut pattern $\cutset$ by
\begin{equation*}
\label{eq:colour-type-2}
	\colourset = \colourset_{\cutset} \subset \mathfrak{P}_k.
\end{equation*}
\end{definition}
Since by definition $\mathfrak P_k$ consists of exhaustive partitions, any primitive $k$-partition is necessarily exhaustive.

\subsection{Partition convergence}
\label{subsec:part-convergence}

The equivalence discussed in the previous subsection gives rise to a notion of convergence of partitions within each primitive partition, which we now wish to introduce. To begin with, we need the notion of convergence of a sequence of graphs having the same underlying discrete topology, similar to what was considered in \cite[\S~1]{BaLe17}.

Given a finite discrete graph $\DG = (\DV,\DE)$, let $\Gamma_\DG$ be the set of all ur-graphs whose underlying discrete graph is $\DG$,  in the sense of Definition~\ref{def:discrete-graph}. We assume here and throughout that the indexing of the edges and vertices is consistent,  in the sense that if $\DE = \{ \De_1, \ldots, \De_M \}$ and $\Graph^{(1)},\Graph^{(2)} \in \Gamma_\DG$, then up to the correct choice of representatives of the ur-graphs we have $\EdgeSet (\Graph^{(n)}) = \{ e_1^{(n)}, \ldots, e_M^{(n)} \}$ and the bijection $\Psi_i : \DE \to \EdgeSet$ maps $\De_i$ to $e_i^{(n)}$ for all $i=1,\ldots,M$, with corresponding statements for the vertices, $n=1,2$. Observe that each $\Graph\in\Gamma_\DG$ is uniquely determined by its vector $(|e|)_{e\in\mathcal E}$ of edge lengths, hence we can define
\begin{equation}
\label{eq:dist-gamma}
	d_{\Gamma_\DG}(\Graph,\NewGraph):=d_{\R^M}\left((|e|)_{e\in \EdgeSet},(|\tilde{e}|)_{\tilde{e}\in \tilde\EdgeSet} \right),
\end{equation}
$\Graph,\NewGraph \in \Gamma_\DG$, where $d_{\R^M}$ is the Euclidean distance on $\R^M$. 

\begin{proposition}
\label{prop:dist-gamma}
Given a discrete graph $\DG$, $(\Gamma_\DG,d_{\Gamma_\DG})$ is a separable metric space with respect to the Euclidean distance in $\R^M$.
\end{proposition}

This metric structure induces the same topology as the one discussed in~\cite[\S~2]{BanBerRaz12} and the one used in \cite{BaLe17}.

We can now consider Cauchy sequences $\Graph^{(n)}$ in $\Gamma_\DG$; however, they need not converge in $\Gamma_\DG$, since one or more  edge lengths may tend to 0.
We  can however consider the canonical completion $\overline{\Gamma_\DG}$ of $\Gamma_\DG$: it consists of equivalence classes of Cauchy sequences of metric graphs in $\Gamma_\DG$ with respect to the equivalence relation of having distance $d_{\Gamma_\DG}(\Graph^{(n)},\NewGraph^{(n)})$ vanishing as $n\to \infty$. 
One can identify $\overline{\Gamma_\DG}$ with the simplex of all vectors in the positive orthant of $\R^M$ whose size agrees with the total length of $\Graph$, i.e.,
\begin{displaymath}
	\overline{\Gamma_\DG}\simeq \left\{(x_1,\ldots,x_M):x_i\ge 0 \hbox{ and  }\sum\limits_{i=1}^M x_i = \sum\limits_{i=1}^M |e_i|\right\}\ .
\end{displaymath}
The limit of a converging sequence $(\Graph^{(n)})_{n\in\N}\subset \overline{\Gamma_\DG}$ may hence be identified with an ur-graph $\Graph^{(\infty)}$ whose edge lengths are the (possibly vanishing) limits of the edge lengths of the approximating graphs $\Graph^{(n)}$; accordingly $\Graph^{(\infty)}$ may well have a different underlying discrete graph with a lower number of vertices and edges; and it may contain loops and parallel edges even if the approximating graphs do not. We may group the vertices of $\Graph^{(n)}$ according to the rule
\begin{quote}
$v^{(n)},w^{(n)} \in \VertexSet (\Graph^{(n)})$ are equivalent if and only if $\dist_{\Graph^{(n)}} (v^{(n)},w^{(n)}) \stackrel{n\to\infty}\longrightarrow 0$.
\end{quote} 
Thus with each vertex $v^{(\infty)}$ of $\Graph^{(\infty)}$ is associated a unique equivalence class of vertices of $\Graph^{(n)}$ of this form, which we will denote by $[v^{(\infty)}]$.

Let us explicitly formulate the following useful observations.

\begin{lemma}
\label{lem:part-convergence}
Let $(\Graph^{(n)})_{n\in\mathbb N}$ converge to $\Graph^{(\infty)}$ in $\overline{\Gamma_\DG}$. Then
\begin{enumerate}
\item the total length $|\Graph^{(n)}|$ tends to $|\Graph^{(\infty)}|$;
\item $\Graph^{(\infty)}$ is connected provided the $\Graph^{(n)}$ are.
\end{enumerate}
\end{lemma}

We also note for future reference that the Laplacian eigenvalues we are considering, introduced in Section~\ref{subsec:def}, behave well with respect to this notion of convergence. Here the correspondence between vertices is necessary to identify the correct limiting vertex conditions.

\begin{lemma}
\label{lem:eig-convergence}
Let $(\Graph^{(n)})_{n\in\mathbb N}$ converge to $\Graph^{(\infty)}\neq\emptyset$ in $\overline{\Gamma_\DG}$. Then
\begin{enumerate}
\item $\mu_2 (\Graph^{(n)}) \to \mu_2 (\Graph^{(\infty)})$;
\item if a vertex set $\DV_D$ in the underlying discrete graph $\DG$ is chosen and Dirichlet conditions are applied at all vertices in $\Graph^{(n)}$ corresponding to $\DV_D$, and if Dirichlet conditions are applied at exactly those vertices $v$ of $\Graph^{(\infty)}$ such that at least one vertex in $[v]$ corresponds to a vertex in $\DV_D$, then $\lambda_1 (\Graph^{(n)}) \to \lambda_1 (\Graph^{(\infty)})$.
\end{enumerate}
If $\Graph^{(n)} \to \emptyset$, then $\mu_2 (\Graph^{(n)}) \to \infty$ as $n \to \infty$. If in addition $\VertexSet_D (\Graph^{(n)}) \neq \emptyset$, then also $\lambda_1 (\Graph^{(n)}) \to \infty$.
\end{lemma}

\begin{proof}
(1) follows from the method described in \cite[Appendix~A]{BaLe17} (which can also be easily adapted to (2)); alternatively, see \cite{BeLaSu19} for a more detailed treatment of both. The case where no edge lengths converge to zero is already covered in \cite[\S~3.1]{BerKuc13}. In the degenerate case where all edge lengths converge to zero, Nicaise' inequalities (Theorem~\ref{thm:nicaise}) imply that $\mu_2 (\Graph^{(n)}) \geq \pi^2/|\Graph^{(n)}|^2 \to \infty$ and $\lambda_1 (\Graph^{(n)}) \geq \pi^2/4|\Graph^{(n)}|^2 \to \infty$ (in the latter case as long as at least one Dirichlet vertex is present).
\end{proof}

With this background, we can now return to partitions and in particular define the notion of convergence of a sequence of partitions. For the rest of the section, we assume that $\Graph$ is a fixed 
ur-graph satisfying (up to the correct choice of representative) Assumption~\ref{assumption}, and fix a primitive $k$-partition $\colourset$ of $\Graph$; we suppose that the clusters of each partition $\parti = \{\Graph_1, \ldots, \Graph_k\}$ have the respective underlying discrete graphs $\DG_1,\ldots,\DG_k$ (for a fixed order). Then, as above, setting $E_i:=|\DE(\DG_i)|$ to be the number of edges of $\DG_i$, each $\Graph_1$ may be uniquely identified with a vector in $\R^{E_i}_+$; this means that each $\parti = (\Graph_1,\ldots,\Graph_k) $ may be identified with a vector in $\R^{E_1}_+ \times \ldots \times \R^{E_k}_+ \simeq \R^E_+$ whose (strictly) positive entries sum to the total length $|\Graph|$ of the graph $\Graph$, that is, we have the identification
\begin{equation}
\label{eq:colourset-identification}
	\colourset \simeq \Theta_\colourset := \left\{ x = (x_1,\ldots,x_E) \in \R^E : x_j > 0
	\text{ for all $j$ and } |x|_1 = |\Graph| \right\},
\end{equation}
where $E=\sum\limits_{i=1}^k E_i=\sum\limits_{i=1}^k |\DE(\mG_i)|$ and $|x|_1$ is the $1$-norm of the vector $x$. Now if two partitions $\parti = \{\Graph_1,\ldots,\Graph_k\}$ and $\widetilde{\parti} = \{\NewGraph_1, \ldots, \NewGraph_k \}$ are similar, $\parti,\widetilde{\parti} \in \colourset = \colourset_{\cutset}$ (and in particular consist of clusters that have the same underlying discrete graphs, say $\DG_1,\ldots,\DG_k$), then we can introduce
\begin{equation}
\label{eq:dist-pk}
	d(\parti,\widetilde{\parti}):=\sum\limits_{i=1}^k d_{\Gamma_{\DG_i}}(\Graph_i,\NewGraph_i),
\end{equation}
where $d_{\Gamma_{\DG_i}}$ is the distance introduced in equation~\ref{eq:dist-gamma}. This distance induces an equivalent topology to the one induced by the Euclidean distance between the points in the set $\Theta_\colourset$ corresponding to the respective partitions $\parti$ and $\widetilde{\parti}$. The following result is immediate.

\begin{lemma}\label{lem:primmetr}
Let $\cutset$ be a cut pattern of $\Graph$. Then $\colourset=\colourset_{\cutset}$ is a metric space with respect to the distance introduced in~\eqref{eq:dist-pk}.
\end{lemma}

In order to check the plausibility of this metrisation of the partition space, let us explicitly record the following observation.

\begin{proposition}
\label{prop:exhaustive-closed}
Suppose $\colourset$ is any primitive partition and $(\parti_n)_{n\in\N}\subset \colourset$ is a sequence of $k$-partitions which is Cauchy with respect to the metric \eqref{eq:dist-pk}. Then the limit partition $\parti_\infty \in \overline{\colourset}$ is also exhaustive.
\end{proposition}

However, this metric space is non-complete, since given a Cauchy sequence $(\parti_n)_{n\in\N}\subset \colourset$ it cannot be excluded that one or more clusters vanish in the limit, i.e., $|\Graph_i^{(n)}|\to 0$, leading to an $m$-partition of $\Graph$ with $m<k$ as a limit object; 
these correspond to the limit points in the Euclidean set $\Theta_\colourset$ from \eqref{eq:colourset-identification} with one or more entries equal to zero.

The spectral energies we will consider in the sequel will turn out to be continuous with respect to this metric, cf.~Lemmata~\ref{lem:energy-convergence} and~\ref{lem:min-energy-convergence}. This is 
an immediate corollary of Lemma~\ref{lem:eig-convergence}. However, if the $\parti_n$ are $k$-partitions and $\parti_n \to \parti_\infty$ for some $m$-partition $\parti_\infty$ with $m<k$, then we do \emph{not} in general expect spectral continuity, since the corresponding partition energies will diverge to $\infty$ (cf.\ the proof of Lemma~\ref{lem:energy-non-convergence}).

Nevertheless, as above, it is natural to consider the canonical completion $\overline{\colourset}$, which consists of equivalence classes of Cauchy sequences of partitions with respect to the equivalence relation of having vanishing distance $d(\parti_n,\widetilde{\parti}_n)$ in the limit, which corresponds to $\overline{\Theta_\colourset} \subset \R^E$.

\begin{lemma}
\label{lem:colour-set-compact}
Let $\cutset$ be a cut pattern of $\Graph$. Then $\overline{\colourset_{\cutset}}$ is compact.
\end{lemma}

\begin{proof}
This is immediate since $\overline{\colourset_{\cutset}}$ may be identified with the closed and bounded subset $\overline{\Theta_\colourset}$ of $E$-dimensional Euclidean space.
\end{proof}

More generally, if $A \subset \mathfrak{P}_k$ is any set of $k$-partitions, then $\overline{A}$ is the union of the sets $\overline{A \cap \colourset}$ over all primitive partitions $\colourset$. Obviously, it is possible that a given $m$-partition $\parti \in \overline{A}$ may lie in the closure of more than one primitive partition. Moreover, the sets $\mathfrak{P}_k$ and $\mathfrak{R}_k$ are themselves not closed, although, as we will see shortly,  $\bigcup\limits_{i\le k}\mathfrak{P}_k$ and $\bigcup\limits_{i\le k}\mathfrak{R}_k$ are.

It is also natural to ask which types of partition from our classification, Definition~\ref{def:classification}, are closed in the metric \eqref{eq:dist-pk}. 

\begin{example}\label{ex:lasso-closure}
Let us review the proper 2-partitions of the lasso graph of Section~\ref{sec:motivation}. As the separating point $\tilde{v}$ wanders towards $v$ in Figure~\ref{fig:basic-example-proper}, the corresponding partition $\mathcal P$ converges, with respect to the metric introduced in~\eqref{eq:dist-pk}, towards the faithful (but non-proper) 2-partition in Figure~\ref{fig:basic-example-faithful}. On the other hand, as the $\tilde{v}_1,\tilde{v}_2$ approach $v$ in Figure~\ref{fig:basic-example-proper-2}, the corresponding proper (and hence faithful) partition $\mathcal P$ converges towards the rigid, non-faithful 2-partition in Figure~\ref{fig:basic-example-rigid}. Observe that the cut pattern and hence the underlying discrete graphs of these two limiting partitions are different.
\end{example}

Hence, neither the class of proper, nor faithful partitions is closed; nor is the class of internally connected partitions, as can be shown using Example~\ref{ex:tree}.
In particular, connectivity of the clusters, even if it holds for a sequence of partitions, can be destroyed in the limit.
On the other hand, if $(\parti_n)_{n\in\N}\subset \colourset$ is a sequence of loose partitions of a given primitive partition, then the limit object is clearly still a well-defined $m$-partition for some $1\leq m \leq k$; in particular, it is loose, and thus $\bigcup\limits_{i\le k}\mathfrak{P}_k$ is closed. The following proposition establishes that a corresponding statement holds for rigid partitions; and it is for this reason that we will tend to favour these two partition classes over the respective classes of proper, faithful and internally connected ones.

\begin{proposition}
\label{prop:rigid-class-closed}
Suppose $\colourset$ is any primitive partition and $(\parti_n)_{n\in\N}\subset \colourset$ is a sequence of rigid $k$-partitions which is Cauchy with respect to the metric \eqref{eq:dist-pk}. Then the limit partition $\parti_\infty \in \overline{\colourset}$ is a rigid $m$-partition for some $m$, $1\leq m \leq k$. In particular, $\bigcup\limits_{i\le k}\mathfrak{R}_k$ is closed.
\end{proposition}

\begin{proof}
Fix $i=1,\ldots,k$. Now since obviously $\Graph_{i}^{(n)} \to \Graph_{i}^{(\infty)}$ with respect to the metric of equation~\eqref{eq:dist-gamma}, by Lemma~\ref{lem:part-convergence} we have that $\Graph_{i}^{(\infty)}$ is connected; in particular, $\parti_\infty$ cannot have more than $k$ clusters. To check the rigidity condition, we also need to show that any vertex $v \in \interior \Omega_i^{(\infty)}$ is not cut through in $\Graph_{i}^{(\infty)}$. So let $v \in \interior \Omega_i^{(\infty)}$ be arbitrary, then we first observe that $v \in \interior \Omega_i^{(n)}$ for all sufficiently large $n$. Since $\parti_n$ was assumed rigid, any edge of $\Graph$ incident with $v$ remains incident with $v$ in $\Graph_{i}^{(n)}$, and none of these edges in $\Graph_{i}^{(n)}$ has length converging to zero. In particular, the incidence relations at $v$ are preserved in the limit graph $\Graph_{i}^{(\infty)}$. We conclude that $\parti_\infty$ is rigid. 
\end{proof}

\subsection{Existence results for energy functionals}\label{sec:exis}

In this section we prove a general existence result for extremisers of functionals $\Lambda : \parti \mapsto \R$ defined on certain sets of partitions. 

Since each primitive partition is a metric space by Lemma~\ref{lem:primmetr}, as is the disjoint union of all primitive partitions (up to allowing the distance function to attain the value $+\infty$), all usual topological notions are well-defined: lower semicontinuity will play a key role in what follows.

\begin{definition}
\label{def:lsc}
Let $A \subset \mathfrak{P} (\Graph)$ be a set of exhaustive partitions. We say the functional $J: A \to \R$ is
\begin{enumerate}
\item \emph{lower semi-continuous} (lsc) if, whenever $\colourset$ is a primitive partition and $(\parti_n)_{n\in\mathbb N}\subset  A \cap \colourset$ converges to some $\parti \in A \cap \colourset$, we have that $J (\parti) \leq \liminf_{n\to\infty} J (\parti_n)$;
\item \emph{strongly lower semi-continuous} (slsc) if, whenever $\colourset$ is a primitive partition and $(\parti_n)_{n\in\mathbb N}$ $\subset  A \cap \colourset$ converges to some $\parti \in \overline{A \cap \colourset}$, we have that $J (\parti) \in \R$ is well defined and $J (\parti) \leq \liminf_{n\to\infty} J (\parti_n)$.
\end{enumerate}
\end{definition}

(Strong) upper semi-continuity and (strong) continuity may be defined analogously. Note, however, that continuity of $J$ is not assumed on the closure of its domain $A$; in particular, even if $J$ is continuous on the whole of $\mathfrak{P}$ or $\mathfrak{R}$, it need not be bounded from above or below, not even on the set of all $k$-partitions, since we do not rule out discontinuities, or even divergence, $J (\parti_n) \to \pm\infty$, if one or more clusters of $\parti_n$ disappear in the limit. (This will, for example, be the case for the \textit{continuous} functionals $\denergy[k,p]$ and $\nenergy[k,p]$, see Lemmata~\ref{lem:energy-convergence} and~\ref{lem:energy-non-convergence}.)

\begin{theorem}
\label{thm:abstract-min-existence}
Let $k\geq 1$ and let  $A \subset \mathfrak{P} = \mathfrak{P} (\Graph)$ with $A \cap \mathfrak{P}_k \neq \emptyset$. Suppose that the functional $J : \overline{A} \to \R$ is strongly lower semi-continuous on $A$.
Suppose in addition that \emph{at least one} of the following conditions holds:
\begin{enumerate}
\item $J (\parti_n) \to \infty$ whenever there exist clusters $\Graph_n $ in $\parti_n \in A$ such that $|\Graph_n| \to 0$ as $n\to\infty$; or
\item for every $\ell$-partition $\parti^{(\ell)} \in \overline{A}$, $\ell=1,\ldots,k-1$, there exists an $(\ell+1)$-partition $\parti^{(\ell+1)} \in \overline{A}$ such that $J (\parti^{(\ell+1)}) \leq J (\parti^{(\ell)})$.
\end{enumerate}
Then there is at least one exhaustive $k$-partition $\parti^\ast \in \overline{A} \cap \mathfrak{P}_k$ realising
\begin{equation}
\label{eq:abstract-min-existence}
	J (\parti^\ast) = \inf \{ J (\parti) : \parti \in A \cap \mathfrak{P}_k \}.
\end{equation}
If $A \subset \mathfrak{R}$, that is, if we restrict to rigid partitions, then there is at least one exhaustive rigid $k$-partition $\parti^\ast$ satisfying \eqref{eq:abstract-min-existence}.
\end{theorem}

If we assume $A$ to be contained in the set of proper, or faithful, or internally connected partitions, then in general the minimiser $\parti^\ast$ is merely rigid, since the former sets are not closed. We will give concrete examples of this elsewhere; see for example Example~\ref{ex:3-star} and also Section~\ref{sec:exhaustive-issue}, and cf.\ also Example~\ref{ex:lasso-closure}. 
We emphasise that lower semi-continuity by itself is not enough to guarantee the existence of a minimiser in $A$, since the lower semi-continuity condition does not require $J (\parti_n) \to J (\parti_\infty)$ if $A$ is open and $A \ni \parti_n \to \parti_\infty \in \partial A$, even if $J (\parti_\infty)$ is actually well defined. We likewise need (1) or (2) to prevent the only limits of any minimising sequences from being $m$-partitions for some $m<k$.

While the monotonicity-like condition in (2) may seem a little unusual, it will be directly applicable to the partitions of max-min type considered in Section~\ref{sec:ex-max}.

\begin{proof}[Proof of Theorem~\ref{thm:abstract-min-existence}]
Let $(\parti_n)_{n\geq 1}$ be a sequence of $k$-partitions in $A \cap \mathfrak{P}_k$ such that $J (\parti_n) \to \inf_{A\cap \mathfrak{P}_k} J (\parti)$ as $n\to \infty$. Since there are only finitely many primitive partitions, there must exist a subsequence, which we shall still denote by $(\parti_n)$, such that the $\parti_n$ are all similar, $\parti_n \in \colourset$ for some primitive partition $\colourset$. We will write $\parti_n := \{\Graph_{1}^{(n)},\ldots,\Graph_{k}^{(n)}\}$.

By Lemma~\ref{lem:colour-set-compact} there exists an $m$-partition $\parti_\infty \in \overline{A \cap \colourset}$, $m \leq k$, such that up to a subsequence $\parti_n \to \parti_\infty$. If $A \subset \mathfrak{R}$, that is, if all partitions under consideration are rigid, then since $\mathfrak{R}$ is closed by Proposition~\ref{prop:rigid-class-closed}, also $\parti_\infty \in \mathfrak{R}$.

To finish the proof, it suffices to show that $\parti_\infty$ is actually a $k$-partition, since the strong lower semi-continuity of $J$ already implies that
\begin{equation}
\label{eq:lsc-limit}
	J (\parti_\infty) \leq \liminf_{n\to\infty} J (\parti_n).
\end{equation}
Assume condition (1). Then since the sequence $(J(\parti_n))$ is bounded from above, $|\Omega_{i}^{(n)}|$ cannot converge to zero for any $i$, and hence, by Lemma~\ref{lem:part-convergence}, also $|\Graph_{i}^{(\infty)}|>0$ as required. We can thus take $\parti^\ast = \parti_\infty$.

Instead assume condition (2). Suppose that the minimising partition $\parti_\infty$ found above is an $m$-partition for some $1\leq m \leq k$. In this case, (2) still gives us \eqref{eq:lsc-limit}, and then, upon sufficiently many applications of the second part of (2) to $\parti_\infty$ we obtain a $k$-partition $\widetilde{\parti}$ with
\begin{displaymath}
	J (\widetilde{\parti}) \leq J (\parti_\infty) = \inf_{\parti \in A \cap \mathfrak{P}_k } J (\parti),
\end{displaymath}
meaning we have found a minimal $k$-partition $\widetilde{\parti} = \parti^\ast$. If $A \subset \mathfrak{R}$, then $\widetilde{\parti} \subset \overline{A} \subset \mathfrak{R}$ by assumption.
\end{proof}

\section{Existence of spectral minimal partitions}
\label{sec:ex-min}

We can now introduce the first major types of spectral energy functionals we wish to consider. From now on, we will no longer need to distinguish between metric and ur-graphs, so we will always suppress this technicality and assume that $\Graph$ is a fixed metric graph -- say, the canonical representative of an ur-graph. We also fix $k\geq 1$ and suppose that $\parti = \{ \Graph_1, \ldots, \Graph_k \} \in \mathfrak{P}_k = \mathfrak{P}_k(\Graph)$ is a $k$-partition of $\Graph$. On each of the graphs $\Graph_i$, $i=1,\ldots,k$, we consider either:
\begin{enumerate}
\item the smallest nontrivial eigenvalue $\mu_2 (\Graph_i)$ of the Laplacian with natural vertex conditions, given by \eqref{eq:mu-2}; we have $\mu_2 (\Graph_i) > 0$ since $\Graph_i$ is connected by definition; or
\item the smallest eigenvalue $\lambda_1 (\Graph_i) = \lambda_1 (\Graph_i; \VertexSet_D (\Graph_i)) > 0$ of the Laplacian with Dirichlet conditions at the vertex set $\VertexSet_D (\Graph_i)$, cf.~\eqref{eq:lambda-1} and~\eqref{eq:dirichlet-vertex-definition}.
\end{enumerate}
In either case, we associate a spectral energy with the graph $\Graph_i$, and thus, collating these over all $i$, with the partition $\parti$ of $\Graph$. There are multiple possible ways to do so; the particular problems we shall consider in this section are as follows: for any given $p \in (0,\infty]$, we consider the energies
\begin{equation}
\label{eq:nenergy}
	\nenergy[p] (\parti) = \begin{cases} \left(\frac{1}{k}\sum\limits_{i=1}^k \mu_2(\Graph_i)^p\right)^{1/p}
	\qquad &\text{if } p \in (0,\infty),\\ \max\limits_{i=1,\ldots,k} \mu_2(\Graph_i) 
	\qquad &\text{if } p = \infty, \end{cases}
\end{equation}
and
\begin{equation}
\label{eq:denergy}
	\denergy[p] (\parti) = \begin{cases} \left(\frac{1}{k}\sum\limits_{i=1}^k \lambda_1(\Graph_i)^p\right)^{1/p}
	\qquad &\text{if } p \in  (0,\infty),\\ \max\limits_{i=1,\ldots,k} \lambda_1(\Graph_i) 
	\qquad &\text{if } p = \infty, \end{cases}
\end{equation}
in each case for a given $k$-partition $\parti = \{\Graph_1,\ldots,\Graph_k \} \in \mathfrak{P}_k$ of a graph $\Graph$. Here we have written, and we will always understand,
\begin{displaymath}
	\lambda_1 (\Graph_i) = \lambda_1 (\Graph_i; \VertexSet_D (\Graph_i)),
\end{displaymath}
where we will \emph{always} take the set of Dirichlet vertices of $\Graph_i$ to be the set $\VertexSet_D (\Graph_i)$ of cut points in $\Graph_i$, as defined in Definition~\ref{def:separation-set} (indeed, this motivates the notation $\VertexSet_D$); likewise, we will write $H^1_0 (\Graph_i)$ in place of $H^1_0 (\Graph_i;\VertexSet_D (\Graph_i))$. The problem is then to minimise the energies \eqref{eq:nenergy} and \eqref{eq:denergy}, respectively, that is, to solve for
\begin{displaymath}
	\inf_{\parti \in A \cap \mathfrak{P}_k} \nenergy[p] (\parti) \qquad \text{and} \qquad
	\inf_{\parti \in A \cap \mathfrak{P}_k} \denergy[p] (\parti)
\end{displaymath}
for a suitable set or class of partitions $A$. There are multiple reasonably natural possible choices for the set $A$ over which we can seek the infimum, in particular, we may consider any of the classes listed in Definition~\ref{def:classification}. Iin keeping with the usual convention when dealing with domains we will be mostly interested in exhaustive partitions, 
although the problems we consider would also be well posed without this restriction.

For example, we recall that in \cite{BanBerRaz12} the authors were interested in proper (and exhaustive) partitions, and in particular thoroughly studied the local minima of $\denergy[\infty]$ and their geometric properties (see especially~\cite[Theorems~2.7 and 2.10]{BanBerRaz12}); however, the actual question of existence of minimisers was not discussed (whether within the class of proper partitions or in general). We also recall that the internally connected partitions are exactly those rigid partitions for which $\Graph_i$ remains connected after removing the sets $\VertexSet_D (\Graph_i) \simeq \partial\Omega_i$, making them an \emph{a priori} natural class of partitions on which to consider Dirichlet problems. However, as noted in Section~\ref{sec:abstract-existence}, these classes are not closed in the natural partition topology; hence we cannot expect to find a minimiser within the respective classes (see Example~\ref{ex:3-star} below) -- even though the energies \eqref{eq:nenergy} and \eqref{eq:denergy} are continuous, as we will show shortly. Based on the fact that the classes of rigid and loose partitions are closed, they will be of primary interest for us, that is, for $p\in (0,\infty]$ we will principally consider the four problems of finding
\begin{equation}
\label{eq:minimal-energies}
\begin{aligned}
	\noptenergy[k,p] = \noptenergy[k,p](\Graph) &:=\inf_{\parti \in \mathfrak{R}_k} \nenergy[p] (\parti), \\
	\doptenergy[k,p] = \doptenergy[k,p](\Graph) &:=\inf_{\parti \in \mathfrak{R}_k} \denergy[p] (\parti), \\
	\noptenergyloose[k,p] = \noptenergyloose[k,p](\Graph) &:=\inf_{\parti \in \mathfrak{P}_k} \nenergy[p] (\parti), \\
	\doptenergyloose[k,p] = \doptenergyloose[k,p](\Graph) &:=\inf_{\parti \in \mathfrak{P}_k} \denergy[p] (\parti),
\end{aligned}
\end{equation}
where we recall that $\mathfrak{P}_k$ and $\mathfrak{R}_k$ are the sets of all loose and rigid exhaustive $k$-partitions of $\Graph$, respectively. We will refer to these partition problems as \emph{Neumann} (or \emph{natural}) and \emph{Dirichlet} problems. 

\begin{example}
\label{ex:3-pumpkin-2part}
Let $\Graph$ be an equilateral pumpkin graph on 3 edges of, say, length 1. Then it is easy to check (cf.\ Lemma~\ref{lem:neumann-identify-observation} or Example~\ref{ex:pumpkin-example}) that $\mathcal L^N_{2,p}=\frac{4\pi^2}{9}$ for all $p\in (0,\infty]$: this value is attained by the partition $\mathcal P$ in Figure~\ref{fig:3-pumpkin-2-partition}, which is also unique up to isomorphism. (This rigid partition is also minimal among the loose ones, i.e., $\Lambda^N_p(\mathcal P)= \noptenergy[2,p]= \noptenergyloose[2,p]$ for all $p\in (0,\infty]$.)
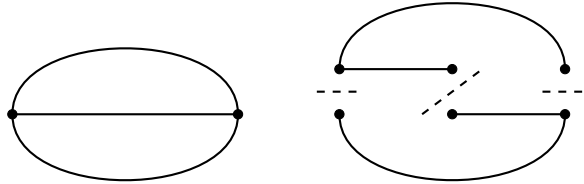
\begin{figure}[H]
\begin{tikzpicture}
\coordinate (a) at (0,0);
\coordinate (b) at (3,0);
\draw[thick,bend left=90]  (a) edge (b);
\draw[thick,bend right=90]  (a) edge (b);
\draw[thick] (a) -- (b);
\draw[fill] (a) circle (1.75pt);
\draw[fill] (b) circle (1.75pt);
\end{tikzpicture}
\qquad
\begin{tikzpicture}
\coordinate (a1) at (0,.3);
\coordinate (b1) at (3,.3);
\coordinate (a2) at (0,-.3);
\coordinate (b2) at (3,-.3);
\coordinate (c1) at (1.5,.3);
\coordinate (c2) at (1.5,-.3);
\draw[thick,bend left=90]  (a1) edge (b1);
\draw[thick,bend right=90]  (a2) edge (b2);
\draw[thick] (a1) -- (c1);
\draw[thick] (c2) -- (b2);
\draw[fill] (a1) circle (1.75pt);
\draw[fill] (a2) circle (1.75pt);
\draw[fill] (b1) circle (1.75pt);
\draw[fill] (b2) circle (1.75pt);
\draw[fill] (c1) circle (1.75pt);
\draw[fill] (c2) circle (1.75pt);
\draw[thick,dashed] (-.3,0) -- (.3,0);
\draw[thick,dashed] (2.7,0) -- (3.3,0	);
\draw[thick,dashed] (1.1,-.3) -- (1.9,.3);
\end{tikzpicture}
\caption{A minimal Neumann 2-partition of a pumpkin on 3 edges}\label{fig:3-pumpkin-2-partition}
\end{figure}
\end{example}

The principal goal of this section is to show that such minimal partitions always exist, and indeed for all four problems listed above.

\begin{remark}
\label{rem:neumann-on-domains}
In the case of domains, only the Dirichlet, not the Neumann problem, has been studied, since unlike on metric graphs, on domains the latter minimisation problem is not well defined. Let us expand on this point a little. Suppose for simplicity that $\Omega \subset \R^2$ is a smooth domain and consider $k$-partitions $\parti = \{\Omega_1,\ldots,\Omega_k\}$ of $\Omega$ for some fixed $k\geq 1$. In the easier case of non-exhaustive partitions, we take $\omega_\varepsilon$ to be a dumbbell consisting of two disks of radius $d=d(k,\Omega)>0$ small but fixed and handle of fixed length and thickness $\varepsilon>0$, where the parameters are chosen in such a way that $\Omega$ constains at least $k$ disjoint copies of $\omega_\varepsilon$, for any $\varepsilon>0$ sufficiently small. If $\parti_\varepsilon$ is then the partition consisting of $k$ such copies of $\omega_\varepsilon$, then since $\mu_2 (\omega_\varepsilon) \to 0$ as $\varepsilon \to 0$ in accordance with \cite[Section~2, pp.~16--19]{Col17}, we have that, for any $p \in (0,\infty]$,
\begin{displaymath}
	0 \leq \inf_\parti \nenergy[p](\parti) \leq \lim_{\varepsilon \to 0} \nenergy[p](\parti_\varepsilon) = 0,
\end{displaymath}
meaning that the minimisation problem is not well defined. If we restrict to exhaustive partitions, then the same principle and conclusion apply, but the construction is harder to describe in general. Instead, we illustrate the idea with a figure (Figure~\ref{fig:square-neumann-2-partition}) sketching such a construction for $2$-partitions of the square.
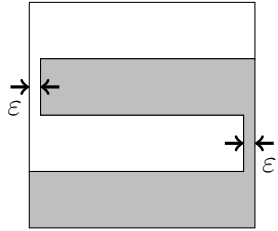
\begin{figure}[h]
\label{fig:square-neumann-2-partition}
\begin{tikzpicture}
\filldraw[fill=white] (0,3) -- (3,3) -- (3,2.25) -- (0.15,2.25) -- (0.15,1.5) -- (2.85,1.5) -- (2.85,0.75) -- (0,0.75) -- (0,3);
\filldraw[fill=lightgray] (3,2.25) -- (0.15,2.25) -- (0.15,1.5) -- (2.85,1.5) -- (2.85,0.75) -- (0,0.75) -- (0,0) -- (3,0) -- (3,2.25);
\draw[very thick,->] (-0.25,1.875) -- (0,1.875);
\draw[very thick,->] (0.4,1.875) -- (0.15,1.875);
\draw[very thick,->] (2.6,1.125) -- (2.85,1.125);
\draw[very thick,->] (3.25,1.125) -- (3,1.125);
\fill (0.05,1.825) node[anchor=north east] {$\varepsilon$};
\fill (2.95,1.075) node[anchor=north west] {$\varepsilon$};
\end{tikzpicture}
\caption{A sequence of exhaustive $2$-partitions of the unit square: one partition element is white, the other grey. The thin joining passages are of width $\varepsilon>0$; it can be shown that as $\varepsilon \to 0$, the corresponding Neumann partition energy also converges to $0$.}
\end{figure}
\end{remark}

Let us next give a few basic properties of the problems \eqref{eq:minimal-energies}.
It is immediate that $\noptenergy[k,p] \geq \noptenergyloose[k,p]$ and $\doptenergy[k,p] \geq \doptenergyloose[k,p]$, since $\mathfrak{R}_k \subset \mathfrak{P}_k$. Actually, before proceeding let us note that the only quantity of interest in the Dirichlet case is $\doptenergy[k,p]$:

\begin{lemma}
\label{lem:dirichlet-min-max-rigid-loose}
For any graph $\Graph$, any $k\geq 1$ and any $p \in (0,\infty]$, we have $\doptenergy[k,p] = \doptenergyloose[k,p]$.
\end{lemma}

\begin{proof}
We only have to prove ``$\leq$''. Let $\widetilde\parti \in \mathfrak{P}_k$; we will construct a partition $\parti \in \mathfrak{R}_k$ which has the same cluster supports and whose every cluster has an eigenvalue which is no larger than the eigenvalue of the corresponding cluster of $\widetilde\parti$. In fact, if $\NewGraph_i$ has support $\Omega_i$ and $\Graph_i \in \rho_{\Omega_i}$ is any rigid cluster with the same cluster support, then $H^1_0 (\NewGraph_i)$ may be identified with a subspace of $H^1_0(\Graph_i)$ since the zero condition can only be imposed at more points of $\NewGraph_i$ than of $\Graph_i$. It follows from the variational characterisation \eqref{eq:lambda-1} that $\lambda_1 (\NewGraph_i) \geq \lambda_1 (\Graph_i)$ for all $i$. The claim now follows.
\end{proof}

\begin{remark}
\label{rem:dirichlet-subset-graph-distinction}
Let us explicitly stress that if $\parti = \{\Graph_1,\ldots,\Graph_k\} \in \mathfrak{R}_k$ is a \emph{rigid} partition (hence, cut points and separation points agree) and we are interested in $\denergy[p](\parti)$, then $\lambda_1 (\Graph_i)$ is independent of the choice of the graph $\Graph_i \in \rho_{\Omega_i}$ associated with $\Omega_i$, $i=1,\ldots,k$, as long as this is made in accordance with Definition~\ref{def:partition}. In other words, $\lambda_1 (\Graph_i)$ is independent of $\Graph_i \in \rho_{\Omega_i}$. This is because a Dirichlet condition is imposed at all cut/separation points anyway; thus, it does not matter whether (or how) these vertices are joined in $\Graph_i$. Hence, in these cases, one may ignore the distinction between the cluster supports $\Omega_i$ and the clusters $\Graph_i$ (in particular, minimising over the class of faithful partitions is the same as minising over rigid partitions for the Dirichlet problem). In practice, we will always do this, that is, \emph{when considering (only) Dirichlet partition problems (among all rigid partitions) we will not distinguish between clusters and their supports}.
\end{remark}

We next establish that both our spectral energy functionals \eqref{eq:nenergy} and \eqref{eq:denergy} are indeed continuous with respect to the notion of partition convergence introduced in Section~\ref{subsec:part-convergence}.

\begin{lemma}
\label{lem:energy-convergence}
Suppose $\parti_n$ and $\parti_\infty$ are exhaustive loose $k$-partitions of a graph $\Graph$ such that the $\parti_n$ are similar, and $\parti_n \to \parti_\infty$ with respect to the metric of \eqref{eq:dist-pk}. Then, for any given $p \in (0,\infty]$,
\begin{displaymath}
	\nenergy[p] (\parti_n) \to \nenergy[p] (\parti_\infty) \quad \text{and} \quad \denergy[p] (\parti_n) \to \denergy[p] (\parti_\infty)
\end{displaymath}
as $n\to\infty$.
\end{lemma}

\begin{proof}
This follows immediately from the definitions of $\nenergy[p]$ and $\denergy[p]$ and Lemma~\ref{lem:eig-convergence}.
\end{proof}

Actually, we can say more.

\begin{lemma}
\label{lem:energy-non-convergence}
The functionals $\nenergy[p]$ and $\denergy[p]$ are \emph{strongly} lower semi-continuous on $\mathfrak{P}_k (\Graph)$ (see Definition~\ref{def:lsc}(2)), for any given $p \in (0,\infty]$.
\end{lemma}

\begin{proof}
If $\parti_n \in \mathfrak{P}_k$ are $k$-partitions of a graph $\Graph$ and $\parti_n \to \parti_\infty$ for some $m$-partition $\parti_\infty$ with $m < k$, then there exists at least one sequence of clusters, say $\Graph_i^{(n)}$, whose total length $|\Graph_i ^{(n)}| \to 0$. Nicaise' inequalities (Theorem~\ref{thm:nicaise}) yield
\begin{equation}\label{eq:nic-coerc}
	\mu_2 (\Graph_{i}^{(n)}) \geq \frac{\pi^2}{|\Graph_{i}^{(n)}|^2} = \frac{\pi^2}{|\Omega_{i}^{(n)}|^2} \quad \text{and} \quad
	\lambda_1 (\Graph_{i}^{(n)}) \geq \frac{\pi^2}{4|\Graph_{i}^{(n)}|^2} = \frac{\pi^2}{4|\Omega_{i}^{(n)}|^2}
\end{equation}
for all $i=1,\ldots,k$ and all $n\geq 1$, whence $\nenergy[p] (\parti_n), \denergy[p] (\parti_n) \to \infty$ as $n \to \infty$, for any $p \in (0,\infty]$. The strong lower semi-continuity follows since $\nenergy[p](\parti_\infty), \denergy[p](\parti_\infty) < \infty$, when combined with the result of Lemma~\ref{lem:energy-convergence}.
\end{proof}

Given a graph $\Graph$, we can now prove the existence of $k$-partitions achieving the infimal values $\noptenergy[k,p] (\Graph)$ and $\doptenergy[k,p] (\Graph)$ (and hence $\doptenergyloose[k,p](\Graph)$), as well as $\noptenergyloose[k,p](\Graph)$.

\begin{theorem}
\label{thm:smp-existence}
Fix $k\geq 1$ and $p\in  (0,\infty]$ and let $A \subset \mathfrak{P}_k$ be any set of $k$-partitions. Then there exist $k$-partitions $\parti^N,\parti^D  \in \overline{A} \cap \mathfrak{P}_k$ of $\Graph$ such that
\begin{displaymath}
	\nenergy[p] (\parti^N) = \inf_{A \subset \mathfrak{P}_k} \nenergy[p] (\parti) \qquad \text{and} \qquad \denergy[p] (\parti^D) = \inf_{A \subset \mathfrak{P}_k} \denergy[p] (\parti).
\end{displaymath}
In particular: if $A \subset \mathfrak{R}_k$ is a set of rigid $k$-partitions, then $\parti^N$ and $\parti^D$ are also rigid $k$-partitions, respectively.
\end{theorem}

We will see in Section~\ref{sec:p-dependence} that for fixed $k\geq 1$ the minimal partitions may depend on $p$.

\begin{proof}
It suffices to show that the functionals $\nenergy[p]$ and $\denergy[p]$, which are both defined on $\mathfrak{P} \supset \overline{A}$, satisfy the conditions of Theorem~\ref{thm:abstract-min-existence}(1). Strong lower semi-continuity was established in Lemma~\ref{lem:energy-non-convergence}, while condition (1) follows immediately from \eqref{eq:nic-coerc}.
\end{proof}

\begin{corollary}
\label{cor:smp-existence}
Fix $k\geq 1$ and $p\in (0,\infty]$. Then there exist a loose $k$-partition $\widetilde\parti^N \in \mathfrak{P}_k$ and rigid $k$-partitions $\parti^N, \parti^D \in \mathfrak{R}_k$ such that
\begin{displaymath}
	\nenergy[p](\widetilde\parti^N)=\noptenergyloose[k,p](\Graph),\qquad \nenergy[p](\parti^N)=\noptenergy[k,p](\Graph), \qquad \text{and}\qquad
	\denergy[p](\parti^D)=\doptenergy[k,p](\Graph)=\doptenergyloose[k,p](\Graph).
\end{displaymath}
\end{corollary}

\begin{remark}
\label{rem:noptenergyloose}
If, for any $k\geq 1$ and $p\in (0,\infty]$, $\widetilde\parti^N = \{\Graph_1,\ldots,\Graph_k\}$ is a loose $k$-partition achieving the minimum for $\noptenergy[k,p](\Graph)$, then we may always assume without loss of generality that the clusters $\Graph_1,\ldots,\Graph_k$ are all trees. This is because cutting through any vertices in $\Graph_i$ can only decrease $\mu_2 (\Graph_i)$ (see, e.g., \cite[Theorem~3.10(1)]{BeKeKuMu18}). We will see this principle in action in Example~\ref{ex:spantree} below.
\end{remark}

We next give a simple example to show that the minimal partition realising $\doptenergy[k,p] (\Graph)$ within the class of all internally connected exhaustive partitions need not be internally connected and exhaustive, even if it can be approximated by such partitions.

\begin{example}
\label{ex:3-star}
Let $\Graph$ be a star graph consisting of three edges $e_1$, $e_2$, $e_3$, each of length one, attached at a common vertex $v_0$. Then there is an optimal internally connected but non-exhaustive $2$-partition, for both $\doptenergy[k,\infty]$ and $\noptenergy[k,\infty]$, given by, say, $\parti^\ast = \{\Graph_1, \Graph_2 \}$, with $\Graph_1 = e_1$, $\Graph_2 = e_2$.

If, however, we search for rigid and exhaustive minimal partitions, then up to permutation of the edges, their cluster supports must all have the form $\Omega_1 = e_1$ and $\Omega_2 = e_2 \cup e_3$, see Example~\ref{ex:tree}. In the Dirichlet case, since $\partial\parti^\ast = \{v_0\}$ and removing $v_0$ disconnects $\Omega_2$, any rigid minimiser is not internally connected. A similar principle holds if we search for an optimal $k$-partition of a star on $n$ equal rays $e_1,\ldots, e_n$, with $n > k$.

We observe in passing that in all these cases the non-exhaustive partition achieving $\doptenergy[k,\infty] (\Graph)$ is \emph{nodal}, while the exhaustive partitions are \emph{generalised nodal}, see Definitions~\ref{def:nodal-partition} and~\ref{def:generalised-nodal} below.
\end{example}

\begin{remark}
\label{rem:dirichlet-k-mon}
(1) For any graph $\Graph$ and any given $p\in (0,\infty]$ and $1 \leq k_1 \leq k_2$, we have the monotonicity statements
\begin{equation}
\label{eq:dirichlet-k-mon}
	\doptenergy[k_1,p](\Graph) \leq \doptenergy[k_2,p](\Graph)
\end{equation}
and
\begin{equation}
\label{eq:neumann-loose-k-mon}
		\noptenergyloose[k_1,p](\Graph) \leq \noptenergyloose[k_2,p](\Graph).
\end{equation}
The argument is the same in both cases, so we restrict ourselves to the Dirichlet case: we suppose $\parti = \{\Graph_1,\ldots,\Graph_{k_2}\}$ is an exhaustive rigid $k_2$-partition realising $\doptenergy[k_2,p](\Graph)$, and for simplicity we take $k_1 = k_2 - 1$. Suppose without loss of generality that $\Graph_{k_2-1}$ and $\Graph_{k_2}$ are neighbours (see Definition~\ref{def:neighbours}), and that $\lambda_1 (\Graph_{k_2}) = \max_{i=1,\ldots,k_2} \lambda_1 (\Graph_i)$. We glue these two clusters together: more precisely, we define $\NewGraph_{k_1}$ to be the rigid cluster, unique in the sense of Remark~\ref{rem:dirichlet-subset-graph-distinction}, whose support is exactly $\Omega_{k_1} \cup \Omega_{k_2}$; we also set $\NewGraph_i := \Graph_i$ for all $i=1,\ldots,k_1$. Then it is easy to check that $\widetilde{\parti} := \{\NewGraph_1, \ldots, \NewGraph_{k_1}\}$ is an exhaustive rigid $k_1$-partition of $\Graph$. Moreover, by eigenvalue monotonicity with respect to graph inclusion, $\lambda_1 (\NewGraph_{k_1}) \leq \min \{\lambda_1 (\Graph_{k_2-1}), \lambda_1( \Graph_{k_2})\} = \lambda_1 (\NewGraph_{k_2-1})$. Hence, for any $p \in (0,\infty)$,
\begin{displaymath}
\begin{aligned}
	\frac{1}{k} \sum_{i=1}^k \lambda_1 (\Graph_i)^p &\geq \frac{1}{k} \sum_{i=1}^{k-1} \lambda_1(\widetilde\Graph_i)^p 
		+ \frac{1}{k}\max_{i=1,\ldots,k} \lambda_1 (\widetilde\Graph_i)^p\\
	&=\frac{1}{k-1}\sum_{i=1}^k \lambda_1 (\widetilde\Graph_i)^p  -\frac{1}{k(k-1)} \sum_{i=1}^k \lambda_1 (\widetilde\Graph_i)^p 
		+ \frac{1}{k}\max_{i=1,\ldots,k} \lambda_1 (\widetilde\Graph_i)^p\\
	&\geq \frac{1}{k-1}\sum_{i=1}^k \lambda_1 (\widetilde\Graph_i)^p  -\frac{k-1}{k(k-1)} \max_{i=1,\ldots,k} \lambda_1 (\widetilde\Graph_i)^p 
	+ \frac{1}{k}\max_{i=1,\ldots,k}\lambda_1 (\widetilde\Graph_i)^p = \denergy[k,p] (\widetilde\parti)^p.
\end{aligned}
\end{displaymath}
Since the inequality $\denergy[\infty] (\parti) \geq \denergy[\infty] (\widetilde\parti)$ is immediate, we obtain $\denergy[p] (\parti) \geq \denergy[p](\widetilde{\parti})$ for all $p \in (0,\infty]$. This yields \eqref{eq:dirichlet-k-mon}. For $\noptenergyloose[]$ we use Remark~\ref{rem:noptenergyloose} to guarantee that without loss of generality the clusters are all trees; now \cite[Theorem~3.10(1)]{BeKeKuMu18} gives the monotonicity when gluing $\Graph_{k_2-1}$ and $\Graph_{k_2}$ together, which may be done at a single vertex.

(2) Note that inequality in (1) need not always be strict: for $p=\infty$, if $\Graph$ is the equilateral star graph from Example~\ref{ex:3-star} (see also Examples~\ref{ex:tree} and~\ref{ex:pavels-task}), then
\begin{displaymath}
	\doptenergy[2,\infty](\Graph) = \doptenergy[3,\infty](\Graph) = \frac{\pi^2}{4},
\end{displaymath}
where in the notation of Example~\ref{ex:3-star} the optimal $3$-partition is given by $\Omega_i = e_i$ for $i=1,2,3$. Inequality need not be strict even if the corresponding minimal partitions are internally connected; an example is the graph considered in Proposition~\ref{prop:min-part-3M}: we will show there that for this graph we even have $\doptenergy[4,\infty] = \doptenergy[5,\infty]$ despite there being internally connected partitions realising both minima.
\end{remark}

However,  although for each $k$ there is a rigid $k$-partition achieving $\noptenergy[k,p](\Graph)$, it is not actually clear whether the monotonicity property analogous to \eqref{eq:dirichlet-k-mon} holds: a necessary condition is the following seemingly obvious conjecture, which will also play a role in Section~\ref{sec:ex-max}.

\begin{conjecture}
\label{conj:kirchhoff-two-cut}
Suppose $\Graph$ is a finite, compact, connected metric graph. Then there exists an exhaustive rigid $2$-partition $\parti_2 = \{\Graph_1, \Graph_2 \}$ of $\Graph$ such that
\begin{displaymath}
	\mu_2 (\Graph) \leq \min \{ \mu_2 (\Graph_1), \mu_2 (\Graph_2) \}.
\end{displaymath}
\end{conjecture}

In fact, if Conjecture~\ref{conj:kirchhoff-two-cut} is \emph{not} true, then there is a graph $\Graph$ such that $\mu_2 (\Graph) = \noptenergy[1,p](\Graph) > \nenergy[p] (\parti)$ for every exhaustive rigid $2$-partition $\parti$ of $\Graph$ and any $p \in (0,\infty]$, and hence $\noptenergy[1,p](\Graph) > \noptenergy[2,p] (\Graph)$. However, it is true for a large class of graphs, as the following observation shows. This will also be used in the proof of Theorem~\ref{thm:max-existence} below.

\begin{lemma}
\label{lem:kirchhoff-two-cut}
Conjecture~\ref{conj:kirchhoff-two-cut} is true whenever $\Graph$ has a \emph{bridge}, that is, an edge or a vertex whose removal disconnects $\Graph$. In particular, it is true for trees.
\end{lemma}

\begin{proof}
Suppose that removing $v \in \Graph$ disconnects $\Graph$ into two subsets $\Omega_1$ and $\Omega_2$ whose intersection is only $\{v\}$, and let $\Graph_1 \in \rho_{\Omega_1}$ and $\Graph_2 \in \rho_{\Omega_2}$ be the clusters of maximal connectivity, for which no further cuts are made at $v$, i.e., such that every other graph in $\rho_{\Omega_i}$ is a cut of $\Graph_i$, $i=1,2$ (in other words, the corresponding partition is faithful). It follows immediately (see, e.g., \cite[Section~5]{KuMaNa13} or \cite[Theorem~3.4]{BeKeKuMu18}) that
\begin{displaymath}
	\mu_2 (\Graph) \leq \mu_2 (\Graph_1),
\end{displaymath}
since $\Graph$ may be formed by attaching $\Graph_2$ as a pendant to $\Graph_1$ at the single vertex $v \in \Graph_2$. Interchanging the roles of $\Graph_1$ and $\Graph_2$ yields $\mu_2 (\Graph) \leq \mu_2 (\Graph_2)$ as well. 
\end{proof}

Moreover, using similar ideas, we can show that for \emph{sufficiently large} $k$, $\noptenergy[k,\infty](\Graph)$ is monotonically increasing in $k$. The na\"ive intuition behind this is that merging clusters should produce a partition with lower energy, but~\cite[Rem.~3.13]{BeKeKuMu18} shows that things are not that simple. Even the case of $p \in (0,\infty)$ requires a relatively fine control of the behaviour of the optimal partitions, and will be deferred to a later work \cite{HoKeMuPl20}.

\begin{proposition}
\label{prop:natural-monotonicity}
There exists $k_0 \in \N$ such that for any $k_2\geq k_1\geq k_0$,
\begin{equation}
\label{eq:natural-monotonicity}
	\noptenergy[k_2,\infty](\Graph) \geq \noptenergy[k_1,\infty](\Graph).
\end{equation}
If $\Graph$ is a tree, then we may take $k_0=1$.
\end{proposition}

\begin{proof}
Let us first give the proof for trees. Fix $k \geq 2$ arbitrary, and suppose $\parti_{k} = \{\Graph_1, \ldots, \Graph_{k} \}$ is an optimal exhaustive $k$-partition for $\noptenergy[k,\infty](\Graph)$, which we know exists by Corollary~\ref{cor:smp-existence}. We will construct a (test) ($k-1$)-partition from $\parti_k$ whose energy is no larger than $\noptenergy[k,\infty](\Graph)$, from which we may conclude that $\noptenergy[k,\infty](\Graph) \geq \noptenergy[k-1,\infty](\Graph)$; the claim of the proposition for trees then follows immediately. Suppose without loss of generality that $\Omega_{k-1}$ and $\Omega_k$ are neighbours (see Definition~\ref{def:neighbours}). Since $\Graph$ is a tree, they can only meet at a single point, without loss of generality a vertex $v$. Moreover, since $\Graph_{k-1}$ and $\Graph_k$ are connected and $\Graph$ was a tree, the image of $v$ in $\Graph_i$ is a single vertex, $i=k,k-1$ (cf.\ Example~\ref{ex:tree}). In particular, if we create a new graph $\widetilde\Graph_{k-1}$ by attaching $\Graph_k$ to $\Graph_{k-1}$ at $v$, then $\Graph_k$ is a pendant of $\widetilde\Graph_{k-1}$ at $v$ and vice versa, and so, as above,
\begin{equation}
\label{eq:join-for-less}
	\mu_2 (\widetilde\Graph_{k-1}) \leq \min \{ \mu_2 (\Graph_{k-1}), \mu_2 (\Graph_k) \}.
\end{equation}
Moreover, if we set $\widetilde\Omega_{k-1} := \Omega_{k-1} \cup \Omega_k \subset \Graph$, then $\widetilde\Graph_{k-1} \in \rho_{\widetilde \Omega_{k-1}}$ and the new partition $\widetilde\parti := \{\Graph_1, \ldots, \Graph_{k-2}, \widetilde\Graph_{k-1} \}$ is a rigid $k-1$-partition of $\Graph$. Combining \eqref{eq:join-for-less} with the definition of $\nenergy[\infty]$ as a maximum and the fact that no other cluster was affected, we immediately have
\begin{equation}
\label{eq:newnenergy}
	\nenergy[\infty] (\widetilde\parti) \leq \nenergy[\infty] (\parti_k).
\end{equation}
This proves the proposition for trees.

The proof for general $\Graph$ is based on the idea above together with the principle that for sufficiently large $k$, we can always find neighbouring clusters whose supports meet at a single vertex like $\Omega_{k-1}$ and $\Omega_k$ did above. For simplicity, we take $k_0 := 4\cardE$, where as usual $\cardE$ is the number of edges of $\Graph$, although this $k_0$ will in general be far from optimal. Fix $k\geq k_0+1$ and as above denote by $\parti_k = \{\Graph_1, \ldots, \Graph_{k} \}$ an optimal $k$-partition for $\noptenergy[k,\infty](\Graph)$. Now by the pigeonhole principle, there exists at least one edge of $\Graph$ with non-empty intersection with at least four cluster supports. It follows that there exist two neighbouring cluster supports, call them $\Omega_{k-1}$ and $\Omega_k$, which are contained in the interior of this edge. Since their intersection must consist of a single point, we may apply verbatim the above argument for trees to the corresponding clusters $\Graph_{k-1}$ and $\Graph_k$ to generate a test $k-1$-partition $\widetilde\parti$ with lower energy $\nenergy[\infty]$. This completes the proof.
\end{proof}

\begin{remark}
As in the Dirichlet case (Remark~\ref{rem:dirichlet-k-mon}), it is easy to construct examples such as stars for which there is equality in \eqref{eq:natural-monotonicity}. 
\end{remark}

\section{Nodal and bipartite minimal Dirichlet partitions}
\label{sec:nodal}

We now wish to consider in detail the relationship between Dirichlet spectral minimal partitions of a graph $\Graph$ and eigenfunctions of the Laplacian on $\Graph$, analogous to the results that have been established in recent years linking partitions of domains $\Omega$ with eigenfunctions of the Dirichlet Laplacian on $\Omega$, such as discussed in \cite{HelHofTer09} and related works. On graphs $\Graph$, however, the correct analogue of the Dirichlet Laplacian on $\Omega$ will be the Laplacian with natural vertex conditions, see Section~\ref{subsec:def}.

Since in this section we will be working exclusively with the Dirichlet minimisation problem for rigid partitions, we will not generally distinguish between the cluster supports $\Omega_i \subset \Graph$ and the clusters $\Graph_i$ themselves of a partition $\parti$ of $\Graph$, in accordance with Remark~\ref{rem:dirichlet-subset-graph-distinction}. In such cases, in a slight but simplifying abuse of our own terminology we will speak of the clusters themselves as being subsets of $\Graph$. Recall that for the Dirichlet minimisation problem there is no reason to consider loose partitions.

\subsection{Nodal and equipartitions}
\label{subsec:nodal-and-equi}

We begin by introducing two properties of generic loose partitions of a metric graph.

\begin{definition}
\label{def:equipartition}
We say that a $k$-partition $\parti = \{\Graph_1,\ldots,\Graph_k\}$  of $\Graph$ is a \emph{Dirichlet ($k$-)equipartition} if $\lambda_1 (\Graph_1) = \ldots = \lambda_1 (\Graph_k)$ and a \emph{natural ($k$-)equipartition} if $\mu_2 (\Graph_1) = \ldots = \mu_2 (\Graph_k)$; or simply an \emph{equipartition} whenever the spectral problem being considered is clear from the context.
\end{definition}

If $\parti = \{\Graph_1,\ldots,\Graph_k\}$ is a Dirichlet equipartition, then its energy $\denergy[p](\parti)$ is independent of $p$, being identically equal to $\lambda_1(\Graph_1)$; in this case, we will refer to $\lambda_1 (\Graph_1)$ as the \emph{Dirichlet energy} of the partition, or just \emph{energy} if the Dirichlet condition is clear from the context. Likewise, if $\parti$ is a natural equipartition, then $\mu_2 (\Graph_1)$ is its \emph{natural energy} (or just \emph{energy}).

\begin{definition}
\label{def:nodal-partition}
Let $\psi$ be an eigenfunction associated with $\mu_j(\Graph)$. We call the \emph{nodal partition associated with $\psi$} the unique internally connected, possibly non-exhaustive partition whose support $\Omega$ (see Definition~\ref{def:cluster-support}) is the union of edges of $\Graph$ on which $\psi$ does not vanish identically, and whose cut set is the zero set of $\psi$ on the support $\Omega$. We call the cluster supports of this partition the \emph{nodal domains of $\psi$}, and denote by $\nu(\psi)$ the number of nodal domains. We say that a given partition is \emph{nodal} if it is the nodal partition associated with some eigenfunction.
\end{definition}

Since it is possible for eigenfunctions to vanish identically on one or more edges of a graph, corresponding to a non-exhaustive nodal partition, in such cases there is some freedom as to how exactly to construct a partition out of the eigenfunction; this leads to the following definition. Let us stress once again that $\lambda_1(\mathcal G_i)$ denotes the lowest eigenvalue of the Laplacian on $\mathcal G_i$, where Dirichlet conditions are imposed at all cut points.

\begin{definition}
\label{def:generalised-nodal}
Let $\parti = \{\Graph_1,\ldots,\Graph_k\}$ be a $k$-partition of $\Graph$. Then we say that $\parti$ is a \emph{generalised nodal partition} if there exist eigenfunctions $\psi_1,\ldots,\psi_k$ for $\lambda_1 (\Graph_1),\ldots,\lambda_1(\Graph_k)$ with the following properties:
\begin{enumerate}
\item for each $i$, there exists a cut $\NewGraph_i$ of $\Graph_i$ such that on each connected component of $\NewGraph_i$ either $\psi_i$ is identically zero, or the connected component is the closure of a nodal domain of $\psi_i$ on $\Graph_i$; and
\item the $k_0$-partition $\widetilde{\parti}$ of $\Graph$, $k_0 \geq k$, consisting of all connected components of $\NewGraph_i$ on which $\psi_i$ is not identically zero, for all $i=1,\ldots,k$, is a nodal partition associated with some eigenfunction of $\Graph$.
\end{enumerate}
\end{definition}

At the risk of being redundant, let us elaborate on Definition~\ref{def:generalised-nodal}. As $\mathcal P$ may fail to be internally connected, Dirichlet conditions may be imposed on vertices of clusters $\Graph_i$ in such a way that $\Graph_i$ is \textit{de facto} disconnected; and in particular, the lowest eigenvalue need not be simple and the ground state need not be strictly positive; indeed it may vanish identically on an edge, as the following examples show (see also Example~\ref{ex:3-star}).

\begin{example}\label{exa:nodagnp}
(1) Let $\Graph$ be the equilateral pumpkin on $3$ edges of length $1$  (Example~\ref{ex:3-pumpkin-2part}); then the eigenvalue $\mu_2 (\Graph) = \pi^2$ has multiplicity three. If $\psi$ is taken as the eigenfunction which is monotonic on each edge and invariant under permutations of the edges (\emph{longitudinal}, in the language of \cite[Section~5.1]{BeKeKuMu18}), then the corresponding nodal partition is an exhaustive faithful $2$-partition of $\Graph$ whose clusters are both $3$-stars with edges of length $1/2$ each. If $\psi$ is a \emph{transversal} eigenfunction, supported on two of the edges and identically zero on the third, then the nodal partition is a non-exhaustive rigid $2$-partition of graph, whose clusters are each edges of $\Graph$ (and the third edge is not in the cluster support). If we take $\psi$ a linear combination of transversal eigenfunctions which has its zeros at the vertices of $\Graph$, is positive on two edges and negative on the third, then the result is an exhaustive rigid $3$-partition whose clusters are the edges of $\Graph$.

(2) If we now take $\Graph$ to be a star on $3$ edges $e_1,e_2,e_3$ of lengths $1$, $1$ and $\varepsilon \in (0,1)$, respectively, then $\mu_2 (\Graph) = \pi^2$ has multiplicity one, with eigenfunction supported on $e_1 \cup e_2$. The unique corresponding nodal partition has clusters $e_1$ and $e_2$; $e_3$ is not in the support of the eigenfunction and hence not in the support of the partition. However, the exhaustive $2$-partition $\parti = \{e_1, e_2 \cup e_3\}$, which is rigid but not internally connected, is a generalised nodal partition, as we can recover the nodal partition upon removing the extraneous edge $e_3$ on which the eigenfunction vanishes from the cluster $e_2 \cup e_3$.
\end{example}

Thus a partition $\parti$ is a generalised nodal partition of $\Graph$ if there exists an eigenfunction whose eigenvalue equals the energy of the partition, and whose nodal domains correspond exactly to subsets of clusters of $\parti$ -- but there may be parts of these clusters on which the eigenfunction vanishes identically. The non-exhaustive partition obtained by removing the latter parts is then nodal. We will return to a slightly different aspect of extracting such nodal-type partitions from more general partitions in Proposition~\ref{prop:exhaustive-shrinking}.

All nodal partitions are rigid, since no eigenfunction on $\Graph$ has an isolated zero without changing sign in a neighbourhood of it; this fact follows from the fact that the eigenfunction satisfies the Kirchhoff condition at every point of the graph (see also~\cite{Kur19} for a more general discussion of eigenfunction positivity). Furthermore, all nodal partitions are Dirichlet equipartitions whose energy is the associated eigenvalue. In fact, more generally, the Dirichlet energy $\denergy[\infty](\parti)$ of any generalised nodal partition is necessarily equal to the corresponding eigenvalue. Next, we extend to metric graphs two relationships between Laplacian eigenvalues and optimal Dirichlet  partitions which are well known in the case of domains (see \cite[Proposition~10.6 and eq.~(10.44)]{BNHe17}), and give a partial extension to metric graphs of Courant's Nodal Domain Theorem.

\begin{proposition}
\label{prop:inequalities-k}
We have $\mu_k(\Graph)\le\doptenergy[k,\infty](\Graph)$.
\end{proposition}

\begin{proof} Let us denote by $\parti = \{\Graph_1,\ldots,\Graph_k\}$ an arbitrary loose $k$-partition and by $\varphi_1,\dots,\varphi_k$ the normalised positive ground states of $\Graph_1,\dots,\Graph_k$, respectively; that is, $\varphi_i$ is the positive eigenfunction associated with $\lambda_1(\Graph_i)$, $i=1,\ldots,k$, chosen to have $L^2$-norm $1$. Let us denote by $\psi_1,\dots,\psi_{k-1}$ orthonormalised eigenfunctions associated with $\mu_1(\Graph),\dots,\mu_{k-1}(\Graph)$ respectively. We set, for a fixed $k$-tuple $(t_1,\dots,t_k)\in \mathbb R^k$ which will be specified later,
\begin{displaymath}
	\phi:=t_1\varphi_1+\dots+t_k\varphi_k.
\end{displaymath}
The system of equations
\begin{displaymath}
	\langle \psi_i,\phi\rangle=\sum\limits_{j=1}^k t_j\langle \psi_i,\varphi_j\rangle=0
\end{displaymath}
has size $(k-1)\times k$ and so rank at most $k-1$. Hence there exists $(t_1,\dots,t_k)\in \mathbb R^k$ such that $\langle \psi_i,\phi\rangle=0$ for all $i\in\{1,\dots,k-1\}$ and $\sum\limits_{j=1}^k t_j^2=1$. Then, from the variational characterisation of the eigenvalues,
\begin{displaymath}	
	\mu_k(\Graph)\le\int_\Graph(\phi')^2\,\textrm{d}x=\sum\limits_{j=1}^k t_j^2\int_{\Graph_j}(\varphi_j')^2\,\textrm{d}x=\sum\limits_{j=1}^kt_j^2\lambda_1(\Graph_j)\le \denergy[\infty](\parti) \le \doptenergy[k,\infty](\Graph).
\end{displaymath}
This concludes the proof.
\end{proof}

\begin{proposition}
\label{prop:link-to-second-eigenvalue}
We have $\mu_2(\Graph)=\doptenergy[2,\infty](\Graph)$, and any partition realising $\doptenergy[2,\infty](\Graph)$ is a generalised nodal partition.
\end{proposition}

\begin{proof} Let $\psi_2$ be an eigenfunction associated with $\mu_2(\Graph)$, $\parti$ its associated nodal partition and $\nu$ the cardinality of $\parti$. The function $\psi_2$ is orthogonal to the constants in $L^2(\Graph)$ and therefore changes sign. It follows that $\nu\ge2$ and thus
$\doptenergy[2,\infty](\Graph)\le\doptenergy[\nu,\infty](\Graph)=\mu_2(\Graph)$ by Remark~\ref{rem:dirichlet-k-mon}.(1), while $\mu_2(\Graph)\le\doptenergy[2,\infty](\Graph)$ according to Proposition \ref{prop:inequalities-k}. We have shown that $\doptenergy[2,\infty](\Graph)=\mu_2(\Graph)$.

Let us now consider a partition $\parti^* = \{\Graph_1^*, \Graph_2^* \}$ realising $\doptenergy[2,\infty](\Graph)$ and let us denote by $\varphi_1^*$ and $\varphi_2^*$ the normalised positive ground states of $\Graph_1^*$ and $\Graph_2^*$. There exists $(t_1,t_2)\in \mathbb R$ such that $\psi=t_1\varphi_1^*+t_2\varphi_2^*$ is orthogonal to the constants and $t_1^2+t_2^2=1$.  Then, from the variational characterisation of the eigenvalues,
\begin{displaymath}	
	\mu_2(\Graph)\le\int_\Graph(\psi')^2\,\textrm{d}x=t_1^2\int_{\Graph_1^*}((\varphi_1^*)')^2\,\textrm{d}x+t_2^2\int_{\Graph_2^*}((\varphi_2^*)')^2\,\textrm{d}x
=t_1^2\mu_1(\Graph_1^*)+t_2^2\mu_1(\Graph_2^*)\le \doptenergy[2,\infty](\Graph).
\end{displaymath}
Since $\doptenergy[2,\infty](\Graph)=\mu_2(\Graph)$, the above inequality implies that $\parti^*$ is an equipartition and $\mu_2(\Graph)=\int_\Graph(\psi')^2\,\textrm{d}x$. By the variational characterisation, $\psi$ is an eigenfunction associated with $\mu_2(\Graph)$.
\end{proof}

\begin{proposition}[Weak Courant Theorem]
\label{prop:weak-courant}
Given an eigenvalue $\mu_{k}( \Graph)$ and an associated eigenfunction $\psi$, denote by $\kappa(\mu_k(\Graph))$ the integer
\begin{displaymath}
	\kappa(\mu_k(\Graph)):=\max\{j\in \mathbb N: \mu_j(\Graph)=\mu_k(\Graph)\},
\end{displaymath}
and by $\nu(\psi)$ the number of nodal domains of $\psi$. Then $\nu(\psi)\le\kappa(\mu_k(\Graph))$.
\end{proposition}

This theorem is well known for domains, where in fact the stronger statement $\nu(\psi) \leq k$ always holds, and has been the subject of much work since Pleijel's groundbreaking paper \cite{Ple56}; see, for example, \cite[Section~10.2.4]{BNHe17} or \cite{Len19} and the references therein. This weaker version also holds for discrete graph Laplacians; see \cite[Theorem~2]{DavGlaLey01}. For metric graphs for which all Laplacian eigenvalues are simple, $\kappa(\mu_k(\Graph))=k$, while $\nu(\psi)$ is the corresponding nodal count.Hence the above inequality reduces to the main result in~\cite{GnuSmiWeb04}, obtained for more general Schr\"odinger operators.

\begin{proof} To simplify notation, we set $\nu:=\nu(u)$ and $\kappa:=\kappa(\mu_k(\Graph))$. Let us assume for a contradiction that $\nu\ge \kappa+1$.  We denote by $\parti$ the nodal partition associated with $u$.  We have $\mu_\kappa(\Graph)=\denergy[\infty](\parti)\ge\doptenergy[\nu,\infty](\Graph)$. According to Proposition \ref{prop:inequalities-k}, $\mu_\nu(\Graph)\le\doptenergy[\nu,\infty](\Graph)$.  Since $\mu_{\kappa+1}(\Graph)\le\mu_\nu(\Graph)$ and $\mu_\kappa(\Graph)<\mu_{\kappa+1}(\Graph)$, by definition of $\kappa=\kappa(\mu_k(\Graph))$, we obtain $\mu_\kappa(\Graph)<\mu_\kappa(\Graph)$. 
\end{proof}

\subsection{Bipartite minimal partitions}
\label{subsec:bipartite}

As observed in \cite{BanBerRaz12}, the links between minimal and nodal partitions appear more clearly if  we restrict ourselves to partitions that are proper. 
The following theorem can be deduced immediately from results in \cite{BanBerRaz12}.
We state it using our notation for future reference.

\begin{definition}\label{def:proximity}
Let $\parti$ be a partition of $\Graph$.  We say that $\parti$ is \textit{bipartite} if each of its clusters can be marked with signs $+$ or $-$ in such a way that any two clusters have different signs if their supports are neighbours.
\end{definition}

\begin{theorem}
\label{thm:min-nodal}
Let $\parti$ be an exhaustive, proper Dirichlet minimal $k$-partition of $\Graph$, that is, such that $\denergy[\infty](\parti) = \doptenergy[k,\infty] (\Graph)$. Then $\parti$ is bipartite if and only if it is nodal.
\end{theorem}

Note that the assumptions that $\parti$ is proper and Dirichlet minimal imply that $\parti$ is necessarily exhaustive.

\begin{proof}If $\parti$ is nodal, we see that it is bipartite by using the sign of the corresponding eigenfunction and the assumption that $\parti$ is proper. Conversely, let us assume that $\parti$ is bipartite. Since $\parti$ is minimal, it is in particular critical for the functional $\Lambda$ introduced in Definition 2.6 in \cite{BanBerRaz12} and corresponding to $\Lambda_\infty^D$ in our notation. From \cite[Theorem~2.10(1)]{BanBerRaz12}, we conclude that $\parti$ is nodal, under the condition that $k$ is large enough. According to \cite[Theorem~5.2]{BanBerRaz12}, the conclusion actually holds without this last condition.
\end{proof}

The assumption that $\parti$ is proper is crucial in the previous theorem. However, if $\Graph$ is a \emph{tree}, then \emph{any} (exhaustive) equipartition is nodal, since we can recursively construct an eigenfunction out of the ground  states of the clusters.

\begin{theorem}
\label{thm:nodal-tree}
Let $\Graph$ be a compact tree and suppose $\parti = \{\Graph_1,\ldots,\Graph_k\}$ is an exhaustive Dirichlet $k$-equipartition of $\Graph$, $k\geq 2$. Then $\parti$ is nodal.
\end{theorem}

We note explicitly that $\parti$ does not have to be a \emph{minimal} partition of $\Graph$ for the theorem to hold: important is merely the equipartition property.

While Theorem~\ref{thm:nodal-tree} is obvious for \emph{proper} partitions (and indeed the corresponding eigenfunction is the $k$-th eigenfunction of $\Graph$, see \cite{Ban14}), the proof in the general case requires more work. We first wish to establish some basic structural properties of partitions of trees, since the proof will consist of gluing together the eigenfunctions of the individual clusters recursively in the right way. We will use the terminology introduced in Section~\ref{sec:part}, in particular Definitions~\ref{def:cluster-support},~\ref{def:separation-set} and~\ref{def:neighbours}, without further comment; in particular, we assume without loss of generality that the cut set of a partition consists only of vertices of $\Graph$.

\begin{lemma}
\label{lem:tree-partition-structure}
Under the assumptions of Theorem~\ref{thm:nodal-tree} the following assertions hold.
\begin{enumerate}
\item Each $\Graph_i$ is itself a tree, $i=1,\ldots,k$.
\item Any two cluster supports share at most one separating point, that is, for all $i\neq j$, $\Omega_i \cap \Omega_j$ consists of at most one vertex.
\item There exists at least one $i=1,\ldots,k$ for which $\partial\Omega_i$ consists of exactly one vertex.
\end{enumerate}
\end{lemma}

\begin{proof}
(1) Trivial.

(2) Follows from a simple argument tracing paths between clusters.

(3) Suppose for a contradiction that for each $i=1,\ldots,k$, $\partial\Omega_i$ contains at least two vertices. We will construct a loop in $\Graph$ using (2). Start at any cut point, call it $v_0$, choose a neighbouring support $\Omega_1$ and a second vertex $v_0 \neq v_1 \in \partial\Omega_1$. Let $\gamma_1$ be (the unique image in $\Graph$ of) a continuous, injective mapping of $[0,1]$ into $\Omega_1$ which goes from $v_0$ at $0$ to $v_1$ at $1$. Now, at $v_1$, pick some other neighbour $\Omega_2 \neq \Omega_1$ with another vertex $v_1 \neq v_2 \in \partial\Omega_2$ and an injective path $\gamma_2$ from $v_1$ to $v_2$ within $\Omega_2$. Repeat this process inductively: then, for some $2 \leq m \leq k$, we must have that $\Omega_m = \Omega_i$ for some $1 \leq i \leq m$. At this point, on choosing $v_m = v_{i-1}$, we see that $\gamma_i \cup \ldots \gamma_m$ forms a closed path from $v_{i-1}$ to itself, which is injective except at at most a finite number of vertices. This contradicts the assumption that $\Graph$ was a tree.
\end{proof}

We may imagine any cluster support satisfying condition (3) to be at the ``end'' of the graph in some sense, and build a natural hierarchy of neighbours emanating from it. More precisely, suppose $\Omega_1$ is any such support. We will refer to the \emph{level} of any support within the same partition, with respect to $\Omega_1$, via the rules:
\begin{itemize}
\item $\Omega_1$ has level zero;
\item any neighbour of $\Omega_1$ has level one;
\item the level of any support other than $\Omega_1$ is the minimum of the levels of its neighbours, plus one.
\end{itemize}
Thus the level is, loosely speaking, the number of supports we need to traverse to reach $\Omega_1$ (excluding $\Omega_1$ itself). By Lemma~\ref{lem:tree-partition-structure}, this is well defined, and indeed any support $\Omega_{i_m}$ of level $m\geq 1$ can only have supports of level $m-1$, $m$ and/or $m+1$ as neighbours. It always shares a separating point $v_m$ with exactly one support $\Omega_{i_{m-1}}$ of level $m-1$: $\partial\Omega_{i_m} \cap \partial\Omega_{i_{m-1}} = \{v_m\}$; moreover, if $\Omega_{j_m}$ is any other neighbour of $\Omega_{i_m}$ of level $m$, then also $\partial\Omega_{i_m} \cap \partial\Omega_{j_m} = \{v_m\}$.

With this background and terminology, we can give the proof of Theorem~\ref{thm:nodal-tree}.

\begin{proof}[Proof of Theorem~\ref{thm:nodal-tree}]
Set $\mu := \lambda_1 (\Graph_1) = \ldots = \lambda_1 (\Graph_k)$, denote by $\Omega_1,\ldots,\Omega_k \subset \Graph$ the cluster supports corresponding to $\Graph_1,\ldots,\Graph_k$, and denote by $\psi_1,\ldots, \psi_k \in H^1 (\Graph)$ the functions supported on $\Omega_i$, $i=1,\ldots,k$, which correspond to the eigenfunctions of $\Graph_1,\ldots, \Graph_k$, respectively, normalised to have $L^2$-norms one; that is, $\psi_i$ satisfies $-\psi''(x)=\mu\psi(x)$ in the interior of each edge of $\Omega_i$, and $\psi(x)=0$ for all $x \in \Graph \setminus \Omega_i$.

We will define an eigenfunction $\psi$ on $\Graph$ of the form
\begin{equation}
\label{eq:def-tree-psi}
	\psi(x) = \sum\limits_{i=1}^k t_i \psi_i(x)
\end{equation}
for all $x \in \Graph$, for coefficients $t_i \in \R \setminus \{0\}$ to be chosen in such a way that $\psi$ satisfies the Kirchhoff condition at each vertex of $\Graph$. (Since the $\Omega_i$ are pairwise disjoint, it is clear from the outset that $-\psi''=\mu\psi$ on each edge and $\psi$ satisfies the continuity condition at each vertex of $\Graph$; hence, to check that $\psi$ is an eigenfunction on $\Graph$ for $\mu$, we only need to check the Kirchhoff condition.)

To this end, we first note that at every boundary vertex $v \in \VertexSet_D$, for any edge $e $ incident to $v$, if, say, $e \subset \Omega_i$, then by Lemma~\ref{lem:tree-partition-structure} and the fact that $\psi_i$ is strictly positive everywhere in $\Omega_i$ (see \cite[Theorem~3]{Kur19}), the outer normal derivative $\partial_\nu \psi_i|_e (v)$ of $\psi_i$ on $e$ at $v$ is different from zero.

Based on this observation and the level structure of the supports $\Omega_i$ as established in Lemma~\ref{lem:tree-partition-structure}, a routine induction on the level can be used to specify the $t_i$.
\end{proof}

We finish this subsection with an example of an exhaustive minimal partition which is not an equipartition, nor is it bipartite; but it still corresponds to a nodal partition. This illustrates the difficulties in extending Theorem~\ref{thm:min-nodal} to non-proper partitions; in particular, there does not seem to be a natural concept of ``bipartite'' partitions that would allow the theorem to hold.

\begin{example}
\label{ex:non-equi-3-star}
Suppose that $\mathcal{H}_\varepsilon$ is the $3$-star (cf.~Example~\ref{ex:3-star}) having edges $e_1$, $e_2$ and $e_3$ of length $1+\varepsilon$, $1$ and $1$, respectively, for some small $\varepsilon\geq 0$, joined at a common vertex $v_0$. Then the unique internally connected partition achieving $\doptenergy[3,\infty](\mathcal{H}_\varepsilon)$ has $\{v_0\}$ as its cut set, with $\Graph_i = e_i$, $i=1,2,3$ and thus energy $\pi^2/4$; this corresponds to the eigenvalue $\mu_3 (\mathcal{H}_\varepsilon)$, whose eigenfunction is supported on the edges of length $1$ and identically zero on $e_1$: in particular, this eigenfunction has a nodal pattern corresponding to the partition. Note that this is not an equipartition if $\varepsilon>0$, in opposition to domains where a Dirichlet minimal partition is always an equipartition (see \cite[Proposition~10.45]{BNHe17}). If $\varepsilon=0$, then this is a (Dirichlet) equipartition and corresponds to the nodal partition for $\mu_2 (\mathcal{H}_0) = \mu_3 (\mathcal{H}_0)$ (more precisely, it is nodal for $\mu_3 (\mathcal{H}_0)$, and generalised nodal with respect to $\mu_2 (\mathcal{H}_0)$). However, it is not bipartite.
\end{example}

\subsection{Courant-sharp eigenfunctions}
\label{subsec:courant-sharp}

The next theorem, like Theorem~\ref{thm:min-nodal}, was proved in \cite{BanBerRaz12}, but here we wish to give a fundamentally different proof, based on a continuity argument involving the introduction of Robin-type conditions at the cut set of the partition; we imagine that such ideas might also be adaptable to the corresponding problem on domains.

\begin{theorem} \label{thm:min-cs}
Given a positive integer $k$, suppose there exists a $k$-partition of $\Graph$ realising $\doptenergy[k,\infty](\Graph)$ which is proper and nodal. Then $\doptenergy[k,\infty](\Graph)=\mu_k(\Graph)$.
\end{theorem}

\begin{proof}
Let $\parti$ be a nodal, minimal and proper (thus necessarily exhaustive) $k$-partition, and let us denote the corresponding eigenvalue by $\mu$, so that $\mu=\doptenergy[k,\infty](\Graph)$. Let $\psi$ be an eigenfunction associated with $\mu$. We now define a family of vertex conditions in the following way; we will denote the corresponding (quantum) graphs by $(\Graph_\theta)_{\theta\in[0,\pi/2]}$: at each cut point $v$ of the partition $\parti$, which by assumption is a vertex of degree two, we impose the Robin-type vertex condition
\begin{equation}
\label{eq:theta-robin}
	\cos\theta\left(f'(v_-)+f'(v_+)\right)=(\sin\theta) f(v),
\end{equation}
with $f$ additionally required to be continuous. (In fact, the condition \eqref{eq:theta-robin} corresponds to putting a $\delta$-potential of strength $\tan\theta$ at $v$, for $\theta \in [0,\pi/2)$, and imposing a Dirichlet condition if $\theta = \pi/2$.)
For all integers $j\ge1$, we write $\mu_j(\theta):=\mu_j(\Graph_\theta)$. Note that when $\theta=0$, \eqref{eq:theta-robin} reduces to the usual Kirchhoff condition, so that $\mu_j (0)=\mu_j(\Graph)$ for all $j\geq 1$. We then have the following further basic properties.

\begin{lemma}
\label{lem:family}
The family of graphs $(\Graph_\theta)_{\theta\in[0,\pi/2]}$ has the following properties.
\begin{enumerate}
\item For all integers $j\ge1$, the function $\theta\mapsto \mu_j(\theta)$ is continuous and non-decreasing on $[0,\pi/2]$.
\item For all integers $1\le j\le k$, $\mu_j(\pi/2)=\mu$, and $\mu_{k+1}(\pi/2)>\mu$.
\item For all $\theta\in [0,\pi/2]$, $\mu$ is an eigenvalue of the graph $\Graph_\theta$, with $\psi$ an associated eigenfunction.
\end{enumerate}
\end{lemma}

We will not give a detailed proof: for part (1), cf.\ \cite[Section~3.17]{BerKuc13} or \cite[Theorem~3.4]{BeKeKuMu18}; for (2), note that if we treat Dirichlet points as cut points, then $\Graph_{\pi/2}$ has exactly $k$ connected components, corresponding to the clusters of $\parti$, each having the same eigenvalue $\mu$ with multiplicity one; for (3), we simply note that each $v$ is a zero of $\psi$, so that the vertex condition \eqref{eq:theta-robin} reduces to the usual Kirchhoff condition.

We denote by $U$ the subspace of the eigenspace of $\Graph = \Graph_0$ associated with $\mu$, whose functions vanish at each cut point of $\parti$. We have $\psi \in U$, but $U$ could contain functions which are not proportional to $\psi$, in particular if $\mu$ has multiplicity greater than one.\footnote{An example for such a $\Graph$ is the equilateral $3$-star with edges of length $1$ each, where we take $\mu = \pi^2/4$ corresponding to $\mu_2 (\Graph) = \mu_3 (\Graph)$ and we consider any partition whose cut set is the central vertex of $\Graph$.} For all $\theta\in[0,\pi/2]$, $U$ is contained in the eigenspace of the graph $\Graph_\theta$ associated with the eigenvalue $\mu$. Conversely, if $\varphi$ is an eigenfunction of the graph $\Graph_\theta$ for some $\theta\in[0,\pi/2]$ and if $\varphi$ vanishes at each cut point of $\parti$, then $\varphi\in U$. This means that $U$ can alternatively be described as the intersection of all the eigenspaces of $\Graph_\theta$, associated with $\mu$, for all $\theta\in[0,\pi/2]$. 

Let us denote by $\ell$ the smallest positive integer such that $\mu=\mu_\ell(\Graph)$; it follows from Proposition~\ref{prop:inequalities-k} and the assumption $\mu = \doptenergy[k,\infty](\Graph)$ that $\ell \geq k$. Let us assume for a contradiction that $\ell>k$, which is equivalent to $\mu_k(0)<\mu$. 

\begin{lemma}
\label{lem:contra}
Under this assumption, there exists a smallest $\bar\theta:=\theta\in(0,\pi/2)$ such that $\mu_k(\theta)=\mu$, and an eigenfunction $\varphi$ of $\Graph_{\bar\theta}$ associated with $\mu$, which does not belong to $U$.
\end{lemma}

\begin{proof}[Proof of Lemma \ref{lem:contra}]
According to property (1) of Lemma~\ref{lem:family}, $\mu_k(\theta)\to \mu_k(\pi/2)=\mu$ and $\mu_{k+1}(\theta) \to \mu_{k+1}(\pi/2)$ as $\theta \to \pi/2$. Suppose that $\mu_k(\theta)<\mu$ for all $\theta\in (0,\pi/2)$. Then since $\mu$ is also an eigenvalue for each $\theta$ by (2), that is, for each $\theta$ there exists some $\ell(\theta)$, necessarily no smaller than $k+1$, such that $\mu = \mu_{\ell(\theta)} (\theta)$, by continuity we would also have $\mu_{k+1}(\theta)\le \mu$, and thus, from (1), $\mu_{k+1}(\pi/2)\le \mu$. This contradicts (2), and we conclude that there exists $\theta\in(0,\pi/2)$ such that $\mu_k(\theta)=\mu$. We can then choose $\bar\theta$ to be the greatest lower bound of all such $\theta$.

For each $\theta\in[0,\bar\theta)$, we pick an eigenfunction $\varphi_\theta$ of $\Graph_\theta$ associated with $\mu_k(\theta)$, normalised to have $L^2$-norm one. We claim that $(\varphi_\theta)_\theta$ admits a convergent subsequence as $\theta \to \bar \theta$. Indeed, since each function satisfies $-\varphi_\theta'' = \mu_k(\theta)\varphi_\theta$ edgewise, and $(\mu_k(\theta))_\theta$ is bounded in $\R$ and $(\varphi_\theta)_\theta$ is bounded in $L^2(\Graph)$, there exists a subsequence converging weakly in $H^2$ and hence strongly in $H^1$ -- in fact in $C^1$ -- on each edge to a function $\varphi$. From the Rayleigh quotients, possibly up to a further subsequence (which we will still simply denote by $\varphi_\theta$) we also see that the convergence is weak in $H^1 (\Graph)$. In particular, by compactness of the injection $H^1(\Graph) \hookrightarrow C(\Graph)$, the limit function $\varphi$ is continuous. The strong $C^1$-convergence on each edge implies that $\varphi$ also satisfies the condition \eqref{eq:theta-robin} at each vertex for $\theta = \bar\theta$. From the fact that
\begin{displaymath}
	\frac{\int_\Graph |\varphi'|^2\,\textrm{d}x}{\int_\Graph |\varphi|^2\,\textrm{d}x} = \lim_{\theta \to \bar\theta} \mu_k (\theta)
\end{displaymath}
and an inductive argument involving convergence of the eigenfunctions of the lower eigenfunctions, we can finally conclude that $\varphi$ is in fact an eigenfunction of $\Graph_{\bar\theta}$ associated with $\mu_k (\bar\theta)$, and in particular also with $\mu = \mu_k (\bar\theta)$.

Finally, for each $\theta < \bar\theta$, since $\varphi_\theta$ is associated with $\mu_k (\theta) < \mu$ and $U$ is contained in the eigenspace associated with the eigenvalue $\mu$, we have that $\varphi_\theta \not\in U$, and in fact $\varphi_\theta$ is orthogonal to each function of $U$. Strong convergence $\varphi_\theta \to \varphi$ in $L^2(\Graph)$ now implies that $\varphi$ is orthogonal to every function in $U$.
\end{proof}

Now, since $\varphi\notin U$, there exists at least one cut point of $\parti$ where $\varphi$ does not vanish, which we denote by $v_0$. For $\varepsilon>0$ small enough,  the function $\psi_\varepsilon:=\psi+\varepsilon \varphi$ has exactly one zero close to each cut point of $\parti$, and does not vanish elsewhere. Therefore, the nodal partition associated with $\psi_\varepsilon$, which we denote by $\parti_\varepsilon$, is a proper $k$-partition. Let us denote the clusters of $\parti_\varepsilon$ by $\Graph_i^{(\varepsilon)}$, with $i\in\{1,\dots,k\}$, and the restriction of $\psi_\varepsilon$ to $\Graph_i^{(\varepsilon)}$ by $\psi_i^{(\varepsilon)}$. There is one cluster $\Graph_{i_0}^{(\varepsilon)}$ whose interior contains the cut point $v_0$.

Let us now consider a general cluster $\Graph_i^{(\varepsilon)}$ and denote by $\{v_1,\dots,v_\ell\}$ the cut points contained in its interior (it is possible that there is no such cut point, in which case $\ell=0$ and this set is empty). We have
\[
\begin{split}
	\int_{\Graph_i^{(\varepsilon)}}\left\vert \left(\psi_i^{(\varepsilon)}\right)'\right\vert^2\,dx&=\mu\int_{\Graph_i^{(\varepsilon)}}\left\vert \psi_i^{(\varepsilon)}\right\vert^2\,dx-\tan\theta\sum\limits_{j=1}^\ell \left(\psi_i^{(\varepsilon)}(v_j)\right)^2\\
	&=\mu\int_{\Graph_i^{(\varepsilon)}}\left\vert \psi_i^{(\varepsilon)}\right\vert^2\,dx-\varepsilon\tan\theta\sum\limits_{j=1}^\ell \varphi(v_j)^2\le\mu\int_{\Graph_i^{(\varepsilon)}}\left\vert \psi_i^{(\varepsilon)}\right\vert^2\,dx. 
\end{split} 
\]
By the variational characterisation of the eigenvalues, $\lambda_1(\Graph_i^{(\varepsilon)})\le\mu$. In the particular case of the cluster $\Graph^{(\varepsilon)}_{i_0}$, the sum on the right-hand side is strictly positive, and therefore $\lambda_1(\Graph_{i_0}^{(\varepsilon)})<\mu$. We obtain $\denergy[\infty] (\parti_\varepsilon)\le \denergy[\infty] (\parti)$, with $\lambda_1(\Graph_{i_0}^{(\varepsilon)})<\denergy[\infty] (\parti)$. We have reached a contradiction: either $\lambda_1(\Graph_{i}^{(\varepsilon)})<\denergy[\infty] (\parti)$ for all $1\le i \le k$, in which case $\parti$ is not minimal, or  $\lambda_1(\Graph_{i}^{(\varepsilon)})=\denergy[\infty] (\parti)$ for some $1\le i \le k$, in which case $\parti_\varepsilon$ is minimal without being an equipartition, contradicting Theorem \ref{thm:min-nodal}.
\end{proof}

\subsection{Minimal partitions for non-Courant-sharp eigenfunctions}\label{sec:knok}
We now consider an example showing that for $k\ge3$, a nodal partition which achieves $\doptenergy[k,\infty](\Graph)$ is not necessarily given by an eigenfunction associated with $\mu_k(\Graph)$. This is in contrast to the situation for domains in $\mathbb R^2$ (see \cite[Theorem~1.17]{HelHofTer09}).

We will consider a \emph{pumpkin graph} on three edges, of length $\pi$, $2\pi$ and $2\pi$, respectively; in other words, we consider a graph $\mathcal{H}$ with two vertices $\{v,w\}$ and three edges $\{e_1,e_2,e_3\}$. We set $e_1=[x_1,x_1+\pi]$, $e_2=[x_2,x_2+2\pi]$ and $e_3=[x_3,x_3+2\pi]$ (see Figure \ref{fig:graph-d3}).

\begin{figure}[htb]
\begin{minipage}[c]{4cm}
\begin{tikzpicture}
\coordinate (g) at (0,0);
\coordinate (h) at (2,0);
\coordinate (i) at (1,1.3);
\coordinate (j) at (1,-1.3);

\draw[fill] (g) circle (2pt);
\draw[fill] (h) circle (2pt);
\draw[fill] (i)  circle (2pt);
\draw[fill] (j) circle (2pt);

\draw[thick] (g) [bend left=45] to (i);
\draw[thick] (g) [bend left=-45] to (j);
\draw[thick] (g) to (h);
\draw[thick] (i) [bend left=45] to (h);
\draw[thick] (j) [bend left=-45] to (h);

\node at (g) [anchor=east] {$v$};
\node at (h) [anchor=west] {$w$};
\node at (1,0) [anchor=south] {$e_1$};
\node at (1.8,0.8) [anchor=west] {$e_2$};
\node at (1.8,-0.8) [anchor=west] {$e_3$};
\end{tikzpicture}
\end{minipage}
\begin{minipage}[c]{1.5cm}
$ $
\end{minipage}
\begin{minipage}[c]{4cm}
\begin{tikzpicture}
\coordinate (k) at (0,0);
\coordinate (l) at (2,0);
\coordinate (m) at (1,1.3);
\coordinate (n) at (1,-1.3);

\draw[fill] (1,0) circle (2pt);
\draw[fill] (0.28,0.9) circle (2pt);
\draw[fill] (1.72,0.9) circle (2pt);
\draw[fill] (0.28,-0.9) circle (2pt);
\draw[fill] (1.72,-0.9) circle (2pt);

\draw[thick] (k) [bend left=45] to (m);
\draw[thick] (k) [bend left=-45] to (n);
\draw[thick] (k) to (l);
\draw[thick] (m) [bend left=45] to (l);
\draw[thick] (n) [bend left=-45] to (l);
\end{tikzpicture}
\end{minipage}
\caption{The pumpkin graph $\mathcal{H}$ after insertion of two dummy vertices (top and bottom, corresponding to the points $x_2+\pi$ and $x_3+\pi$, respectively) to make it equilateral (left); an optimal $4$-partition of $\mathcal{H}$, where the thick black dots denote the partition cut set, which divides $\mathcal D_3$ into two equilateral $3$-stars whose edges all have length $\pi/2$, and two intervals of length $\pi$ each (right).}
\label{fig:graph-d3}
\end{figure}
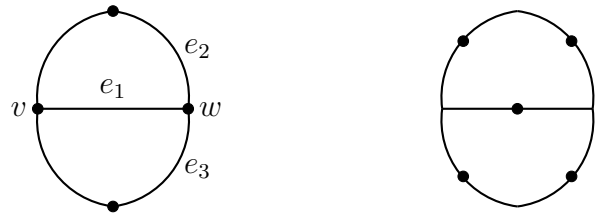

Because the edges of $\mathcal{H}$ have rationally dependent lengths, we can easily exploit von Below's formula to compute the eigenvalues of the Laplacian on $\mathcal{H}$. More precisely, insertion of two dummy vertices turns $\mathcal{H}$ into an equilateral metric graph on four vertices and five edges, each of length $\pi$. The underlying discrete graph has transition matrix
\begin{displaymath}
{\mathcal T}=
\begin{pmatrix}
0 & -\frac13 & -\frac13 & -\frac13\\
-\frac12 & 0 & 0 & -\frac12\\
-\frac12 & 0 & 1 & -\frac12\\
-\frac13 & -\frac13 & -\frac13 & 1
\end{pmatrix}
\end{displaymath}
whose eigenvalues are $\mu_j = 1, 0,-\frac{1}{3},-\frac{2}{3}$. In view of~\cite[Theorem in \S~5]{Bel85} and by rescaling, the eigenvalues of $\mathcal{H}$ are 
\begin{itemize}
\item 0 (with multiplicity 1, since $\mathcal D_3$ is connected);
\item the infinitely many values attained by $ \left( \frac{1}{\pi}\arccos \mu_j \right)^2$,  $\mu_j \neq 1$, from
 the three non-trivial eigenvalues of $\mathcal T$;
\item $k^2$ with multiplicity $3$ for even $k$;
\item $k^2$ with multiplicity $1$ for odd $k$.
\end{itemize}
In particular, the fifth-lowest eigenvalue of the Laplacian on $\mathcal{H}$ with natural vertex conditions is 1 and is simple.
Moreover, one can check directly that its eigenfunction vanishes identically on the edge $e_1$, and at the four vertices marked in Figure~\ref{fig:graph-d3}-left, for a total of four nodal domains. The following lemma summarises these statements.

\begin{lemma}
\label{lem:ev-3M}
The eigenvalue $\mu_5(\mathcal{H})$ is simple and equal to $1$. In addition, if $\psi$ is an associated eigenfunction, the nodal (non-exhaustive) $4$-partition associated with $\psi$ is, in the notation introduced above (see Figure~\ref{fig:graph-d3}):
\begin{displaymath}
\mathcal N_4:=\{[x_2,x_2+\pi],[x_2+\pi,x_2+2\pi],[x_3,x_3+\pi],[x_3+\pi,x_3+2\pi]\}.
\end{displaymath}
\end{lemma}

\begin{proposition}
\label{prop:min-part-3M}
We have $\doptenergy[4,\infty](\mathcal{H})=\mu_5(\mathcal{H})=1$.
\end{proposition}

An optimal (exhaustive, internally connected, equilateral) partition corresponding to $\doptenergy[4,\infty](\mathcal{H})$ is depicted in Figure~\ref{fig:graph-d3} (right). Note that this also equals $\doptenergy[5,\infty](\mathcal{H})$, which is realised by the cut set indicated in Figure~\ref{fig:graph-d3} (left). Our proof will show that any minimal partition must have energy $1$. In fact, with a little more effort, one could show that even when minimising among all not necessarily exhaustive partitions, the minimal energy is still $1$, even though the set of minimal partitions is larger (as Figure~\ref{fig:graph-d3} demonstrates). In Section~\ref{sec:exhaustive-issue} we compare exhaustive and non-exhaustive Dirichlet minimal partitions, and refer in particular to Conjecture~\ref{conj:non-admissible-exhaustion}.

\begin{proof}[Proof of Proposition~\ref{prop:min-part-3M}]
Let $\parti=\{\Graph_1,\Graph_2,\Graph_3,\Graph_4\}$ be a partition achieving $\doptenergy[4,\infty] (\mathcal{H})$. It follows from Lemma \ref{lem:ev-3M} and Remark~\ref{rem:dirichlet-k-mon} that $\denergy[\infty](\parti)=\doptenergy[4,\infty] (\mathcal{H})\le 1$. In the rest of the proof, for $i\in\{1,2,3,4\}$ and $j\in\{1,2,3\}$, we say that the cluster $\Graph_i$ \emph{is contained} (resp.\ \emph{strictly contained}) in $e_j$ when $\Graph_i\subset e_j$ (resp.\ $\Graph_i\subsetneq e_j$), and we say that $\Graph_i$ \emph{meets} $e_j$ when $\interior \Graph_i \cap\interior e_j \neq\emptyset$. (Here we emphasise that we are identifying $\Graph_i$ with $\Omega_i \subset \Graph$.) Since the energy of $\parti$ is at most $1$, we see easily that $e_1$ cannot strictly contain a cluster of $\parti$; in particular any cluster contained in $e_1$ has to be equal to $e_1$. Similarly, if  $e_2$ contains two clusters, they are both intervals of length $\pi$ and their union is equal to $e_2$. The same holds for $e_3$. 

Let us now discuss the possible cases. Since our partition is by assumption exhaustive, at least one cluster of $\parti$ meets $e_1$. Up to relabelling of the supports and without loss of generality, we can assume that $\Graph_1$ meets $e_1$. Let us first consider the case where $\Graph_1$ is contained in $e_1$. Then $\Graph_1=e_1$, and $\{v,w\}$ is contained in the cut set of $\parti$. Each of the other clusters is thus contained either in $e_2$ or in $e_3$. From the preliminary remarks, there is, up to relabelling of the edges and the support, only one possiblity: $e_2$ is the union of $\Graph_2$ and $\Graph_3$, which are both intervals of length $\pi$, and $\Graph_4 = e_3$. This gives energy exactly $1$.

Up to relabelling of the clusters, there thus remains only one case: $\Graph_1$ meets $e_1$ without being contained in $e_1$. Then one of the vertices $\{v,w\}$ is contained in $\interior \Graph_1$. Without loss of generality, we can assume $v\in \interior \Graph_1$. If no other cluster meets $e_1$, then either $e_2$ or $e_3$ must strictly contain two clusters, in contradiction to the preliminary remarks. Without loss of generality we can assume that $\Graph_2$ meets $e_1$, and therefore that $w\in \interior \Graph_2$. Since $\parti$ is rigid, it follows that $\Graph_3$ and $\Graph_4$ are intervals of length at least $\pi$ contained in $\interior e_3$ and $\interior e_4$, respectively, and $\{\Graph_1,\Graph_2\}$ is a partition of $\mathcal{H}\setminus \Graph_3\cup\Graph_4$,  whose cut set consists of a single point in $e_1$. We construct a new graph $\widetilde {\mathcal{H}}$  from $\mathcal{H}\setminus \Graph_3\cup\Graph_4$ by gluing together the two degree one vertices corresponding to the two extremities of $\Graph_3$, and likewise gluing together the two vertices corresponding to the extremities of $\Graph_4$. To be more explicit, let us set $\Graph_3=[y_3,z_3]$ and $\Graph_4=[y_4,z_4]$, with $x_2<y_3<y_3+\pi\le z_3<x_2+2\pi$ and $x_3<y_4<y_4+\pi\le z_4<x_3+2\pi$. The graph $\widetilde {\mathcal{H}}$ then has three edges $\widetilde e_1:=e_1$, $\widetilde e_2=[x_2,x_2+2\pi-z_3+y_3]$ and $\widetilde e_2=[x_3,x_3+2\pi-z_4+y_4]$; and two vertices, $\widetilde v:=v$ and $\widetilde w=\{x_1+\pi,x_2+2\pi-z_3+y_3,x_2+2\pi-z_3+y_3,2\pi-z_4+y_4\}$. It is a $3$-pumpkin graph with edges of length at most $\pi$ each and $\widetilde\parti:=\{\Graph_1,\Graph_2\}$ is a (internally connected) $2$-partition of $\widetilde {\mathcal{H}}$. From Proposition \ref{prop:link-to-second-eigenvalue}, 
\begin{displaymath}
	\mu_2(\widetilde {\mathcal{H}})= \doptenergy[2,\infty](\mathcal{H})\le \denergy[\infty](\widetilde\parti)\le1.
\end{displaymath}
By monotonicity of the eigenvalues with respect to edge length, see \cite[Corollary~3.12(1)]{BeKeKuMu18}, we have, for any integer $j\ge1$, $1 = \mu_j(\mathcal{H}^*)\le \mu_j(\widetilde{\mathcal{H}})$, where $\mathcal{H}^*$ is the $3$-pumpkin graph with edges of length $\pi$. In particular, $\doptenergy[2,\infty](\widetilde{\mathcal{H}})=\mu_2(\widetilde{\mathcal{H}})\geq 1$. From this we conclude that $\denergy[\infty](\parti)=1$ in the case where $v \in \interior \Graph_1$ and $w \in \interior \Graph_2$. We have seen that, in all cases, we necessarily have $\denergy[\infty](\parti)=1$.
\end{proof}

\subsection{Non-bipartite minimal partitions}
\label{subsec:non-bipartite}

The goal of this section is to show that any proper minimal partition of a metric graph is the projection of a nodal partition on a double covering. To reach it, we have to give some additional definitions. We first recall the (standard) definition: given a metric graph $\Graph$ (the \emph{base graph}), a double covering is another metric graph $\widehat\Graph$ (the \emph{covering graph}) equipped with a surjective map $\Pi:\widehat\Graph\to\Graph$ (the \emph{covering map}). We require that
\begin{enumerate}
	\item $\Pi$ is locally an isometry; 
	\item for all $x\in \Graph$, $\Pi^{-1}(\{x\})$ contains two elements.
\end{enumerate}
Furthermore, given such a covering, we define the \emph{deck map} $\sigma:\widehat\Graph\to\widehat\Graph$ by $\sigma(y)\neq y$ and $\Pi(\sigma(y))=\Pi(y)$ for all $y\in\widehat\Graph$. The map $\sigma$ clearly satisfy $\sigma\circ\sigma=Id$ and is locally an isometry, and therefore is globally an isometry of $\widehat\Graph$.

We have the orthogonal decomposition 
\[L^2(\widehat\Graph)=\mathcal S(\widehat\Graph)\oplus\mathcal A(\widehat\Graph),\]
where $\mathcal S(\widehat\Graph)$ and $\mathcal A(\widehat\Graph)$ are the subspaces of functions $f\in L^2(\widehat\Graph)$ satisfying respectively $f\circ\sigma=f$ or $f\circ\sigma=-f$. This decomposition is preserved by $\widehat L$, the natural Laplacian on $\widehat\Graph$, in the following sense. For any sufficiently regular function $f$, $\widehat L(f\circ\sigma)=\widehat L(f)\circ\sigma$. It follows that the domain $\widehat{\mathcal D}$ of $\widehat L$ also has an orthogonal decomposition:
\[\widehat{\mathcal D}=(\widehat{\mathcal D}\cap\mathcal S(\widehat\Graph))\oplus(\widehat{\mathcal D}\cap\mathcal A(\widehat\Graph)).\]
Accordingly, we define the operator $\widehat L_a$ as the restrictions of $\widehat L$ to  $\widehat{\mathcal D}\cap\mathcal A(\widehat\Graph)$. It is self-adjoint with compact resolvent in the Hilbert space $\mathcal A(\widehat\Graph)$. Its spectrum therefore consists of a sequence of eigenvalues with finite multiplicity which we denote by $(\mu_j^a(\widehat\Graph))_{j\ge1}$ (counting multiplicities).

Let us now consider a partition $\parti$ of $\Graph$. For each cluster $\Graph_i$, $\Pi^{-1}(\Graph_i)$ is a closed set having at most two connected components. The collection of all these connected components when $i$ runs over all the possible values is a partition of $\widehat\Graph$, once we have given it an (arbitrary) indexation. We denote it by $\widehat \parti$ and call it the partition of $\widehat\Graph$ \emph{lifted from $\parti$}.

\begin{theorem} \label{thm:covering} Let $\parti $ be a proper Dirichlet minimal non-bipartite $k$-partition of $\Graph$. 
\begin{enumerate}
\item There exists a double covering $\widehat\Graph$ such that the lifted partition  $\widehat \parti$ is the nodal partition of an  eigenfunction of $\widehat L_a$.  
\item For \emph{any} such double covering, $\Lambda_\infty^D (\parti)=\mu^a_k(\widehat\Graph)$.
\end{enumerate}
\end{theorem}

\begin{proof}
Most of the work goes into proving (1). We follow closely the argument in the proof of \cite[Theorem~2.10(1)]{BanBerRaz12} (see Section 4.1 of that reference). Since $\parti$ is Dirichlet minimal, it is a critical point for the functional $\Lambda$, with the parametrisation described by \cite[Theorem~2.8]{BanBerRaz12}.   More precisely, this means that the following holds, according to \cite[Section~4.1]{BanBerRaz12}.

We choose a set  of edges $\{e_j\,;\,1\le j\le \beta\}$ whose removal turns $\Graph$ into a tree. We choose (arbitrarily) a point $v_j$ in the interior of each $e_j$. These are called \emph{section points} in \cite{BanBerRaz12}. We split each $v_j$ into two vertices $v_j^-$ and $v_j^+$ according to the orientation of $e_j$,   $v_j^-$ and $v_j^+$ being the end of the left and right part of $e_j$ respectively. We obtain a metric tree  $\Tree$.
For any $(\varphi_1,\dots,\varphi_\beta)\in (-\pi,\pi]^\beta$, we define the self-adjoint operator on $\Tree$ acting as the (opposite of) the second derivative, and whose domain consistes of the functions $f\in \oplus_{e\in\EdgeSet(\Tree)}H^2(e)$ which satisfy the standard boundary conditions at the vertices of $\Graph$ and, for all $j\in\{1,\dots,\beta\}$,
\begin{equation*}
	\begin{cases}
			\cos\left(\varphi_j/2\right) f'(v_j^-)&=-\sin\left(\varphi_j/2\right) f(v_j^-),\\
			\cos\left(\varphi_j/2\right) f'(v_j^+)&=\sin\left(\varphi_j/2\right) f(v_j^+).
	\end{cases}
\end{equation*}
In this way, we have defined a quantum graph (i.e. triple of metric graph, differential expression and vertex conditions) that we denote by $\Tree_{(\varphi_1,\dots,\varphi_\beta)}$. According to \cite[Theorem 2.8]{BanBerRaz12}, there exists $(\overline\varphi_1,\dots,\overline\varphi_\beta)\in (-\pi,\pi]^\beta$ such that $\parti$ is the nodal partition of an eigenfunction $\psi$ of $\Tree_{(\overline\varphi_1,\dots,\overline\varphi_\beta)}$ associated with a simple eigenvalue. Furthermore, according to \cite[Section 4.1]{BanBerRaz12}, $|\psi(v_j^+)|=|\psi(v_j^-)|$ and  $|\psi'(v_j^+)|=|\psi'(v_j^-)|$ for all $j\in\{1,\dots,\beta\}$.

For each $j\in\{1,\dots,\beta\}$, we denote by $C_j$ the shortest path in the metric graph $\Tree$ which connects the vertices $v_j^-$ and $v_j^+$.  We have the following alternative.
\begin{enumerate}[(A)]
 \item $\psi(v_j^\pm)\neq0$ and $C_j$ contains an even number of zeros of $\psi$ \emph{or}  $\psi(v_j^\pm)=0$ and $C_j$ contains an odd number of zeros of $\psi$. In that case, $\psi(v_j^+)=\psi(v_j^-)$ and  $\psi'(v_j^+)=-\psi'(v_j^-)$. 
 \item $\psi(v_j^\pm)\neq0$ and $C_j$ contains an odd number of zeros of $\psi$ \emph{or}  $\psi(v_j^\pm)=0$ and $C_j$ contains an even number of zeros of $\psi$. In that case, $\psi(v_j^+)=-\psi(v_j^-)$ and  $\psi'(v_j^+)=\psi'(v_j^-)$.
 \end{enumerate}
We construct the double cover $\widehat\Graph$ by gluing two copies of $\widetilde\Graph$, denoted by $\Tree_u$ and $\Tree_d$, according to the following rules. For each $j\in\{1,\dots,\beta\}$, we denote by $v^-_{j,u}$, $v^+_{j,u}$, $v^-_{j,d}$ and $v^+_{j,d}$ the vertices corresponding to $v^-_j$ and  $v^+_j$ in  $\Tree_u$ and $\Tree_d$ respectively. 
\begin{enumerate}
\item In Case (A), we identify back $v^-_1$ with $v^+_1$ and $v^-_2$ with $v^+_2$, that is we glue back the cut points in each copy separately.
\item In Case (B), we identify $v^-_1$ with $v^+_2$ and $v^-_2$ with $v^+_1$.
\end{enumerate}
The covering map $\Pi:\widehat\Graph\to\Graph$ is defined as the unique continuous extension of the map sending each point in $\Tree_u$ or $\Tree_d$, distinct from the selected section points, to the corresponding point in $\Graph$. This map is a local isometry and each of its fibres has two elements.  
 
We then define the function $\widehat\psi$ as the continuous extension to $\widehat\Graph$ of the function which is equal to $\psi$ on $\Tree_u$ and $-\psi$ on $\Tree_d$ (this continuous extension exists by construction of $\widehat\Graph$). By construction, $\widehat\psi$ is an antisymmetric eigenfunction of standard Laplacian on $\widehat\Graph$, associated with the eigenvalue $\Lambda_\infty^D (\parti)$. Furthermore, its nodal partition is $\widehat\parti$. This conclude the proof of Part (1).

In order to prove Part (2), we repeat the proof of Theorem \ref{thm:min-cs}, replacing the Hilbert space $L^2(\Graph)$ with $\mathcal A(\widehat\Graph)$, and more generally all the function spaces on $\Graph$ by the corresponding antisymmetric function spaces on $\widehat\Graph$.
\end{proof}

For readers familiar with magnetic Schr\"odinger operators on metric graphs, we mention that the previous construction has a magnetic interpretation. Indeed, $\widehat L_a$ is unitarily equivalent to a magnetic Laplacian, by which we mean a magnetic Schr\"odinger operator with zero potential (see \cite[Section~2.6]{BerKuc13} for the relevant definitions). According to \cite[Corollary~2.6.3]{BerKuc13}, such an operator is defined, up to unitary equivalence, by specifying the magnetic flux, modulo $2\pi$, through each cycle of $\Graph$. For this, we use the following rule. We count the number of times the cycle crosses a cut-point of $\parti$. The flux is $\pi$ if this number is odd and $0$ if it is even. Although we will not go further into details, let us also mention that for such a magnetic Laplacian, $\parti$ is the nodal partition of an eigenfunction associated with the $k$-th eigenvalue.

\subsection{Comparison with two-dimensional domains}\label{sec:domgra}

In order to put the results of this section into perspective, let us compare them with those previously obtained for minimal partitions of two-dimensional domains. We first briefly review this last theory, as developed in \cite{HelHofTer09}. The interested reader can find a more extensive survey (including numerical results) and a much more complete bibliography in \cite{BNHe17}. 

Following \cite{HelHofTer09}, we take $\Omega\subset \R^2$ to be open, bounded and connected, with a piecewise $C^{1,\alpha}$ boundary for some $\alpha>0$ and call such a set a \emph{domain}. In this subsection only, we call a \emph{$k$-partition} of $\Omega$ a family $\parti=\{\Omega_1,\dots,\Omega_k\}$ of $k$ subsets of $\Omega$ which are mutually disjoint, connected and open. Comparing this definition with Section \ref{sec:partitions}, we see that partitions are defined in \cite{HelHofTer09} by their cluster supports, taken to be open sets. The following example is particularly important. Let $\psi$ be an eigenfunction of $-\Delta_\Omega$, the Dirichlet realisation of the Laplacian in $\Omega$. The \emph{nodal partition} associated with $\psi$ is the family of all the nodal domains of $\psi$, that is the connected components of the complement of its zero set. This is analogous to Definition~\ref{def:nodal-partition}. 

We denote by $\mathfrak{P}_k(\Omega)$ the set of all $k$-partitions of $\Omega$ in the above sense. The following definitions are in complete analogy with Section \ref{sec:ex-min}. We first set, for a given partition $\parti$,
\begin{equation}
\label{eq:dom-energy}
	\denergy[p] (\parti) := \begin{cases} \left(\frac{1}{k}\sum\limits_{i=1}^k \lambda_1(\Omega_i)^p\right)^{1/p}
	\qquad &\text{if } p \in  (0,\infty),\\ \max\limits_{i=1,\ldots,k} \lambda_1(\Omega_i) 
	\qquad &\text{if } p = \infty, \end{cases}
\end{equation}
where $\lambda_1(\Omega_i)$ denotes the first eigenvalue of $-\Delta_{\Omega_i}$. We then define
\begin{equation}
 \label{eq:dom-minimal-energy}
    \doptenergy[k,p](\Omega) :=\inf_{\parti \in \mathfrak{P}_k(\Omega)} \denergy[p] (\parti)
\end{equation}
and call \emph{minimal} any $k$-partition $\parti$ satisfying $\denergy[p] (\parti)=\doptenergy[k,p](\Omega)$.
It was proved in \cite{HelHofTer09}, building on results from \cite{BuBuHe98,ConTerVer05,CaLi07}, that minimal partitions exist for any positive integer $k$ and any $p\in[1,\infty]$, and are very regular. In the terminology of \cite{HelHofTer09}, they are \emph{strong}, meaning that 
\begin{equation*}
 \overline{\Omega}=\bigcup\limits_{i=1}^k\overline{\Omega_i};
\end{equation*}
in our terminology such partitions would be called exhaustive. It is natural to define the \emph{boundary} of a strong (i.e., exhaustive) partition by 
\begin{equation*}
 \mathcal N(\parti):=\overline{\bigcup\limits_{i=1}^k\partial \Omega_i\cap\Omega}.
\end{equation*}
When $\parti$ is minimal, the set $\mathcal N(\parti)$ enjoys regularity properties which make it analogous to the nodal set of an eigenfunction of $-\Delta_\Omega$. In particular, it consists of a finite number of regular curves. More details can be found in \cite[Theorem 1.12]{HelHofTer09} or \cite[Theorem 10.43]{BNHe17}. We additionally define what it means for two distinct cluster supports $\Omega_i$ and $\Omega_j$ to be neighbours: the interior of the set $\overline{\Omega_i\cup \Omega_j}\cap\Omega$ is connected. This is the analogue of Definition \ref{def:neighbours} in the case of domains. We then say that the partition $\parti$ is \emph{bipartite} if we can colour its cluster supports, using only two colours, in such a way that two neighbours have  different colours. This is the analogue of Definition \ref{def:proximity}, although in the case of domains we do not have to restrict ourselves to a special class of partitions, such as proper partitions for graphs. As pointed out in \cite{HelHofTer09}, a nodal partition is strong and bipartite.

In the case $p=\infty$, which we will assume for the rest of this section, the results in \cite{HelHofTer09} point to a clearer connection between minimal and nodal partitions for domains than for graphs. Indeed, \cite{HelHofTer09} establishes the following result.

\begin{theorem}[Theorems 1.14 and 1.17 in \cite{HelHofTer09}] 
\label{thm:dom-minimal-nodal}
 Let $\parti$ be a minimal $k$-partition (for some positive integer $k$) realising $\doptenergy[k,\infty](\Omega)$. If $\parti$ is bipartite, then it is nodal; more precisely, it is the nodal partition for an eigenfunction associated with $\lambda_k(\Omega)$, the $k$-th eigenvalue of $-\Delta_\Omega$.
\end{theorem}

Let us recall that if $\psi$ is an eigenfunction associated with $\lambda_k(\Omega)$ and $\nu(\psi)$ its number of nodal domains, Courant's Nodal Theorem holds in its strong form and states that $\nu(\psi)\le k$. Following \cite{HelHofTer09}, we say that $\psi$ is \emph{Courant sharp} if $\nu(\psi)=\kappa(\lambda_k(\Omega))$, where
\begin{equation*}
 \kappa(\lambda_k(\Omega))=\min\{\ell\in \N^*\,;\,\lambda_\ell(\Omega)=\lambda_k(\Omega)\}.
\end{equation*}
The second statement in Theorem \ref{thm:dom-minimal-nodal} then tells us that if a minimal $k$-partition $\parti$ is nodal, $\lambda_{k-1}(\Omega)<\lambda_k(\Omega)$ and $\parti$ is given by a Courant-sharp eigenfunction associated with $\lambda_k(\Omega)$.

As seen in this section, the results connecting nodal and spectral minimal partitions are less sharp for metric graphs than for domains, except when the partitions considered are proper (see point (5) of Definition \ref{def:classification}). Indeed, even when we restrict ourselves to nodal partitions, only the weak form of Courant's Nodal Domain Theorem (Proposition \ref{prop:weak-courant}) holds on graphs in general. However (using the notation of this Proposition \ref{prop:weak-courant}) the stronger inequality $\nu(\psi)\le k$ holds when the eigenvalue is simple or the nodal partition associated with $\psi$ is proper.

Similarly, it seems unclear how to give an appropriate definition for a bipartite partition of a metric graph that would allow us to transpose the first statement in Theorem \ref{thm:dom-minimal-nodal}. Nevertheless, this statement makes sense and is correct for proper minimal partitions. It is indeed formulated as Theorem \ref{thm:min-nodal} in this section. Theorem \ref{thm:nodal-tree} shows that considering only proper partitions is unnecessarily restrictive, but we do not have a comprehensive theory yet.

The analogue of the second statement in Theorem \ref{thm:dom-minimal-nodal} is given by Theorem \ref{thm:min-cs} for proper minimal partitions. However, there seems to be no natural way of extending the notion of bipartiteness, and hence the scope of Theorem~\ref{thm:min-nodal}, to non-proper partitions, as Proposition~\ref{prop:min-part-3M} exemplifies. Furthermore, as seen on Figure~\ref{fig:graph-d3}, one can also find proper minimal partitions of $\mathcal{H}$, although they are not nodal. The existence of such partitions therefore does not guarantee that $\doptenergy[k,\infty]=\mu_k$.

Finally, we point out that on domains there is a construction corresponding to the realisation of proper minimal partitions as  projection of nodal partitions in a double covering, described in Theorem \ref{thm:covering}. This correspondence appears more clearly when we consider the magnetic interpretation of this realisation, given at the end of Subsection \ref{subsec:non-bipartite}. As shown in \cite{HeHO13}, any minimal $k$-partition of a domain is a nodal partition associated with the $k$-th eigenvalue of a magnetic Schr\"odinger operator having a finite number (possibly zero) of Aharonov--Bohm singularities with magnetic fluxes equal to $\pi$, where the number of fluxes depends on $k$ in a rather complicated way.

\section{Existence of spectral maximal partitions}
\label{sec:ex-max}

It turns out that for some classes of partitions the problem of \emph{maximising} spectral quantities is also well defined: we define the energies
\begin{equation}
\label{eq:nminenergy}
	\nminenergy (\parti) := \min_{i=1,\ldots,k} \mu_2 (\Graph_i)
\end{equation}
and
\begin{equation}
\label{eq:dminenergy}
	\dminenergy (\parti) := \min_{i=1,\ldots,k} \lambda_1 (\Graph_i)
\end{equation}
for any exhaustive rigid $k$-partition $\parti \in \mathfrak{R}_k$, and thus the maximal natural and Dirichlet energies, respectively:
\begin{equation}
\label{max-min-energies}
\begin{aligned}
	\nmaxmin[k] = \nmaxmin[k] (\Graph) &:=\sup_{\parti \in \mathfrak{R}_k} \,\nminenergy (\parti),\\
	\dmaxmin[k] = \dmaxmin[k] (\Graph) &:=\sup_{\parti \in \mathfrak{R}_k} \,\dminenergy (\parti).
\end{aligned}
\end{equation}
Here it is important to restrict to exhaustive partitions, see Remark~\ref{rem:general-max-min}.
In the sequel we will prove similar properties of these to $\noptenergy[k,p]$ and $\doptenergy[k,p]$, in particular the existence of maximisers.

In Section~\ref{sec:examples} we will give examples comparing both the behaviour of the optimal partitions with respect to $p$, and comparing these notions of spectral extremal partition with the minimal partitions introduced in Section~\ref{sec:ex-min}; it should be profitable to have a more systematic understanding of the relations between them. In this section we treat the existence of partitions having the optimal energies for the max-min problems $\nmaxmin[k]$ and $\dmaxmin[k]$.

\begin{remark}
\label{rem:general-max-min}
The more general max-min problems
\begin{displaymath}
	\sup_{\parti \in A} \nminenergy (\parti), \qquad \sup_{\parti \in A} \dminenergy (\parti),
\end{displaymath}
unlike their min-max counterparts in Section~\ref{sec:ex-min}, are only well posed for rather particular choices of sets $A$ of partitions. For example, if we seek the optimum among non-exhaustive $k$-partitions, even among rigid partitions, then both suprema are clearly infinite: any sequence of partitions $\parti_n$ each of whose clusters has total length at most $1/n$, say, satisfies $\nminenergy (\parti_n), \dminenergy (\parti_n) \to \infty$ (just use Nicaise' inequalities).
\end{remark}

There is an analogue of Lemma~\ref{lem:dirichlet-min-max-rigid-loose}, but this time for the natural problem.

\begin{lemma}
\label{lem:neumann-max-min-rigid-loose}
For any graph $\Graph$ and any $k\geq 1$, we have
\begin{displaymath}
	\sup_{\parti \in \mathfrak{R}_k} \,\nminenergy (\parti) = \sup_{\parti \in \mathfrak{P}_k} \,\nminenergy (\parti).
\end{displaymath}
\end{lemma}

\begin{proof}
The proof is, to an extent, analogous to the proof of Lemma~\ref{lem:dirichlet-min-max-rigid-loose}, but here we have to prove ``$\geq$'', since the latter supremum is over a larger set. It suffices to prove that for an arbitrary exhaustive $\widetilde\parti = \{\NewGraph_1,\ldots,\NewGraph_k\} \in \mathfrak{P}_k$, for each $i$ there exists some $\Graph_i \in \rho_{\widetilde\Omega_i}$ (where $\widetilde\Omega_i \subset \Graph$ is the cluster support of $\NewGraph_i$) such that $\mu_2 (\Graph_i) \geq \mu_2 (\NewGraph_i)$, since then the exhaustive rigid partition $\parti := \{\Graph_1,\ldots,\Graph_k\} \in \mathfrak{R}_k$ satisfies $\nminenergy (\parti) \geq \nminenergy (\widetilde{\parti})$. To this end, we simply take $\Graph_i$ to be the unique faithful cluster in $\rho_{\widetilde{\Omega}_i}$; then by construction $\NewGraph_i$ may be obtained as a cut of $\Graph_i$. Standard surgery results (e.g., \cite[Theorem~3.4]{BeKeKuMu18}) now imply that $\mu_2 (\Graph_i) \geq \mu_2 (\NewGraph_i)$, as required.
\end{proof}

On the other hand, the conclusion of Remark~\ref{rem:dirichlet-subset-graph-distinction} also holds for $\dminenergy (\parti)$: for the Dirichlet problem, there is no difference between different rigid clusters associated with the same supports: in particular, maximising over all rigid partitions is equivalent to maximising over all faithful ones.

Before turning to the existence of maximising partitions, we first observe that $\nminenergy$ and $\dminenergy$ are continuous with respect to partition convergence, even in the degenerate cases.

\begin{lemma}
\label{lem:min-energy-convergence}
Suppose $\parti_n \in \mathfrak{R}_k$ are rigid $k$-partitions of $\Graph$, all similar to each other, and $\parti_n \to \parti_\infty$ as $n \to \infty$ in the sense of \eqref{eq:dist-pk}. Then also
\begin{displaymath}
	\nminenergy (\parti_n) \to \nminenergy (\parti_\infty) \quad \text{and} \quad \dminenergy (\parti_n) \to \dminenergy (\parti_\infty).
\end{displaymath}
\end{lemma}

\begin{proof}
If $\parti_\infty$ is itself a $k$-partition, then this follows immediately from Lemma~\ref{lem:eig-convergence}. So suppose it is not; then at least one cluster has total length converging to zero, and thus eigenvalues diverging to $\infty$, see again Lemma~\ref{lem:eig-convergence}. Consider the Dirichlet problem (the natural case is entirely analogous). Suppose without loss of generality that the clusters $\Graph_1^{(\infty)},\ldots,\Graph_j^{(\infty)}$, $1 \leq j<n$, give the minimum in $\dminenergy (\parti_\infty)$:
\begin{displaymath}
	\lambda_1 (\Graph_1^{(\infty)}) = \ldots = \lambda_1 (\Graph_j^{(\infty)}) = \dminenergy (\parti_\infty);
\end{displaymath}
note that this energy is finite since at least one cluster of $\parti_\infty$ has positive total length and thus a finite eigenvalue. But now each corresponding cluster $\Graph_i^{(n)}$ of $\parti_n$ converges to $\Graph_i^{(n)}$; in particular, $\lambda_1 (\Graph_i^{(n)}) \to \lambda_1 (\Graph_i^{(\infty)})$ for all $i=1,\ldots,j$, while $\liminf_{n\to\infty} \lambda_1 (\Graph_i^{(n)}) > \dminenergy (\parti_\infty)$ for all $i=j+1,\ldots,k$. It now follows from the definition of $\dminenergy$ as a minimum that indeed $\dminenergy (\parti_n) \to \dminenergy (\parti_\infty)$.
\end{proof}

\begin{theorem}
\label{thm:max-existence}
Fix a graph $\Graph$ and $k\geq 1$. Then there exist exhaustive rigid partitions $\parti^N = \parti^N (k)$ and $\parti^D = \parti^D (k)$ of $\Graph$ such that
\begin{displaymath}
	\nminenergy (\parti^N) = \nmaxmin[k] \quad \text{and} \quad \dminenergy (\parti^D) = \dmaxmin[k]
\end{displaymath}
and such that there exist $k$-partitions $\parti_n^N, \parti_n^D \in \mathfrak{R}_k$ with $\parti_n^N \to \parti^N$ and $\parti_n^D \to \parti^D$. For all $k\geq 1$, $\parti^D$ may be taken as a $k$-partition itself; moreover, there exists a constant $k_0 \geq 1$ possibly depending on $\Graph$ such that $\parti^N$ may be taken as a $k$-partition for all $k\geq k_0$. In particular, $\nmaxmin[k]$ and $\dmaxmin[k]$ are monotonically increasing functions of $k$ for all $k\geq 1$ and $k\geq k_0$, respectively.
\end{theorem}

Here the idea is to apply Theorem~\ref{thm:abstract-min-existence} to the functionals $\Lambda = -\nminenergy, -\dminenergy$ via condition (2). In the case of natural conditions, however, there is an additional difficulty with this condition; namely, it holds for all $1 \leq \ell \leq k$ if and only if Conjecture~\ref{conj:kirchhoff-two-cut} is true. If it is, then as we shall see we may choose $k_0=1$ in Theorem~\ref{thm:max-existence}.

\begin{proof}[Proof of Theorem~\ref{thm:max-existence}]
The strong lower semi-continuity condition of Theorem~\ref{thm:abstract-min-existence} follows from the (strong) continuity property established in Lemma~\ref{lem:min-energy-convergence}. Hence, if $\parti_n^N$ and $\parti_n^D$ are maximising sequences of exhaustive rigid $k$-partitions for $\nminenergy$ and $\dminenergy$, respectively, in both cases we obtain the existence of exhaustive rigid limit partitions $\parti^N$ and $\parti^D$. Although $\parti^N$ and $\parti^D$ may not be $k$-partitions, Lemma~\ref{lem:min-energy-convergence} ensures that their energies are equal to $\nmaxmin[k]$ and $\dmaxmin[k]$, respectively.

In the Dirichlet case, we verify (2) of Theorem~\ref{thm:abstract-min-existence} to establish that either $\parti^D$ is already a $k$-partition, or it may be replaced with a rigid $k$-partition whose energy is no smaller. In fact, suppose $\parti = \{\Graph_1,\ldots,\Graph_\ell\}$ is any $\ell$-partition, $\ell \geq 1$, and suppose the minimum in $\min_{i=1,\ldots,k} \lambda_1 (\Graph_i)$ is achieved by $\Graph_1$. We now modify $\Omega_1$, creating an $(\ell+1)$-st cluster support in such a way that $\lambda_1 (\Graph_1)$ is not decreased. Take any vertex $v \in \partial \Omega_1$.

If $\Omega_1 \setminus \{v\}$ is disconnected, then define $\Omega_{\ell+1}$ to be any one of (the closures of) these connected components and $\widetilde{\Omega}_1$ to be (the closure of) $\Omega_1 \setminus \Omega_{\ell+1}$. The corresponding graphs $\widetilde{\Graph}_1$, $\Graph_{\ell+1}$ may be taken to be any graphs in the non-empty sets $\rho_{\widetilde{\Omega}_1}$ and $\rho_{\Omega_{\ell+1}}$, respectively.

If the removal of $v$ does not disconnect $\Omega_1$, define $\Omega_{\ell+1}$ to consist of exactly one edge of $\Omega_1$ adjacent to $v$ and $\widetilde{\Omega}_1$ to be the rest of $\Omega_1$, unless $\Omega_1$ consists of just one edge, in which case take $\Omega_{\ell+1}$ to consist of the half of this edge adjacent to $v$. The graphs are defined accordingly. In any case, the monotonicity of the Dirichlet eigenvalues with respect to domain inclusion implies $\lambda_1 (\widetilde{\Graph}_1), \lambda_1 (\Graph_{\ell+1}) \geq \lambda_1 (\Graph_1)$. This establishes (2) and completes the proof of the theorem in the Dirichlet case.

In the natural case, it remains to establish the existence of some $k_0 \geq 1$ with the claimed properties. Here the proof is somewhat different from the corresponding proof of Proposition~\ref{prop:natural-monotonicity}: we will show that for $k$ sufficiently large, if a partition realises $\nmaxmin[k]$ then none of its cluster supports can wholly contain any cycle in $\Graph$, and thus each cluster is a tree. We may then apply Lemma~\ref{lem:kirchhoff-two-cut} to subdivide these if necessary, without decreasing the energy. To this end, we will need the following function. By way of analogy with \eqref{eq:rigid-cluster-set}, for any closed subset $\Omega \subset \Graph$ we define $\rho_\Omega$ to be the set of all possible rigid clusters associated with $\Omega$. We then define a function $\mathfrak{b} : [0,\infty) \to [0,|\Graph|]$ by
\begin{equation}
\label{eq:auxiliary-b}
	\mathfrak{b} (\lambda) := \sup \{ |\Omega|: \Omega \subset \Graph \text{ closed and connected and }
	\max_{\HGraph \in \rho_\Omega} \mu_2 (\HGraph) \geq \lambda \}.
\end{equation}
Now set $\ell_{\max}$ to be the length of the longest edge of $\Graph$ and $c_{\min}$ to be the length of its shortest cycle, and choose an integer $m\geq 1$ such that
\begin{displaymath}
	\mathfrak{b} \left(\frac{\pi^2m^2}{\ell_{\max}^2}\right) < c_{\min}.
\end{displaymath}
This is possible because, by Lemma~\ref{lem:auxiliary-b} below, $\mathfrak{b}(\cdot) \to 0$ as $m\to\infty$. We next choose $k_0$ to satisfy $m = \lfloor k_0/{\cardE} \rfloor$, where we recall ${\cardE}$ is the number of edges of $\Graph$.

Now fix $k\geq k_0$. We take any (rigid) $k$-partition $\parti$ of $\Graph$ in which each edge is partitioned equally into at least $m$ clusters; there exists such a partition since $m \leq \lfloor k/{\cardE} \rfloor$. Then each cluster support $\Omega_i$ of $\parti$ is identifiable with an interval; and thus the same is true of $\Graph_i$. The longest of these has length no greater than $\ell_{\max}/m$, and so
\begin{displaymath}
	\nmaxmin[k] \geq \min_{i=1,\ldots,k} \mu_2 (\Graph_i) \geq \frac{\pi^2m^2}{\ell_{\max}^2}.
\end{displaymath}
Now let $\parti^N = \{\Graph_1,\ldots, \Graph_{j_k} \}$, $j_k \leq k$, be an optimal partition for $\nmaxmin[k]$. Then we must have $\mu_2 (\Graph_i) \geq \pi^2m^2/\ell_{\max}^2$ for all $i=1,\ldots,j_k$. Hence, by choice of $m$ and definition of $\mathfrak{b}$, we have $|\Graph_i| < c_{\min}$ for all $i$: in particular, every cluster (and every cluster support) of $\parti^N$ is a tree. If $\parti^N$ has fewer than $k$ of them, then we may use Lemma~\ref{lem:kirchhoff-two-cut} to subdivide as many of the clusters of $\parti^N$ as necessary to create a $k$-partition whose energy is at least as large as $\nminenergy (\parti^N)$.
\end{proof}

We finish by proving the properties of the function $\mathfrak{b}$ claimed in the above proof.

\begin{lemma}
\label{lem:auxiliary-b}
The function $\mathfrak{b} : [0,\infty) \to [0,|\Graph|]$ defined by \eqref{eq:auxiliary-b} is well defined and monotonically decreasing, with $\mathfrak{b} (\lambda) \to 0$ as $\lambda \to \infty$.
\end{lemma}

\begin{proof}
The function is well defined on $[0,\infty)$ since for any $\lambda \geq 0$ the corresponding set is non-empty; that it is monotonically decreasing follows directly from the definition. Now suppose there exists a sequence of subsets $\Omega^{(n)} \subset \Graph$, with associated graphs $\Graph^{(n)} \in \rho_{\Omega^{(n)}}$ satisfying $\mu_2 (\Graph^{(n)}) \to \infty$ but $|\Graph^{(n)}| = |\Omega^{(n)}| \geq c > 0$ for all $n\geq 1$. Since $\Omega^{(n)}$ and $\Graph^{(n)}$ are connected, the number of edges ${\cardE}(\Graph^{(n)})$ of the latter is certainly not greater than $2{\cardE}$ (where ${\cardE}$ is the number of edges of the fixed graph $\Graph$), and hence, by \cite[Theorem~4.2]{KeKuMaMu16},\footnote{Note that the eigenvalue numbering convention in \cite{KeKuMaMu16} is different. Also, as observed in \cite[\S~2.4]{BaLe17}, there is an error in part of \cite[Theorem~4.2]{KeKuMaMu16}: the uniqueness statement in the case ${\cardE}=2$ is not correct, as any $2$-flower (among certain other graphs called \emph{symmetric necklaces} in \cite{BaLe17}) is a maximiser in this case.}
\begin{displaymath}
	\mu_2 (\Graph^{(n)}) \leq \frac{4\pi^2 {\cardE}(\Graph^{(n)})^2}{|\Graph^{(n)}|^2} \leq \frac{16\pi^2{\cardE}^2}{|\Graph|^2}
\end{displaymath}
for all $n\geq 1$, a contradiction to $\mu_2 (\Graph^{(n)}) \to \infty$.
\end{proof}

Observe that it is easy to find a sequence of graphs $\Graph^{(n)}$ for which $\mu_2 (\Graph^{(n)}) \to \infty$, even as $|\Graph^{(n)}|$ remains bounded from below, if the $\Graph^{(n)}$ are not embedded in a larger finite graph $\Graph$.

\begin{remark}
Further spectral maximal partitioning problems are conceivable, in analogy to the ones we investigated in Section~\ref{sec:ex-min}. In particular, we may look for \textit{loose} maximisers of the functional $\nminenergy (\parti),\dminenergy (\parti)$: we strongly expect that this problem always admits a solution, since each partition is only defined by a finite number of cuts.

Also, we may well introduce further functionals based on the $p$-means of $\mu_2(\Graph_i)^{-1}$, rather then on their maximum. Again, we are confident these generalisations can be handled by the theory developed here, but do not go into details.
\end{remark}

\section{Dependence of the optimal partitions on the parameters}
\label{sec:examples}

In the final two sections we will collect a number of miscellaneous properties of, and illustrative examples for, the minimisation and maximisation problems from the previous sections. In this section we will consider the dependence of the two quantities $\noptenergy[k,p]$ and $\doptenergy[k,p]$, which we consider to be the most natural, and the partitions realising them, on $p$ (for fixed $k$), and also on the edge lengths of the graph $\Graph$ being partitioned for a fixed topology.

\subsection{Dependence on $p$}
\label{sec:p-dependence}

Let us remark that the quantities $\noptenergy[k,p]$ and $\doptenergy[k,p]$ are, for fixed $\Graph$ and $k\geq 1$, continuous and monotonically increasing in $p \in [1,\infty]$. This is a general result that follows in exactly the same way as on domains, cf.~\cite[Proposition~10.53]{BNHe17}. We include the short proof for the sake of completeness. 
Note that here and throughout this section, in keeping with the convention on domains we will restrict ourselves to considering $p\in [1,\infty]$.

\begin{proposition}\label{prop:lknp}
Fix $k\geq 1$. For all $1 \leq q \leq p \leq \infty$, we have
\begin{displaymath}
	\noptenergy[k,q] (\Graph) \leq \noptenergy[k,p] (\Graph) \leq k^{\frac{1}{q}-\frac{1}{p}}\noptenergy[k,q] (\Graph) \quad \text{and} \quad
	\doptenergy[k,q] (\Graph) \leq \doptenergy[k,p] (\Graph) \leq k^{\frac{1}{q}-\frac{1}{p}}\doptenergy[k,q] (\Graph)
\end{displaymath}
(where $1/p=0$ if $p=\infty$). Consequently, the mappings $p \mapsto \noptenergy[k,p] (\Graph)$ and $p \mapsto \doptenergy[k,p] (\Graph)$ are continuous and monotonically increasing in $p \in [1,\infty]$.
\end{proposition}

\begin{proof}
We give the proof for $\noptenergy[k,p]$; the proof for $\doptenergy[k,p]$ is identical. In fact, it suffices to show that
\begin{equation}
\label{eq:cont-mon-part}
	\nenergy[q] (\parti) \leq \nenergy[p] (\parti) \leq k^{\frac{1}{q}-\frac{1}{p}} \nenergy[q] (\parti)
\end{equation}
for any rigid $k$-partition $\parti$, since then the same is true for the corresponding infima over all such partitions. But \eqref{eq:cont-mon-part} is a direct consequence of the H\"older inequality, using the definition \eqref{eq:nenergy} of $\nenergy[p] (\parti)$.
\end{proof}

We continue discussing the dependence of optimal partitions and energies on $p$. To begin with, let us present a concrete example illustrating how the optimal partition, say in the simplest case for $\doptenergy[2,p]$, can depend nontrivially on $p$. On domains, relatively little seems to be known, and most of the work to date seems to have been of (largely) numerical nature; see in particular \cite{BNBo18}. Our example, in addition to establishing that $\doptenergy[2,p]$ and the corresponding optimal partitions can, in fact, depend on $p$, should also demonstrate how in the case of metric graphs it seems possible to prove more properties (such as monotonicity of the deformation in $p$) analytically. On the other hand, since $\doptenergy[k,\infty]$ need not be realised by an equipartition, it follows that known criteria for establishing the inequality $\doptenergy[k,1] < \doptenergy[k,\infty]$ (see \cite[Proposition~10.54]{BNHe17} or \cite{HeHO10}) have no direct equivalent; see Proposition~\ref{prop:lk1lkinfty} and the discussion around it. Moreover, strict inequality here is possible even if the optimal partition is independent of $p$; see Example~\ref{ex:p-dep-part-indep}.

\begin{example}
\label{ex:pavels-task}
We return to the equilateral star graph $\Graph$ on three edges of length $1$ each, considered in Example~\ref{ex:3-star}, and ask for the partitions achieving $\doptenergy[2,p]$ for $p\in [1,\infty)$ (we recall that when $p=\infty$, up to isometry there is one optimal partition, whose cut set consists of (only) the central vertex of degree $3$). Any 2-partition $\parti = \{\Graph_1,\Graph_2\}$ of $\Graph$ may, up to symmetries, be identified uniquely by the location of its cut set $v_0$ along a given, fixed edge (see Figure~\ref{fig:what-a-star}).
\begin{figure}[H]
\begin{tikzpicture}[scale=1.2]
\coordinate (a) at (0,0);
\coordinate (b) at (0,1.5);
\coordinate (c) at (-1.3,-0.75);
\coordinate (d) at (1.3,-0.75);
\draw[thick] (a) -- (b);
\draw[thick] (a) -- (c);
\draw[thick] (a) -- (d);
\draw[fill] (a) circle (1.75pt);
\draw[fill] (b) circle (1.75pt);
\draw[fill] (c) circle (1.75pt);
\draw[fill] (d) circle (1.75pt);
\filldraw[thick, draw=black, fill=white] (-0.78,-0.45) circle (1.75pt);
\node at (-0.75,-0.45) [anchor=south east] {$v_0$};
\node at (-0.47,-0.2) [anchor=north west] {$a$};
\node at (0.5,0.5) [anchor=west] {$\Graph_1$};
\node at (-1.4,-0.45) [anchor=east] {$\Graph_2$};
\end{tikzpicture}\vspace{5pt}
\caption{The equilateral 3-star $\Graph$. The white circle denotes the cut set $\{v_0\}$ of the two-partition $\parti = \{\Graph_1,\Graph_2\}$; the corresponding edge is divided into pieces of length $a$ in $\Graph_1$ and $1-a$ in $\Graph_2$.}
\label{fig:what-a-star}
\end{figure}
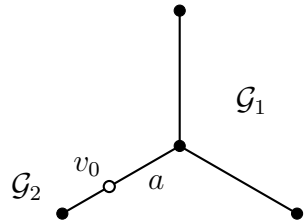
Since the edge has length $1$, the partition may uniquely be described by a single parameter $a \in [0,1)$, with $a=0$ corresponding to a cut in the central vertex and thus the partition realising $\doptenergy[2,\infty]$, and $a=1$ formally giving $\Graph_1 = \Graph$ and $\Graph_2 = \{v_0\}$. The following proposition describes how the optimal cut point depends on $p$. We will give its proof below.
\end{example}

\begin{proposition}
\label{prop:pavels-task}
For the equilateral 3-star $\Graph$, in the notation and setup of Example~\ref{ex:pavels-task}, for each $p \in [1,\infty]$, there is a unique value of $a \in [0,1)$ whose corresponding partition achieves $\doptenergy[2,p] (\Graph)$, which we denote by $a_p$. Then $a_p$ is a smooth function of $p$, with $a_p>0$ and $\frac{d}{dp} a_p < 0$ for all $p \in [1,\infty)$, and $\lim_{p\to\infty}a_p=0$.
\end{proposition}

In particular, the optimal partition is never the equipartition except for $p=\infty$ (which we recall corresponds to $a_\infty = 0$), and the cut point $v_0$, as a function of $p$, moves smoothly and monotonically from its location at $p=1$ towards the central vertex as $p \to \infty$. This mirrors very much the numerically observed behaviour of the (conjectured) optimal partitions on domains in \cite{BNBo18}.

\begin{remark}
As an immediate consequence of Proposition~\ref{prop:pavels-task}, we obtain the inequalities
\begin{displaymath}
	\mu_2 (\Graph) = \doptenergy[2,\infty] (\Graph) > \doptenergy[2,p] (\Graph)
\end{displaymath}
for all $p \in [1,\infty)$, for the example where $\Graph$ is an equilateral 3-star. In particular, there is no generalisation of Proposition~\ref{prop:inequalities-k} to $p \neq \infty$.
\end{remark}

The above example suggests that the optimal partition for $\doptenergy[k,p](\Graph)$ should depend on $p$ whenever there is an optimal partition for $p=\infty$ which is \emph{not internally connected}.  It would take us too far afield to consider this question here, so we formulate it as an open problem.

\begin{conjecture}
Let $\Graph$ be given and let $k\geq 1$.
\begin{enumerate}
\item Suppose there exists an exhaustive rigid $k$-partition $\parti$ achieving $\doptenergy[k,\infty](\Graph)$ which is not internally connected. Then $\parti$ does \emph{not} achieve $\doptenergy[k,p](\Graph)$ for any $p < \infty$, that is, $\denergy[p](\parti) > \doptenergy[k,p](\Graph)$ for all $p < \infty$.
\item Whenever there exists a rigid $k$-partition $\parti$ achieving $\doptenergy[k,\infty](\Graph)$ but which, for some $p < \infty$, does \emph{not} achieve $\doptenergy[k,p](\Graph)$, then we have that $\doptenergy[k,p](\Graph)$ is a \emph{strictly} monotonic function of $p$.
\end{enumerate}
\end{conjecture}

\begin{proof}[Proof of Proposition~\ref{prop:pavels-task}]
First we compute the energy as a function of $a \in [0,1)$:
\begin{displaymath}
	\denergy[p] (a) := \denergy[p] \left(\left\{\Graph_1 (a),\Graph_2 (a) \right\}\right)
\end{displaymath}
We clearly have $\lambda_1 (\Graph_2(a)) = {\pi^2}/{4(1-a)^2}$. Noting that the eigenfunction is identical on the two identical edges, we can obtain that $\lambda_1 (\Graph_1) =: \omega(a)^2$, with $ \omega(a)^2$ the smallest positive solution of the secular equation
\begin{equation}
\label{eq:secular-fork}
	2\tan (a\omega) = \cot (\omega).
\end{equation}
We note that $\omega(a)<\pi/2$ for $a\in(0,1)$ and, by implicit differentiation,
\begin{equation}
\label{eq:fork-out}
	\frac{d\omega}{da} = - \frac{\omega}{a+\frac{\cos^2(a\omega)}{2\sin^2(\omega)}}.
\end{equation}

Let us now show that, for all $a\in(0,1)$, 
\begin{equation}
\label{eq:D2omega}
 \frac{d^2\omega}{da^2}(a)>0 .
\end{equation}
Differentiating equation \eqref{eq:fork-out}, we find
\begin{displaymath}
	\frac{d^2\omega}{da^2}(a)=\frac{\omega}{\left(a+\frac{\cos^2(a\omega)}{2\sin^2(\omega)}\right)^2}+\frac{\omega}{\left(a+\frac{\cos^2(a\omega)}{2\sin^2(\omega)}\right)^2}\left(1+\frac{\cos(a\omega)}{\sin(\omega)}\frac{d}{da}\left(\frac{\cos(a\omega)}{\sin(\omega)}\right)\right).
\end{displaymath}
Using equation \eqref{eq:fork-out} again, we obtain
\[
\begin{split}
\frac{d}{da}\left(\frac{\cos(a\omega)}{\sin(\omega)}\right)&=\frac1{\sin^2(\omega)}\left(-\left(\omega+a\frac{d\omega}{da}\right)\sin(a\omega)\sin(\omega)-\frac{d\omega}{da}\cos(a\omega)\cos(\omega)\right)\\
&=
\frac\omega{\sin^2(\omega)\left(a+\frac{\cos^2(a\omega)}{2\sin^2(\omega)}\right)}\bigg(a\sin(a\omega)\sin(\omega)+\cos(a\omega)\cos(\omega)\\
&\qquad -\left(a+\frac{\cos^2(a\omega)}{2\sin^2(\omega)}\right)\sin(a\omega)\sin(\omega)\bigg)\\
&=
\frac{\omega\cos(a\omega)}{2\sin^3(\omega)\left(a+\frac{\cos^2(a\omega)}{2\sin^2(\omega)}\right)}\left(2\cos(\omega)\sin(\omega)-\cos(a\omega)\sin(a\omega)\right).
\end{split}
\]
We have, using the secular equation \eqref{eq:secular-fork},
\[
\begin{split}
	2\cos(\omega)\sin(\omega)-\cos(a\omega)\sin(a\omega)&=2\cot(\omega)\sin^2(\omega)-\tan(a\omega)\cos^2(a\omega)\\&=(4\sin^2(\omega)-\cos^2(a\omega))\tan(a\omega).
\end{split}
\]
Using again the secular equation, we find successively
\[
\begin{split}
	4\tan^2(a\omega)&=\cot^2(\omega);\\
	4\left(\frac1{\cos^2(a\omega)}-1\right)&=\frac1{\sin^2(\omega)}-1;\\
	\frac4{\cos^2(a\omega)}-\frac1{\sin^2(\omega)}&=3.
\end{split}
\]
It follows that $4\sin^2(\omega)-\cos^2(a\omega)$, and therefore also $\frac{d}{da}\left(\frac{\cos(a\omega)}{\sin(\omega)}\right)$ and $\frac{d^2\omega}{da^2}(a)$, are positive.

 For convenience, we define the function
\begin{displaymath}
\label{eq:energy-of-a-star}
	F(a,p):=2\denergy[p] (a)^p = \left(\frac{\pi}{2(1-a)}\right)^{2p} + \omega(a)^{2p}.
\end{displaymath}
We have immediately
\begin{equation}
\label{eq:DF}
	\frac{\partial F}{\partial a} (a,p) = 2p\left(\left(\frac{\pi}{2}\right)^{2p}\frac{1}{(1-a)^{2p+1}}+\omega^{2p-1}\frac{d\omega}{da}\right).
\end{equation}
We have  $\omega(0)=\frac{\pi}{2}$ and so, using equation \eqref{eq:fork-out},
\begin{displaymath}
	\frac{\partial F}{\partial a}(0,p) = 2p \left(\left(\frac{\pi}{2}\right)^{2p} - 2\left(\frac{\pi}{2}\right)^{2p}\right)=-2p \left(\frac{\pi}{2}\right)^{2p}< 0,
\end{displaymath}
so that $a=0$ is not even a local minimum of $a\mapsto F(a,p)$; while, as $a \to 1$,
\begin{displaymath}
	\left(\frac{\pi}{2}\right)^{2p}\frac{2p}{(1-a)^{2p+1}} \longrightarrow +\infty
\end{displaymath}
and the term
\begin{displaymath}
	\omega^{2p-1}\frac{d\omega}{da}= - \frac{2p\omega^{2p}}{a+\frac{\cos^2(a\omega)}{2\sin^2(\omega)}}
\end{displaymath} 
is bounded. We conclude $\frac{\partial}{\partial a} F(a,p) \to +\infty$ as $a \to 1$.
On the other hand,
\begin{displaymath}
	\frac{\partial^2 F}{\partial a^2} (a,p) = 2p\left(\left(\frac{\pi}{2}\right)^{2p}\frac{2p+1}{(1-a)^{2p+2}}+(2p-1)\omega^{2p-2}\left(\frac{d\omega}{da}\right)^2+\omega^{2p-1}\frac{d^2\omega}{da^2}\right),
\end{displaymath}
which is clearly positive as a consequence of inequality \eqref{eq:D2omega}. The function $a\mapsto \frac\partial{\partial a}F(a,p)$ is therefore increasing, and has a unique zero in $[0,1)$, which is positive, and corresponds to a global minimum of $a\mapsto F(a,p)$. We have proved the first part of Proposition  \ref{prop:pavels-task}.

Since $\frac{\partial^2}{\partial a^2}F(a_p,p)>0$, it follows from the Implicit Function Theorem that $p\mapsto a_p$ is continuously differentiable (indeed, even real analytic) and that
\begin{displaymath}
	\frac{da_p}{dp}(p) = -\frac{\frac{\partial^2}{\partial p \partial a}F(a_p,p)}{\frac{\partial^2}{\partial a^2}F(a_p,p)}.
\end{displaymath} 
Differentiating equation \eqref{eq:DF} with respect to $p$, we find
\begin{displaymath}
	\frac{\partial F}{\partial p\partial a} (a,p) = \frac1p\frac{\partial F}{\partial a}(a,p)+2p\left(2\log\left(\frac{\pi}{2(1-a)}\right)\left(\frac{\pi}{2}\right)^{2p}\frac{1}{(1-a)^{2p+1}}+2\log(\omega)\omega^{2p-1}\frac{d\omega}{da}\right).
\end{displaymath}
Using $\frac{\partial}{\partial a}F(a_p,p)=0$ and equation \eqref{eq:DF}, we obtain in particular
\begin{displaymath}
	\frac{\partial^2 F}{\partial p\partial a} (a_p,p) =8p\left(\frac{\pi}{2}\right)^{2p}\frac{1}{(1-a_p)^{2p+1}}\log\left(\frac{\pi}{2(1-a_p)\omega(a_p)}\right).
\end{displaymath} 
Since $a_p\in (0,1)$ and $\omega(a_p)\in(0,\pi/2)$, the above derivative is positive and  $p\mapsto a_p$ decreasing.

As a positive and decreasing function defined on $[1,\infty)$, $p\mapsto a_p$ has a non-negative limit at $\infty$, which we denote by $a^*$.  Let us assume by contradiction that $a^*>0$. Using equations \eqref{eq:fork-out} and \eqref{eq:DF}, the condition $\frac{\partial}{\partial a}F(a_p,p)=0$ can be written 
\begin{displaymath}
	\left(\frac{\pi}{2}\right)^{2p}\frac{1}{(1-a_p)^{2p+1}}=\frac{\omega(a_p)^{2p}}{a_p+\frac{\cos^2(a_p\omega(a_p))}{2\sin^2(\omega(a_p))}}.
\end{displaymath}
It follows that 
\begin{displaymath}
	\frac{\pi}{2(1-a_p)}(1-a_p)^{-\frac1{2p}}=\omega(a_p)\left(a_p+\frac{\cos^2(a _p\omega(a_p))}{2\sin^2(\omega(a_p))}\right)^{-\frac1{2p}}.
\end{displaymath}
Passing to the limit $p\to\infty$, we obtain 
\begin{displaymath}
	\frac{\pi}{2(1-a^*)}=\omega(a^*),
\end{displaymath}
in contradiction to $\omega(a^*)<\pi/2$. We conclude that $a^*=0=a_\infty$.
\end{proof}

We now return to the meaning of the inequality $\doptenergy[k,1](\Graph) < \doptenergy[k,\infty](\Graph)$ for a metric graph $\Graph$.

\begin{example}
\label{ex:p-dep-part-indep}
We give a simple example where the optimal partition for $\doptenergy[k,p](\Graph)$ is independent of $p \in [1,\infty]$ but the optimal energy $\doptenergy[k,p](\Graph)$ itself is not; this is a direct consequence of the existence of certain minimal partitions which are not equipartitions.  Indeed, take $\Graph$ to be a not quite equilateral star on three edges, say of length $|e_1| = 1+\varepsilon$, $|e_2| = |e_3| = 1$. We denote by $v$ the central vertex of $\Graph$ and by $v_1$, $v_2$ and $v_3$ the pendant vertices of $e_1$, $e_2$ and $e_3$ respectively. Let us denote by $\parti_3^0$ the $3$-partition of $\Graph$ whose cut set is $v$. Then
\begin{equation*}
	\denergy[p] (\parti_3^0)= \left( \frac{1}{3}\left(\frac{\pi^{2p}}{(2+2\varepsilon)^{2p}} + 2\cdot \frac{\pi^{2p}}{2^{2p}}\right)\right)^{1/p}.
\end{equation*}
This energy clearly depends on $p$. The following proposition establishes that our example has the desired properties.
\end{example}

\begin{proposition}
\label{prop:3-star-partitions} There exists $\varepsilon_0>0$ such that, for all $\varepsilon\in[0,\varepsilon_0]$ and all $p\in[1,\infty]$, $\parti_3^0$ is the unique $3$-partition realising $\doptenergy[3,p](\Graph)$.
\end{proposition}

\begin{proof} Let us first prove the proposition for $p\in[1,\infty)$, by a straightforward discussion of the possible topological cases. Let us consider an exhaustive  $3$-partition \[\parti = \{\Graph_1,\Graph_2,\Graph_3\}\] different from $\parti_3^0$. Its cut set $\{w_1,w_2\}$ consists of two points, each distinct from the central vertex. This leaves us with four essentially distinct cases (all the other cases reduce to these four by relabelling of the edges and cut points):
\begin{enumerate}[(i)]
	\item the two points belong to $e_2$;
	\item the two points belong to $e_1$;
	\item $w_1$ belongs to $e_2$ and $w_2$ to $e_3$;
	\item $w_1$ belongs to $e_1$ and $w_2$ to $e_2$.
\end{enumerate} 
In all cases, we denote by $a_i$ the distance of $w_i$ from $v$, $i=1,2$. To simplify notation, we define
\begin{equation*}
	F_0(p):=\frac{3\cdot2^{2p}}{\pi^{2p}}\denergy[p] (\parti_3^0)^{p}=\frac1{(1+\varepsilon)^{2p}}+2.
\end{equation*} 

In Case (i), up to relabelling, we can assume that $a_1<a_2$. We can also assume that $\Graph_1$ is the $3$-star with edges $e_1$, $e_3$ and $[v,w_1]$, with a boundary condition at each pendant vertex respectively Neumann, Neumann and Dirichlet. We recall that here and in the rest of the proof we have a natural boundary condition at the central vertex $v$. We can finally assume that $\Graph_2$ is the segment $[w_1,w_2]$ with Dirichlet boundary conditions and $\Graph_3$ the segment  $[w_2,v_2]$ with a Dirichlet boundary condition at $w_2$ and a Neumann boundary condition at $v_2$. We have, recalling that $p\ge1$,
\[
\begin{split}
	\frac{3\cdot2^{2p}}{\pi^{2p}}\denergy[p] (\parti)^{p}&=\left(\frac{4}{\pi^2}\lambda_1(\Graph_1)\right)^{p}+\left(\frac{4}{\pi^2}\lambda_1(\Graph_2)\right)^{p}+\left(\frac{4}{\pi^2}\lambda_1(\Graph_2)\right)^{p}\\
	&= 
	\left(\frac{4}{\pi^2}\lambda_1(\Graph_1)\right)^{p}+\frac{2^{2p}}{(a_2-a_1)^{2p}}+\frac1{(1-a_2)^{2p}}>2^{2p}+1\ge5>3>F_0(p).
\end{split}
\]	

In Case (ii), we assume, without loss of generality, that $a_1<a_2$,  $\Graph_1$ is the three star with edges $e_2$, $e_3$ and $[v,w_1]$, with a boundary condition at each pendant vertex respectively Neumann, Neumann and Dirichlet, $\Graph_2$ is the segment $[w_1,w_2]$ with Dirichlet boundary conditions and $\Graph_3$ the segment  $[w_2,v_1]$ with a Dirichlet boundary condition at $w_2$ and a Neumann boundary condition at $v_1$. We have 
\[
\begin{split}
	\frac{3\cdot2^{2p}}{\pi^{2p}}\denergy[p] (\parti)^{p}&=\left(\frac{4}{\pi^2}\lambda_1(\Graph_1)\right)^{p}+\frac{2^{2p}}{(a_2-a_1)^{2p}}+\frac1{(1+\varepsilon-a_2)^{2p}}\\ 
	&	>\frac{2^{2p}}{(a_2-a_1)^{2p}}>\frac{2^{2p}}{(1+\varepsilon)^{2p}}.
	\end{split}
	\]
If we assume $\varepsilon\le2/\sqrt3-1<1$, it follows, since $p\ge1$, 
\begin{equation*}		
	\frac{3\cdot2^{2p}}{\pi^{2p}}\denergy[p] (\parti)^{p}>\frac{2^{2p}}{(1+\varepsilon)^{2p}}\ge\frac{2^{2}}{(1+\varepsilon)^{2}}\ge3>F_0(p).
\end{equation*}

In Case (iii), we assume, without loss of generality, that $\Graph_1$ is the $3$-star with edges $e_1$, $[v,w_1]$ and $[v,w_2]$, with boundary conditions respectively Neumann, Dirichlet and Dirichlet at the pendant vertices. We also assume that $\Graph_2$ and $\Graph_3$ are respectively the segments $[w_1,v_2]$ and $[w_2,v_3]$ with Dirichlet-Neumann boundary conditions. To simplify notation, we introduce
\begin{equation*}
	F(a_1,a_2):=\frac{3\cdot2^{2p}}{\pi^{2p}}\denergy[p] (\parti)^{p}=\left(\frac2\pi\omega(a_1,a_2)\right)^{2p}+\frac{1}{(1-a_1)^{2p}}+\frac{1}{(1-a_2)^{2p}},
\end{equation*}
where $\omega(a_1,a_1)$ is the smallest positive solution of the equation
\begin{equation}
\label{eqG1Caseiii}
	\mbox{cotan}(a_1\omega)+\mbox{cotan}(a_2 \omega)=\tan(\omega),
\end{equation}
so that $\lambda_1(\Graph_1)=\omega(a_1,a_2)^2$. To simplify notation, we write $\omega:=\omega(a_1,a_2)$. We now show that $\frac\partial{\partial a_2}F(a_1,a_2)$ is positive for all $(a_1,a_2)\in(0,1)$. This claim implies that $F(a_1,a_2)>\lim_{a\to0}F(a_1,a)$. Noticing that 
\begin{equation*}
	\lim_{a\to0}F(a_1,a)=\frac1{(1+\varepsilon)^{2p}}+\frac{1}{(1-a_1)^{2p}}+1>\frac1{(1+\varepsilon)^{2p}}+2=F_0(p),
\end{equation*}
we conclude that the claim implies $\denergy[p] (\parti)>\denergy[p] (\parti_3^0)$. Let us now prove the claim. By inspection of Equation \eqref{eqG1Caseiii}, it is clear that $\omega\in(0,\pi/2)$, since $\tan$ as a pole at $\pi/2$, so that $a_1\omega$ and $a_2\omega$ also belong to $(0,\pi/2)$. After differentiating Equation \eqref{eqG1Caseiii} with respect to $a_2$ and simplifying, we find
\begin{equation}
\label{eqDiffCaseiii}
	\frac{\partial \omega}{\partial a_2}=-\frac{\omega}{a_2+\left(\frac{\sin a_2\omega}{\cos\omega}\right)^2+\left(\frac{\sin a_2\omega}{\sin a_1\omega}\right)^2}.
\end{equation}
To go further, let us note that since $\omega$, $a_1\omega$ and $a_2\omega$ belong to $(0,\pi/2)$, all the terms in Equation \eqref{eqG1Caseiii} are positive, so that $0<\mbox{cotan}(a_2\omega)<\tan(\omega)$. Taking the square, we find
\begin{equation*}
	\frac1{(\sin a_2 \omega)^2}=\mbox{cotan}( a_2\omega)^2<\tan(\omega)^2=\frac1{(\cos\omega)^2},
\end{equation*}
that is to say
\begin{equation*}
	\frac{(\sin a_2\omega)^2}{(\cos\omega)^2}>1.
\end{equation*}
It follows that
\begin{equation}
\label{eqAbsCaseiii}
	\left\vert\frac{\partial \omega}{\partial a_2}\right\vert<\omega<\frac\pi2.
\end{equation}
Using Inequality \eqref{eqAbsCaseiii}, we now find
\[
\begin{split}
	\frac{\partial F}{\partial a_2}(a_1,a_2)&=2p\left(\left(\frac{2}{\pi}\right)^{2p}\frac{\partial\omega}{\partial a_2}\omega^{2p-1}+\frac1{(1-a_2)^{2p+1}}\right)\\
&\ge	2p\left(\frac1{(1-a_2)^{2p+1}}-\left(\frac{2}{\pi}\right)^{2p}\left\vert\frac{\partial\omega}{\partial a_2}\right\vert\omega^{2p-1}\right)>0,
\end{split}
\]
proving the claim. Let us note that the derivative of $F(a_1,a_2)$ with respect to $a_1$ is also positive, by symmetry.

Let us finally study Case (iv).  We assume, without loss of generality, that $\Graph_1$ is the $3$-star with edges $e_3$, $[v,w_1]$ and $[v,w_2]$, with boundary conditions respectively Neumann, Dirichlet and Dirichlet at the pendant vertices and that $\Graph_2$ and $\Graph_3$ are respectively the segments $[w_1,v_1]$ and $[w_2,v_2]$ with Dirichlet-Neumann boundary conditions. Assuming that $a_1\le1$, we can repeat the computation of Case (iii) and show that $\frac\partial{\partial a_2}F(a_1,a_2)$ is positive, and therefore 
\begin{equation*}
	\frac{3\cdot2^{2p}}{\pi^{2p}}\denergy[p] (\parti)^{p}=F(a_1,a_2)>\lim_{a\to0}F(a_1,a)>F_0(p).
\end{equation*}
If $a_1>1$ and $\varepsilon\le1/\sqrt2$, we have
\[
\begin{split}
	\frac{3\cdot2^{2p}}{\pi^{2p}}\denergy[p] (\parti)^{p}&>\left(\frac{4}{\pi^2}\lambda_1(\Graph_2)\right)^{p}+\left(\frac{4}{\pi^2}\lambda_1(\Graph_2)\right)^{p}\\
	&= \frac{1}{(1+\varepsilon-a_1)^{2p}}+\frac{1}{(1-a_2)^{2p}}>\frac1{\varepsilon^{2p}}+1\ge3>F_0(p).
\end{split}
\]

Altogether, assuming that $\varepsilon\le\varepsilon_0:=\min\left(2/\sqrt3-1,1/\sqrt2\right)=2/\sqrt3-1$, $\denergy[p] (\parti)>\denergy[p] (\parti_3^0)$ in all cases. By a limiting argument, we could show immediately that $\parti_3^0$ is a minimal partition realising $\doptenergy[3,\infty]$. This would however not prove uniqueness. It is better to give a direct proof. We first note that $\denergy[\infty](\parti_3^0)=\pi^{2}/4$. On the other hand, let us again consider $\parti$, an exhaustive $3$-partition distinct  from $\parti_3^0$. In Cases (i), (iii) and (iv), $\parti$ has a domain strictly contained in an edge of length $1$. In case (iii), assuming $\varepsilon\le1$, the edge $e_1$ contains a segment with Dirichlet-Neumann conditions and one with Dirichlet-Dirichlet conditions, one of which has length less than $1$. It follows that, in all cases, $\denergy[\infty](\parti)>\pi^2/4$.  
\end{proof}

The value $\varepsilon_0=2/\sqrt3-1$ arrived at during the proof is probably not the largest for which Proposition \ref{prop:3-star-partitions} holds and we have not tried to optimise it. Let us however note that $\varepsilon$ cannot be chosen too large. More explicitly, if $\varepsilon>\varepsilon_1:=2\sqrt3-1$, and if we choose for $\parti$ a partition of  type (iii) (see the proof of Proposition \ref{prop:3-star-partitions}) such  that $a_2-a_1>2(1+\varepsilon_1)/3$ and $1+\varepsilon-a_2>(1+\varepsilon_1)/3$, we have $\denergy[p](\parti)<\denergy[p](\parti_3^0)$ for all $p\in [1,\infty]$.

 A ``balancing formula'' which can be found in
 \cite[Proposition~10.54]{BNHe17} gives a sufficient condition under which the inequality $\doptenergy[k,1] < \doptenergy[k,\infty]$ holds  for domains $\Omega \subset \R^2$.
   For a $k$-partition $\parti = \{\Omega_1,\ldots,\Omega_k\}$ of $\Omega$, minimal for $\doptenergy[k,\infty]$, if $\Omega_i$ and $\Omega_j$ are neighbours and $\Omega_{ij}$ is the interior of the closure of $\Omega_i \cup \Omega_j$, then
\begin{equation}
\label{eq:domain-equality}
	\lambda_1 (\Omega_i) = \lambda_1 (\Omega_j) = \lambda_2 (\Omega_{ij}) = \doptenergy[k,\infty](\Omega);
\end{equation}
if $\psi_i$ and $\psi_j$ are the respective eigenfunctions on $\Omega_i$ and $\Omega_j$, extended by zero to the rest of $\Omega$, scaled in such a way that $\psi_i + \psi_j$ is an eigenfunction for $\Omega_{ij}$; then the formula states that $\doptenergy[k,1](\Omega) < \doptenergy[k,\infty](\Omega)$ if, under this normalisation,
\begin{equation}
\label{eq:hadamard-imbalance}
	\|\psi_i\|_{L^2(\Omega_i)} \neq \|\psi_j\|_{L^2(\Omega_j)}.
\end{equation}
This may fail on graphs, as Example~\ref{ex:p-dep-part-indep} shows. However, we still have the following positive result, based on a slightly different normalisation.

Here we assume for fixed $k$ that $\parti = \{\Graph_1,\ldots,\Graph_k\}$ is a $k$-partition that achieves $\doptenergy[k,\infty](\Graph)$, and we suppose $\Graph_i$ and $\Graph_j$ to be any two neighbours (in the sense of Definition~\ref{def:neighbours}), with respective eigenfunctions $\psi_i$ and $\psi_j$. We treat $\psi_i$ and $\psi_j$ as elements of $H^1 (\Graph)$, extending them by zero outside $\Omega_i$ and $\Omega_j$, respectively.

\begin{proposition}\label{prop:lk1lkinfty}
Under the conditions described above, if $v \in \partial\Omega_i \cap \partial\Omega_j$ is a vertex of degree two in $\Graph$ and $\psi_i$ and $\psi_j$ are normalised in such a way that $\psi_i + \psi_j$ is continuously differentiable at $v$, then condition \eqref{eq:hadamard-imbalance} implies that $\doptenergy[k,1](\Graph) < \doptenergy[k,\infty](\Graph)$.
\end{proposition}

The idea of the proof, based on a formula of \emph{Hadamard type} for domain perturbation (see, for example, \cite[Remark~3.14]{BeKeKuMu18} for the metric graph version), is essentially identical on graphs, and we omit it.

\subsection{Dependence on the edge lengths}

Here we give an example which shows that $\noptenergy[2,\infty]$, while still continuous, need not be a smooth function of the edge lengths of the underlying graph. This is caused by the existence of an isolated value of the length of a given edge for which there is non-uniqueness of the minimiser, and a transition from one type of minimiser to another at this point: one family of partitions ``leapfrogs'' the other. It would be interesting to know whether a similar phenomenon is possible if, instead of varying $p$, we vary the edge lengths.

\begin{example}
\label{ex:n-edge-lengths}
We consider the lasso graph of Figure~\ref{fig:basic-example}; we will denote by $\Graph_a$ the particular lasso graph for which $|e_2|=|e_3|=1$ but $|e_1|=a \in [2,\infty)$. We distinguish four cases:
\begin{enumerate}
\item $a=2$. Here the optimal partition for $\noptenergy[2,\infty](\Graph_2) = \pi^2/4$ corresponds to the partition depicted in Figure~\ref{fig:basic-example-rigid}, as follows from Lemma~\ref{lem:neumann-identify-observation}. It is clear that this partition is the unique minimiser.
\item $a \in (2,3)$. We claim that the unique minimiser realising $\noptenergy[2,\infty](\Graph_a)$ is still the partition from Figure~\ref{fig:basic-example-rigid}, meaning in particular that $\noptenergy[2,\infty](\Graph_a) = \pi^2/4$ for all $a \in [2,3)$, and that the optimal partition is not an equipartition for $a \in (2,3)$. These claims follow immediately from Lemma~\ref{lem:lasso} below, applied to any $b \in (0,1)$.
\item $a=3$. Here we have two minimisers, depicted in Figure~\ref{fig:lasso-2-n-partition}, and $\noptenergy[2,\infty](\Graph_3) = \pi^2/4$ still. To see that these are minimisers, observe that any optimal partition must involve cutting through $e_1$, as otherwise the right-hand cluster would have an eigenvalue larger than $\pi^2/4$. But the strict dependence of the eigenvalue of an interval on the length of the intercal, plus the strict monotonicity statement of Lemma~\ref{lem:lasso}, guarantees that any other partition of $\Graph_3$ which cuts through $e_1$ will have a higher energy than $\pi^2/4$.
\begin{figure}[h]
\begin{tikzpicture}[scale=1.2]
\coordinate (a) at (-3,0);
\coordinate (b) at (0,0);
\coordinate (c) at (0.5,0.5);
\coordinate (d) at (0.5,-0.5);
\coordinate (e) at (1.5,0);
\draw[thick] (a) -- (b);
\draw[thick,bend right=45] (d) edge (e);
\draw[thick,bend left=45] (c) edge (e);
\draw[fill] (a) circle (1.75pt);
\draw[fill] (b) circle (1.75pt);
\draw[fill] (c) circle (1.75pt);
\draw[fill] (d) circle (1.75pt);
\draw[fill] (e) circle (1.75pt);
\draw[thick,dashed] (0.1,0.5) -- (0.6,0.1);
\draw[thick,dashed] (0.1,-0.5) -- (0.6,-0.1);
\coordinate (f) at (3,0);
\coordinate (g) at (5,0);
\coordinate (h) at (6,0);
\coordinate (i) at (7,0);
\coordinate (j) at (8,0);
\draw[fill] (f) circle (1.75pt);
\draw[fill] (g) circle (1.75pt);
\draw[fill] (h) circle (1.75pt);
\draw[fill] (i) circle (1.75pt);
\draw[fill] (j) circle (1.75pt);
\draw[thick] (f) -- (g);
\draw[thick] (h) -- (i);
\draw[thick,bend right=60] (i) edge (j);
\draw[thick,bend left=60] (i) edge (j);
\draw[thick,dashed] (5.5,0.4) -- (5.5,-0.4);
\node at (-1.5,-0.25) [anchor=north] {$\pi^2/9$};
\node at (1.5,-0.25) [anchor=north west] {$\pi^2/4$};
\node at (4,-0.25) [anchor=north] {$\pi^2/4$};
\node at (6.75,-0.25) [anchor=north] {$\pi^2/4$};
\node at (-1.5,0) [anchor=south] {$e_1$};
\node at (1.25,0.4) [anchor=south] {$e_2$};
\node at (1.05,-0.55) [anchor=north] {$e_3$};
\node at (4,0) [anchor=south] {$2$};
\node at (6.5,0) [anchor=south] {$1$};
\node at (7.6,0.2) [anchor=south west] {$1$};
\node at (7.6,-0.2) [anchor=north west] {$1$};
\end{tikzpicture}
\caption{Two different partitions realising $\noptenergy[2,\infty](\Graph_3)$, together with the respective eigenvalues of the partition clusters. On the left the original edges of $\Graph_3$ are labelled; while on the right the edge lengths for the corresponding partition clusters are displayed.}
\label{fig:lasso-2-n-partition}
\end{figure}
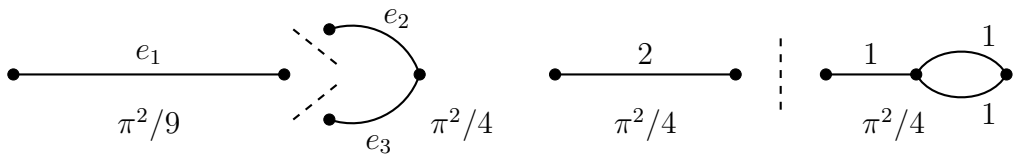
\item $a \in (3,\infty)$. Here there must be a unique minimal partition, which is a smooth evolution of the one depicted on the right of Figure~\ref{fig:lasso-2-n-partition} as $a>3$ becomes larger. Both clusters should grow, and their corresponding equal eigenvalues decrease, monotonically increasing in $a$. This follows, firstly, from the necessity of the partition being an equipartition, and the strictly negative derivative of the eigenvalues with respect to the edge lengths (use \cite[Remark~3.14]{BeKeKuMu18}).
\end{enumerate}
The optimal partition energy $\noptenergy[2,\infty](\Graph_a)$ is thus a continuous and (weakly!)\ monotonically decreasing, but not $C^1$, function of $a \in [2,\infty)$.
\end{example}

\begin{lemma}
\label{lem:lasso}
The lasso graph $\Graph_b$ of Figure~\ref{fig:basic-example}, with side lengths $|e_2|=|e_3|=1$ and $|e_1|=b > 0$, has the following properties:
\begin{enumerate}
\item $\mu_2 (\Graph_b)$ is simple and a smooth, strictly monotonically increasing function of $b>0$;
\item $\mu_2 (\Graph_1) = \pi^2/4$.
\end{enumerate}
In particular, $\mu_2 (\Graph_b) > \pi^2/4$ whenever $b<1$.
\end{lemma}

\begin{proof}
(1) is an immediate consequence of \cite[Lemma~5.5 and Corollary~3.12(1)]{BeKeKuMu18}, together with the well-known analyticity of the eigenvalues in dependence on the edge lengths (see, e.g., \cite{BerKuc12} or \cite[\S 3.1]{BerKuc13}). For (2), we compare $\Graph_1$ with the equilateral 3-star $\Graph$ whose edges are each of length one. Now $\Graph$ can be realised from $\Graph_1$ by cutting through the vertex $z$; moreover, since there exists an eigenfunction for $\mu_2(\Graph)$ which takes on the same value at the two degree one vertices of $\Graph$ which can be glued together to obtain $\Graph_1$, we have $\mu_2 (\Graph_1) = \mu_2 (\Graph) = \pi^2/4$ by \cite[Corollary~3.6]{BeKeKuMu18}.
\end{proof}

\section{Comparison of different partition problems}
\label{sec:comparison}

In this section we will compare different types of partitions problems and their corresponding optimal partitions on a fixed graph, mostly via illustrative examples. We will mostly restrict ourselves to the case of rigid partitions and in particular to the four quantities $\noptenergy[k,p]$, $\doptenergy[k,p]$, $\nmaxmin[k]$ and $\dmaxmin[k]$. We will, however, also look briefly at properties of non-exhaustive partitions and in particular consider their relationship to non-rigid partitions.

\subsection{Comparison of $\noptenergy[k,p]$, $\doptenergy[k,p]$, $\nmaxmin[k]$ and $\dmaxmin[k]$}\label{sec:lndmnd}

Here we give a few prototypical examples illustrating how these four problems give rise to different optimal partitions, and compare the actual optimal energies. In terms of the form of the optimal partitions, the examples provide at least preliminary evidence to suggest that (very roughly speaking) $\noptenergy[k,p]$ tends to seek out the longest possible paths within the graph; $\doptenergy[k,p]$ (and possibly also $\nmaxmin[k]$) tends to divide the graph into $k$ roughly equal pieces, preserving highly connected parts of the graph; while the optimal partition for $\dmaxmin[k]$ tends to cut through the highly connected parts. It would be worthwhile to investigate this systematically and try to provide a rigorous basis for these heuristic claims.

We start with a simple criterion for identifying optimal partitions for the Neumann problem in some cases.

\begin{lemma}
\label{lem:neumann-identify-observation}
Suppose $\Graph$ has total length $L>0$ and $A$ is any set of $k$-partitions of $\Graph$. Suppose there exists a partition $\parti^\ast \in A$ such that
\begin{displaymath}
	\nenergy[\infty] (\parti^\ast) = \frac{\pi^2 k^2}{L^2}.
\end{displaymath}
Then
\begin{equation}
\label{eq:nenergy-min-a}
	\nenergy[\infty] (\parti^\ast) = \inf \left\{\nenergy[\infty] (\parti) : \parti \in A \right\},
\end{equation}
$\parti^\ast$ is exhaustive, and the clusters $\Graph_1,\ldots,\Graph_k$ of $\parti^\ast$ are all path graphs (that is, intervals) of length $L/k$. In particular, $\parti^\ast$ is also a minimising partition for $\noptenergyloose[k,p]$, for all $p \in (0,\infty]$.
\end{lemma}

We will generally take $A=\mathfrak{R}_k$ to be the set of exhaustive rigid $k$-partitions, in which case the infimum in \eqref{eq:nenergy-min-a} is, by definition, $\noptenergy[k,\infty]$, but the result holds for any set $A$ of $k$-partitions. There is also a corresponding statement for Dirichlet minimal $k$-partitions, namely that the same conclusion holds if $\denergy[\infty] (\parti^\ast) = \pi^2 k^2/(4L^2)$, but in this case each cluster must be a path graph of length $L/k$ with a Dirichlet vertex at one endpoint, which in particular requires $\Graph$ to be an equilateral $k$-star.

\begin{proof}[Proof of Lemma~\ref{lem:neumann-identify-observation}]
Let $\parti \in A$ be a $k$-partition with clusters $\Graph_1,\ldots,\Graph_k$. Then the sharp form of Nicaise' inequality (Theorem~\ref{thm:nicaise}) implies that $\mu_2 (\Graph_i) \geq \pi^2/|\Graph_i|^2$ for all $i=1,\ldots,k$, with equality if and only if $\Graph_i$ is a path graph. The claim now follows immediately from the definition of $\nenergy[\infty]$.
\end{proof}

\begin{example}
\label{ex:reinforced-loop}
Consider the graph $\Graph$ depicted in Figure~\ref{fig:comparison-base}, consisting of two longer and four shorter edges, the former of length $1$ and the latter of length $a$, which we imagine to be considerably smaller than $1$. One may visualise $\Graph$ as a loop with two short equal ``reinforcing'' edges placed at antipodal points. On $\Graph$, we will consider the respective partitions achieving the four quantities $\doptenergy[2,\infty]$, $\noptenergy[2,\infty]$, $\dmaxmin[2]$ and $\nmaxmin[2]$; by symmetry considerations, the two clusters must always have equal energies and hence be congruent to each other. Since the purpose of this example is essentially heuristic, we will not give detailed proofs of our claims. This would be possible, albeit tedious, via case-bashing arguments based for example on the advanced surgery techniques of \cite{BeKeKuMu18} together with extensive use of the symmetry of the problems.
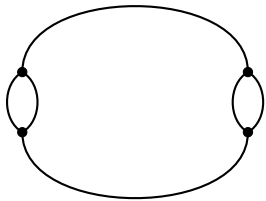
\begin{figure}[H]
\begin{tikzpicture}
\coordinate (a) at (0,0.4);
\coordinate (b) at (3,0.4);
\coordinate (c) at (0,-0.4);
\coordinate (d) at (3,-0.4);
\draw[thick,bend left=90]  (a) edge (b);
\draw[thick,bend right=90]  (c) edge (d);
\draw[thick,bend left=60] (a) edge (c);
\draw[thick,bend left=60] (b) edge (d);
\draw[thick,bend right=60] (a) edge (c);
\draw[thick,bend right=60] (b) edge (d);
\draw[fill] (0,0.4) circle (1.75pt);
\draw[fill] (3,0.4) circle (1.75pt);
\draw[fill] (0,-0.4) circle (1.75pt);
\draw[fill] (3,-0.4) circle (1.75pt);
\end{tikzpicture}
\caption{The graph $\Graph$. We assume the longer edges both have length $1$, and the four shorter edges all have the same length $a<1$.}
\label{fig:comparison-base}
\end{figure}
We start by considering the minimal partitions; the optimisers are depicted in Figure~\ref{fig:min-max-comparison}. For $\doptenergy[2,\infty]$, we are seeking the nodal partition corresponding to $\mu_2(\Graph)$ (see Proposition~\ref{prop:link-to-second-eigenvalue}). An argument similar to the ones in \cite[\S~5.1]{BeKeKuMu18} shows that the eigenfunction of $\mu_2(\Graph)$ must be invariant with respect to permutation of the longer two edges, and of each of the shorter ones within each set of two adjacent short edges. This leaves only the possibilities depicted in Figure~\ref{fig:min-max-comparison} (left) and Figure~\ref{fig:max-min-comparison} (left). Either a direct calculation involving the secular equations or an argument analogous to \cite[Proposition~5.10]{BeKeKuMu18} shows that the former has lower energy and hence corresponds to $\mu_2$ and $\doptenergy[2,\infty]$. For $\noptenergy[2,\infty]$, we see that it is possible to partition $\Graph$ into two equal path graphs as shown in Figure~\ref{fig:min-max-comparison} (right); hence, by Lemma~\ref{lem:neumann-identify-observation}, this is the optimal partition.
\begin{figure}[H]
\begin{tikzpicture}
\coordinate (a) at (0,0.4);
\coordinate (b) at (3.5,0.4);
\coordinate (c) at (0,-0.4);
\coordinate (d) at (3.5,-0.4);
\coordinate (e) at (1.5,1.2);
\coordinate (f) at (2,1.2);
\coordinate (g) at (1.5,-1.2);
\coordinate (h) at (2,-1.2);
\draw[thick,bend left=30]  (a) edge (e);
\draw[thick,bend left=30]  (f) edge (b);
\draw[thick,bend right=30]  (c) edge (g);
\draw[thick,bend right=30]  (h) edge (d);
\draw[thick,bend left=60] (a) edge (c);
\draw[thick,bend left=60] (b) edge (d);
\draw[thick,bend right=60] (a) edge (c);
\draw[thick,bend right=60] (b) edge (d);
\draw[fill] (0,0.4) circle (1.75pt);
\draw[fill] (3.5,0.4) circle (1.75pt);
\draw[fill] (0,-0.4) circle (1.75pt);
\draw[fill] (3.5,-0.4) circle (1.75pt);
\filldraw[thick, draw=black, fill=white] (1.5,1.2) circle (1.75pt);
\filldraw[thick, draw=black, fill=white] (2,1.2) circle (1.75pt);
\filldraw[thick, draw=black, fill=white] (1.5,-1.2) circle (1.75pt);
\filldraw[thick, draw=black, fill=white] (2,-1.2) circle (1.75pt);
\draw[thick,dashed] (1.75,1.5) -- (1.75,0.9);
\draw[thick,dashed] (1.75,-1.5) -- (1.75,-0.9);
\draw[fill] (5,-0.4) circle (1.75pt);
\draw[thick,bend left=60] (5,-0.4) edge (5,0.4);
\draw[fill] (5,0.4) circle (1.75pt);
\draw[thick,bend left=90] (5,0.4) edge (8.5,0.4);
\draw[fill] (8.5,0.4) circle (1.75pt);
\draw[thick,bend right=60] (8.5,0.4) edge (8.5,-0.4);
\draw[fill] (8.5,-0.4) circle (1.75pt);
\draw[fill] (5.5,-0.4) circle (1.75pt);
\draw[thick,bend right=60] (5.5,-0.4) edge (5.5,0.4);
\draw[fill] (5.5,0.4) circle (1.75pt);
\draw[thick,bend right=90] (5.5,-0.4) edge (9,-0.4);
\draw[fill] (9,-0.4) circle (1.75pt);
\draw[thick,bend right=60] (9,-0.4) edge (9,0.4);
\draw[fill] (9,0.4) circle (1.75pt);
\draw[thick,dashed] (5.25,-0.7) -- (5.25,0.7);
\draw[thick,dashed] (8.75,-0.7) -- (8.75,0.7);
\end{tikzpicture}
\qquad
\begin{tikzpicture}
\draw[fill] (5,-0.2) circle (1.75pt);
\draw[fill] (5,0.6) circle (1.75pt);
\draw[fill] (5.5,-0.4) circle (1.75pt);
\draw[fill] (5.5,0.6) circle (1.75pt);
\draw[fill] (8.5,0.6) circle (1.75pt);
\draw[fill] (8.5,-0.4) circle (1.75pt);
\draw[fill] (9,-0.4) circle (1.75pt);
\draw[fill] (9,0.4) circle (1.75pt);
\draw[thick,bend left=60] (5,-0.2) edge (5,0.6);
\draw[thick,bend left=90] (5,0.6) edge (8.5,0.6);
\draw[thick,bend right=60] (9,0.4) edge (8.5,-0.4);
\draw[thick,bend right=60] (5,-0.2) edge (5.5,0.6);
\draw[thick,bend right=90] (5.5,-0.4) edge (9,-0.4);
\draw[thick,bend right=60] (9,-0.4) edge (9,0.4);
\draw[thick,dashed] (5.25,.2) -- (5.25,.9);
\draw[thick,dashed] (8.75,-0.7) -- (8.75,0);
\draw[thick,dashed] (4.9,-0.65) -- (5.8,.1);
\draw[thick,dashed] (9.1,.75) -- (8.2,.1);
\end{tikzpicture}

\caption{Rigid 2-partitions of $\Graph$ realising $\doptenergy[2,\infty](\Graph)$ (left) and $\noptenergy[2,\infty](\Graph)=\noptenergyloose[2,\infty](\Graph)$ (middle). There exists a further loose partition realising $\noptenergyloose[2,\infty](\Graph)$ (right). Black dots denote Neumann-Kirchhoff vertices, white dots denote Dirichlet vertices.}\label{fig:min-max-comparison}
\end{figure}
If we turn to maximal partitions, in the Dirichlet case the same argument as in \cite[Proposition~5.10]{BeKeKuMu18} implies that among all Dirichlet candidates, the eigenvalue $\lambda_1$ is largest when the shorter edges are as close to the Dirichlet (cut) set and as equal as possible, yielding Figure~\ref{fig:max-min-comparison} (left). To maximise the Neumann eigenvalues, the doubled edges should be located as close to, and as symmetrically about, the zero set of the eigenfunctions as possible, corresponding to Figure~\ref{fig:max-min-comparison} (right). We omit the details, which closely follow the principles laid out in \cite[\S~5]{BeKeKuMu18}.
\begin{figure}[h]
\begin{tikzpicture}
\coordinate (i) at (0,0.65);
\coordinate (j) at (3,0.65);
\coordinate (k) at (0,-0.65);
\coordinate (l) at (3,-0.65);
\draw[thick,bend left=90]  (i) edge (j);
\draw[thick,bend right=90]  (k) edge (l);
\draw[fill] (0,0.65) circle (1.75pt);
\draw[fill] (3,0.65) circle (1.75pt);
\draw[fill] (0,-0.65) circle (1.75pt);
\draw[fill] (3,-0.65) circle (1.75pt);
\draw[thick, bend left=45] (-0.3,0.25) edge (i);
\draw[thick, bend right=45] (0.3,0.25) edge (i);
\draw[thick, bend left=45] (k) edge (-0.3,-0.25);
\draw[thick, bend right=45] (k) edge (0.3,-0.25);
\filldraw[thick, draw=black, fill=white] (-0.3,0.25) circle (1.75pt);
\filldraw[thick, draw=black, fill=white] (0.3,0.25) circle (1.75pt);
\filldraw[thick, draw=black, fill=white] (-0.3,-0.25) circle (1.75pt);
\filldraw[thick, draw=black, fill=white] (0.3,-0.25) circle (1.75pt);
\draw[thick, bend left=45] (2.7,0.25) edge (j);
\draw[thick, bend right=45] (3.3,0.25) edge (j);
\draw[thick, bend left=45] (l) edge (2.7,-0.25);
\draw[thick, bend right=45] (l) edge (3.3,-0.25);
\filldraw[thick, draw=black, fill=white] (2.7,0.25) circle (1.75pt);
\filldraw[thick, draw=black, fill=white] (3.3,0.25) circle (1.75pt);
\filldraw[thick, draw=black, fill=white] (2.7,-0.25) circle (1.75pt);
\filldraw[thick, draw=black, fill=white] (3.3,-0.25) circle (1.75pt);
\draw[thick, dashed] (-0.6,0) -- (0.6,0);
\draw[thick, dashed] (2.4,0) -- (3.6,0);
\coordinate (a) at (5,0.4);
\coordinate (b) at (8.5,0.4);
\coordinate (c) at (5,-0.4);
\coordinate (d) at (8.5,-0.4);
\coordinate (e) at (6.5,1.2);
\coordinate (f) at (7,1.2);
\coordinate (g) at (6.5,-1.2);
\coordinate (h) at (7,-1.2);
\draw[thick,bend left=30]  (a) edge (e);
\draw[thick,bend left=30]  (f) edge (b);
\draw[thick,bend right=30]  (c) edge (g);
\draw[thick,bend right=30]  (h) edge (d);
\draw[thick,bend left=60] (a) edge (c);
\draw[thick,bend left=60] (b) edge (d);
\draw[thick,bend right=60] (a) edge (c);
\draw[thick,bend right=60] (b) edge (d);
\draw[fill] (5,0.4) circle (1.75pt);
\draw[fill] (8.5,0.4) circle (1.75pt);
\draw[fill] (5,-0.4) circle (1.75pt);
\draw[fill] (8.5,-0.4) circle (1.75pt);
\draw[fill] (6.5,1.2) circle (1.75pt);
\draw[fill] (7,1.2) circle (1.75pt);
\draw[fill] (6.5,-1.2) circle (1.75pt);
\draw[fill] (7,-1.2) circle (1.75pt);
\draw[thick,dashed] (6.75,1.5) -- (6.75,0.9);
\draw[thick,dashed] (6.75,-1.5) -- (6.75,-0.9);
\end{tikzpicture}
\caption{Proper partitions of $\Graph$ realising $\dmaxmin[2](\Graph)$ (left) and $\nmaxmin[2](\Graph)$ (right).}
\label{fig:max-min-comparison}
\end{figure}
\end{example}

\begin{example}
\label{ex:pumpkin-example}
Our second example will be the \emph{equilateral $6$-pumpkin} $\mathcal{H}$ depicted in Figure~\ref{fig:pumpkin-example}: it is the graph with two vertices and six parallel edges running between them, all of the same length (say, length $1$ each).
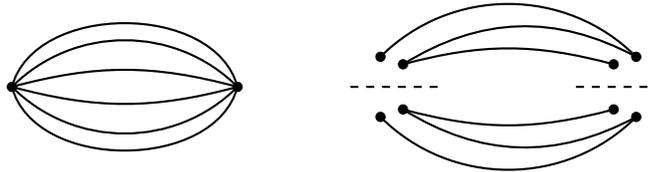
\begin{figure}[H]
\begin{tikzpicture}
\coordinate (a) at (0,0);
\coordinate (b) at (3,0);
\draw[fill] (0,0) circle (1.75pt);
\draw[fill] (3,0) circle (1.75pt);
\draw[thick, bend left=15] (a) edge (b);
\draw[thick, bend left=45] (a) edge (b);
\draw[thick, bend left=75] (a) edge (b);
\draw[thick, bend right=15] (a) edge (b);
\draw[thick, bend right=45] (a) edge (b);
\draw[thick, bend right=75] (a) edge (b);
\draw[fill] (4.9,0.4) circle (1.75pt);
\draw[thick, bend left=45] (4.9,0.4) edge (8.3,0.4);
\draw[fill] (8.3,0.4) circle (1.75pt);
\draw[thick, bend left=30] (5.2,0.3) edge (8.3,0.4);
\draw[fill] (5.2,0.3) circle (1.75pt);
\draw[thick, bend left=15] (5.2,0.3) edge (8,0.3);
\draw[fill] (8,0.3) circle (1.75pt);
\draw[fill] (4.9,-0.4) circle (1.75pt);
\draw[thick, bend right=45] (4.9,-0.4) edge (8.3,-0.4);
\draw[fill] (8.3,-0.4) circle (1.75pt);
\draw[thick, bend right=30] (5.2,-0.3) edge (8.3,-0.4);
\draw[fill] (5.2,-0.3) circle (1.75pt);
\draw[thick, bend right=15] (5.2,-0.3) edge (8,-0.3);
\draw[fill] (8,-0.3) circle (1.75pt);
\draw[thick, dashed] (4.5,0) -- (5.7,0);
\draw[thick, dashed] (7.5,0) -- (8.7,0);
\end{tikzpicture}
\caption{An equilateral ``$6$-pumpkin'' (left) and a rigid partition realising $\noptenergy[2,\infty]$ on it (right). Again, no loose partition achieves an energy lower than $\noptenergy[2,\infty]$.}
\label{fig:pumpkin-example}
\end{figure}
In Figure~\ref{fig:pumpkin-example} we also see how $\mathcal{H}$ can be partitioned into two path graphs of length $3$ each; by Lemma~\ref{lem:neumann-identify-observation}, this is an optimal $2$-partition, corresponding to $\noptenergy[2,\infty](\mathcal{H})$. This is thus another illustration of the assertion that ``$\noptenergy[k,\infty]$ tends to seek out paths embedded in the graph''. In terms of Dirichlet minimal partitions, to find partitions achieving $\doptenergy[2,\infty](\mathcal{H})$, by Proposition~\ref{prop:link-to-second-eigenvalue} we merely have to consider the nodal patterns of eigenfunctions corresponding to $\mu_2 (\mathcal{H})$; these are depicted in Figure~\ref{fig:pumpkin-dirichlet}. We observe explicitly that this is an (easy) example of an intrinsic non-uniqueness: there are two completely different partitions of $\mathcal{H}$ which both yield $\doptenergy[2,\infty](\mathcal{H})$.
\begin{figure}[H]
\begin{tikzpicture}
\coordinate (a) at (-0.5,0);
\coordinate (b) at (3,0);
\draw[thick, bend left=12] (a) edge (1,0.3);
\draw[thick, bend left=24] (a) edge (1,0.7);
\draw[thick, bend left=36] (a) edge (1,1);
\draw[thick, bend right=12] (a) edge (1,-0.3);
\draw[thick, bend right=24] (a) edge (1,-0.7);
\draw[thick, bend right=36] (a) edge (1,-1);
\draw[thick, bend left=12] (1.5,0.3) edge (b);
\draw[thick, bend left=24] (1.5,0.7) edge (b);
\draw[thick, bend left=36] (1.5,1) edge (b);
\draw[thick, bend right=12] (1.5,-0.3) edge (b);
\draw[thick, bend right=24] (1.5,-0.7) edge (b);
\draw[thick, bend right=36] (1.5,-1) edge (b);
\draw[fill] (-0.5,0) circle (1.75pt);
\draw[fill] (3,0) circle (1.75pt);
\filldraw[thick, draw=black, fill=white] (1,1) circle (1.75pt);
\filldraw[thick, draw=black, fill=white] (1,0.7) circle (1.75pt);
\filldraw[thick, draw=black, fill=white] (1,0.3) circle (1.75pt);
\filldraw[thick, draw=black, fill=white] (1,-0.3) circle (1.75pt);
\filldraw[thick, draw=black, fill=white] (1,-0.7) circle (1.75pt);
\filldraw[thick, draw=black, fill=white] (1,-1) circle (1.75pt);
\filldraw[thick, draw=black, fill=white] (1.5,1) circle (1.75pt);
\filldraw[thick, draw=black, fill=white] (1.5,0.7) circle (1.75pt);
\filldraw[thick, draw=black, fill=white] (1.5,0.3) circle (1.75pt);
\filldraw[thick, draw=black, fill=white] (1.5,-0.3) circle (1.75pt);
\filldraw[thick, draw=black, fill=white] (1.5,-0.7) circle (1.75pt);
\filldraw[thick, draw=black, fill=white] (1.5,-1) circle (1.75pt);
\draw[thick, dashed] (1.25,1.3) -- (1.25,-1.3);
\draw[thick, bend left=45] (4.9,0.4) edge (8.3,0.4);
\draw[thick, bend left=30] (5.2,0.3) edge (8.3,0.4);
\draw[thick, bend left=15] (5.2,0.3) edge (8,0.3);
\draw[thick, bend right=45] (4.9,-0.4) edge (8.3,-0.4);
\draw[thick, bend right=30] (5.2,-0.3) edge (8.3,-0.4);
\draw[thick, bend right=15] (5.2,-0.3) edge (8,-0.3);
\filldraw[thick, draw=black, fill=white] (4.9,0.4) circle (1.75pt);
\filldraw[thick, draw=black, fill=white] (8.3,0.4) circle (1.75pt);
\filldraw[thick, draw=black, fill=white] (5.2,0.3) circle (1.75pt);
\filldraw[thick, draw=black, fill=white] (8,0.3) circle (1.75pt);
\filldraw[thick, draw=black, fill=white] (4.9,-0.4) circle (1.75pt);
\filldraw[thick, draw=black, fill=white] (8.3,-0.4) circle (1.75pt);
\filldraw[thick, draw=black, fill=white] (5.2,-0.3) circle (1.75pt);
\filldraw[thick, draw=black, fill=white] (8,-0.3) circle (1.75pt);
\draw[thick, dashed] (4.5,0) -- (5.7,0);
\draw[thick, dashed] (7.5,0) -- (8.7,0);
\end{tikzpicture}
\caption{Two different partitions realising $\doptenergy[2,\infty]$ on the equilateral $6$-pumpkin. The partition on the left is internally connected, the one on the right is not as the Dirichlet vertices disconnect the clusters.}
\label{fig:pumpkin-dirichlet}
\end{figure}
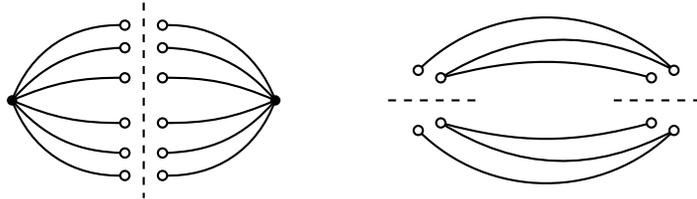
\end{example}

\begin{example}\label{ex:spantree}
Let us finally present an example of a graph with different rigid and loose minimal partitions: an equilateral dumbbell graph $\Graph$. The unique minimiser for $\noptenergy[2,\infty](\Graph)$ and the different minimiser for $\noptenergyloose[2,\infty](\Graph)$ are shown in Figure~\ref{fig:dumbbell-2-partition}. The former is also a minimiser for $\doptenergy[2,p](\Graph)$, $p\in (0,\infty]$.
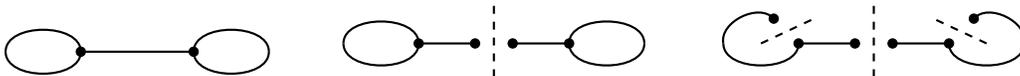
\begin{figure}[H]
\begin{tikzpicture}
\coordinate (a) at (0,0);
\coordinate (b) at (1,0);
\coordinate (c) at (2.5,0);
\coordinate (d) at (3.5,0);
\draw[fill] (b) circle (1.75pt);
\draw[fill] (c) circle (1.75pt);
\draw[thick,bend left=90] (a) edge (b);
\draw[thick,bend right=90] (a) edge (b);
\draw[thick,bend left=90] (c) edge (d);
\draw[thick,bend right=90] (c) edge (d);
\draw[thick] (b) -- (c);
\end{tikzpicture}
\qquad 
\begin{tikzpicture}
\coordinate (a) at (0,0);
\coordinate (b) at (1,0);
\coordinate (c1) at (1.75,0);
\coordinate (c2) at (2.25,0);
\coordinate (d) at (3,0);
\coordinate (e) at (4,0);
\draw[fill] (c1) circle (1.75pt);
\draw[fill] (c2) circle (1.75pt);
\draw[fill] (b) circle (1.75pt);
\draw[fill] (d) circle (1.75pt);
\draw[thick,bend left=90] (a) edge (b);
\draw[thick,bend right=90] (a) edge (b);
\draw[thick] (b) edge (c1);
\draw[thick] (c2) edge (d);
\draw[thick,bend left=90] (e) edge (d);
\draw[thick,bend right=90] (e) edge (d);
\draw[thick,dashed] (2,.5) -- (2,-.5);
\end{tikzpicture}
\qquad
\begin{tikzpicture}
\coordinate (a) at (0,0);
\coordinate (b1) at (1,0);
\coordinate (b2) at (0.67,0.33);
\coordinate (c1) at (1.75,0);
\coordinate (c2) at (2.25,0);
\coordinate (d1) at (3,0);
\coordinate (d2) at (3.33,0.33);
\coordinate (e) at (4,0);
\draw[fill] (b1) circle (1.75pt);
\draw[fill] (b2) circle (1.75pt);
\draw[fill] (c1) circle (1.75pt);
\draw[fill] (c2) circle (1.75pt);
\draw[fill] (d1) circle (1.75pt);
\draw[fill] (d2) circle (1.75pt);
\draw[thick,bend left=75] (a) edge (b2);
\draw[thick,bend right=90] (a) edge (b1);
\draw[thick] (b1) edge (c1);
\draw[thick] (c2) edge (d1);
\draw[thick,bend left=90] (e) edge (d1);
\draw[thick,bend right=75] (e) edge (d2);
\draw[thick,dashed] (2,0.5) -- (2,-0.5);
\draw[thick,dashed] (0.5,0) -- (1.25,0.35);
\draw[thick,dashed] (3.5,0) -- (2.75,0.35);
\end{tikzpicture}\\[10pt]
\caption{A dumbbell graph, its rigid Neumann/Dirichlet minimal partition and a loose Neumann minimal partition}\label{fig:dumbbell-2-partition}
\end{figure}

\end{example}

\subsection{Exhaustive versus non-exhaustive partitions}
\label{sec:exhaustive-issue}

We now give a few remarks on what can be expected if we allow, or disallow, non-exhaustive partitions.

\begin{example}
\label{ex:n-exhaustion}
We give a basic example to show that if instead of $\noptenergy[k,p]$ we consider the corresponding minimisation problem among all  non-exhaustive rigid partitions, then there may be no exhaustive minimising partition. If we take $k=1$ and any $p \in (0,\infty]$, then the claim is that there exist a graph $\Graph$ and a proper subset (i.e.\ a non-exhaustive one-partition) $\Omega_1 \subset \Graph$, as well as a graph $\Graph_1$ associated with $\Omega_1$, such that $\mu_2 (\Graph_1) < \mu_2 (\Graph)$. Take $\Graph$ to be a loop of length $1$ and $\Omega_1$ to be a connected subset of it of length $1-\varepsilon$, so that $\Graph_1$ is a path graph of length $1-\varepsilon$. Then
\begin{displaymath}
	\mu_2 (\Graph_1) = \frac{\pi^2}{(1-\varepsilon)^2} < 4\pi^2 = \mu_2 (\Graph)
\end{displaymath}
as long as $\varepsilon \in (0,1/2)$. In such cases, $\parti := \{\Graph \}$ is not the optimal $1$-partition of $\Graph$ if non-exhaustive partitions are allowed.
\end{example}

\begin{example}
\label{ex:n-exhaustion-2}
Here is another example akin to Example~\ref{ex:n-exhaustion}, but where we can more clearly see the similar role to non-exhaustive partitions played by non-rigid (but exhaustive) ones. We take $\Graph$ to be an equilateral pumpkin graph on three edges of length one each (cf.~Example~\ref{ex:pumpkin-example}) and, analogously to the previous example, seek a $1$-partition minimising $\noptenergy[1,p]$, $p \in (0,\infty]$. Allowing exhaustive loose partitions, we can cut through the vertices of $\Graph$ to produce a path graph $\mathcal{I}$, much as was done in Figure~\ref{fig:pumpkin-example} (right); see Figure~\ref{fig:3-pumpkin-1-partition} (left). Lemma~\ref{lem:neumann-identify-observation} (with $A=\mathfrak{P}_k$) shows that $\parti^\ast = \{\mathcal{I}\}$, and not $\parti = \{\Graph\}$, is the optimal partition for $\noptenergyloose[1,p] (\Graph)$, $p \in (0,\infty]$. Indeed, the former has energy $\nenergy[p](\parti^\ast) = \mu_2 (\mathcal{I}) = \pi^2/9$, while the latter has energy $\nenergy[p](\parti^\ast) = \mu_2 (\Graph) = \pi^2$.
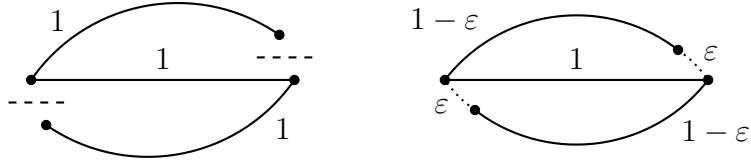
\begin{figure}[H]
\begin{tikzpicture}
\coordinate (a) at (-0.5,0);
\coordinate (b) at (3,0);
\draw[fill] (-0.5,0) circle (1.75pt);
\draw[fill] (3,0) circle (1.75pt);
\draw[thick] (a) -- (b);
\draw[thick, bend left=45] (a) edge (2.8,0.6);
\draw[fill] (2.8,0.6) circle (1.75pt);
\draw[thick, bend left=45] (b) edge (-0.3,-0.6);
\draw[fill] (-0.3,-0.6) circle (1.75pt);
\draw[thick, dashed] (2.5,0.3) -- (3.3,0.3);
\draw[thick, dashed] (-0.8,-0.3) -- (0,-0.3);
\node at (1.25,0) [anchor=south] {$1$};
\node at (0.1,0.5) [anchor=south east] {$1$};
\node at (2.6,-0.65) [anchor=west] {$1$};
\coordinate (c) at (5,0);
\coordinate (d) at (8.5,0);
\draw[fill] (5,0) circle (1.75pt);
\draw[fill] (8.5,0) circle (1.75pt);
\draw[thick] (c) -- (d);
\draw[thick, bend left=45] (c) edge (8.1,0.4);
\draw[thick, dotted, bend left=9] (8.1,0.4) edge (d);
\draw[fill] (8.1,0.4) circle (1.75pt);
\draw[thick, bend left=45] (d) edge (5.4,-0.4);
\draw[thick, dotted, bend left=9] (5.4,-0.4) edge (c);
\draw[fill] (5.4,-0.4) circle (1.75pt);
\node at (6.75,0) [anchor=south] {$1$};
\node at (5.6,0.5) [anchor=south east] {$1-\varepsilon$};
\node at (8.3,0.35) [anchor=west] {$\varepsilon$};
\node at (5.2,-0.35) [anchor=east] {$\varepsilon$};
\node at (8,-0.7) [anchor=west] {$1-\varepsilon$};
\end{tikzpicture}
\caption{A minimal $1$-partition of the 3-pumpkin of Example~\ref{ex:n-exhaustion-2} among loose but exhaustive partitions (left) and a non-exhaustive but rigid $1$-partition approximating it (right; the dotted lines represent the parts of the graph excluded from the cluster support). The numbers give the lengths of the respective edges.}
\label{fig:3-pumpkin-1-partition}
\end{figure}
But now consider rigid non-exhaustive $1$-partitions of $\Graph$. By removing an arbitrarily small piece of edge near each of the vertices, we can produce a partition consisting of a single path graph of length arbitrarily close to $3$; see Figure~\ref{fig:3-pumpkin-1-partition} (right). As the piece removed becomes smaller, we have convergence of the length to $3$, and thus convergence of the eigenvalue (equally, the energy of the partition) to $\pi^2/9 = \nenergy[p](\parti^\ast)$.
\end{example}

The preceding example suggests that in the natural case, infimising over non-exhaustive partitions is equivalent to minimising over loose exhaustive partitions. Indeed, whenever, on forming an optimal partition, one wishes to detach a single edge from a vertex (as was done in Example~\ref{ex:n-exhaustion-2}), one can do this either by allowing loose partitions or infimising over non-exhaustive ones. At work here are two ``surgery'' principles: cutting through vertices decreases $\mu_2$, as does lengthening pendant edges (see, e.g., \cite[Lemma~2.3]{KeKuMaMu16}). However, in general the two are \emph{not} equivalent, as Example~\ref{ex:matthias-exhaustive-example} below shows. Since allowing loose partitions in general seems to give more freedom to cut through vertices than allowing non-exhaustive rigid ones, we still expect the following general principle to hold. Since it would take us too far afield to prove the claim in the current context, we record it as a conjecture for future work.

\begin{conjecture}
\label{conj:non-admissible-exhaustion}
Fix a graph $\Graph$, a number $k\geq 1$ and $p \in (0,\infty]$. Then the quantity
\begin{equation}
\label{eq:non-exhaustive-optimiser}
	\inf \left\{ \nenergy[p](\parti): \parti \text{ is a non-exhaustive but rigid $k$-partition of } \Graph \right\}
\end{equation}
is no smaller than
\begin{equation}
\label{eq:non-admissible-optimiser}
	\noptenergyloose[k,p] = \inf \left\{ \nenergy[p](\parti): \parti \text{ is an exhaustive loose $k$-partition of } \Graph \right\}.
\end{equation}
\end{conjecture}

\begin{example}
\label{ex:matthias-exhaustive-example}
We sketch an example of a graph $\Graph$, with $k=2$ and $p=\infty$, where the quantity in \eqref{eq:non-admissible-optimiser} is strictly smaller than the one in \eqref{eq:non-exhaustive-optimiser}, which in turn is strictly smaller than $\noptenergy[2,\infty] (\Graph)$. We take $\Graph$ to be the ``dumbbell''-type graph depicted in Figure~\ref{fig:double-dumbbell}, consisting of a bridge (``handle'') $e_0$ of length $1$, with a chain of two loops of total length $\varepsilon > 0$ attached at each end of $e_0$; thus $\Graph$ has total length $1+2\varepsilon$.
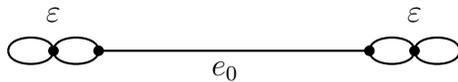
\begin{figure}[H]
\begin{tikzpicture}[scale=1.2]
\coordinate (a) at (-2.5,0);
\coordinate (b) at (-2,0);
\coordinate (c) at (-1.5,0);
\coordinate (d) at (1.5,0);
\coordinate (e) at (2,0);
\coordinate (f) at (2.5,0);
\draw[thick] (c) -- (d);
\draw[thick,bend left=90]  (a) edge (b);
\draw[thick,bend right=90]  (a) edge (b);
\draw[thick,bend left=90] (b) edge (c);
\draw[thick,bend right=90] (b) edge (c);
\draw[thick,bend left=90] (d) edge (e);
\draw[thick,bend right=90] (d) edge (e);
\draw[thick,bend left=90] (e) edge (f);
\draw[thick,bend right=90] (e) edge (f);
\draw[fill] (b) circle (1.5pt);
\draw[fill] (c) circle (1.5pt);
\draw[fill] (d) circle (1.5pt);
\draw[fill] (e) circle (1.5pt);
\node at (-0.1,0) [anchor=north] {$e_0$};
\node at (-2,0.2) [anchor=south] {$\varepsilon$};
\node at (2,0.2) [anchor=south] {$\varepsilon$};
\end{tikzpicture}
\caption{A dumbbell with ``double weights''. The \emph{handle} $e_0$ has length $1$ and the total length of each of the figure-8 ``weights'' is $\varepsilon$.}
\label{fig:double-dumbbell}
\end{figure}

Firstly, a short symmetry argument shows that $\noptenergy[2,\infty] (\Graph)$ is given by the partition bisecting $\Graph$ at the midpoint of $e_0$.

Secondly, if we allow loose partitions, then we can partition $\Graph$ into two equal path graphs of length $1/2+\varepsilon$ each, as depicted in Figure~\ref{fig:double-dumbbell-partitions} (left). Since this gives the smallest possible energy for a $2$-partition, by Lemma~\ref{lem:neumann-identify-observation} it yields the minimum, and the minimum is only achieved by two equal path graphs. (In particular, $\noptenergyloose[2,\infty] (\Graph) < \noptenergy[2,\infty] (\Graph)$.)
\begin{figure}[H]
\begin{tikzpicture}
\coordinate (a) at (-2.8,0);
\coordinate (b) at (-2.3,0);
\coordinate (c) at (-1.8,0);
\coordinate (d) at (-0.3,0);
\coordinate (e) at (0.2,0);
\coordinate (f) at (1.7,0);
\coordinate (g) at (2.2,0);
\coordinate (h) at (2.7,0);
\draw[thick] (c) -- (d);
\draw[thick] (e) -- (f);
\draw[thick,bend left=90]  (a) edge (-2.3,0.3);
\draw[thick,bend right=90]  (a) edge (b);
\draw[thick,bend left=90] (-2.3,0.3) edge (-1.8,0.3);
\draw[thick,bend right=90] (b) edge (c);
\draw[fill] (b) circle (1.5pt);
\draw[fill] (c) circle (1.5pt);
\draw[fill] (d) circle (1.5pt);
\draw[fill] (e) circle (1.5pt);
\draw[fill] (f) circle (1.5pt);
\draw[fill] (g) circle (1.5pt);
\draw[fill] (-2.3,0.3) circle (1.5pt);
\draw[fill] (-1.8,0.3) circle (1.5pt);
\draw[thick,dashed] (-0.05,0.4) -- (-0.05,-0.4);
\draw[thick,dashed] (-2.45,0.15) -- (-2.15,0.15);
\draw[thick,dashed] (-1.95,0.15) -- (-1.65,0.15);
\draw[thick,bend right=90] (f) edge (g);
\draw[thick,bend right=90] (g) edge (h);
\draw[thick,bend right=90] (h) edge (2.2,0.3);
\draw[thick,bend right=90] (2.2,0.3) edge (1.7,0.3);
\draw[fill] (2.2,0.3) circle (1.5pt);
\draw[fill] (1.7,0.3) circle (1.5pt);
\draw[thick,dashed] (2.35,0.15) -- (2.05,0.15);
\draw[thick,dashed] (1.85,0.15) -- (1.55,0.15);
\coordinate (i) at (4,0);
\coordinate (j) at (4.5,0);
\coordinate (k) at (5,0);
\coordinate (l) at (6.5,0);
\coordinate (m) at (7,0);
\coordinate (n) at (8.5,0);
\coordinate (o) at (9,0);
\coordinate (p) at (9.5,0);
\draw[thick] (k) -- (l);
\draw[thick] (m) -- (n);
\draw[fill] (j) circle (1.5pt);
\draw[fill] (k) circle (1.5pt);
\draw[fill] (l) circle (1.5pt);
\draw[fill] (m) circle (1.5pt);
\draw[fill] (n) circle (1.5pt);
\draw[fill] (o) circle (1.5pt);
\draw[thick,bend right=90] (i) edge (4.4,-0.3);
\draw[thick,bend right=90] (j) edge (k);
\draw[thick,bend right=90] (n) edge (o);
\draw[thick,bend right=90] (9.1,-0.3) edge (p);
\draw[thick,bend left=90] (i) edge (j);
\draw[thick,bend left=90] (o) edge (p);
\draw[thick,bend left=90] (j) edge (4.9,0.3);
\draw[thick,bend right=90] (o) edge (8.6,0.3);
\draw[fill] (4.9,0.3) circle (1.5pt);
\draw[fill] (8.6,0.3) circle (1.5pt);
\draw[thick,dashed] (6.75,0.4) -- (6.75,-0.4);
\draw[thick,dashed] (4.8,0.12) -- (5.15,0.2);
\draw[thick,dashed] (8.7,0.12) -- (8.35,0.2);
\draw[fill] (4.4,-0.3) circle (1.5pt);
\draw[fill] (9.1,-0.3) circle (1.5pt);
\draw[thick,dashed] (4.3,-0.12) -- (4.65,-0.25);
\draw[thick,dashed] (9.2,-0.12) -- (8.85,-0.25);
\end{tikzpicture}
\caption{The optimal partition among all loose exhaustive ones (left) and a candidate for the infimum among all rigid but non-exhaustive ones (right).}
\label{fig:double-dumbbell-partitions}
\end{figure}
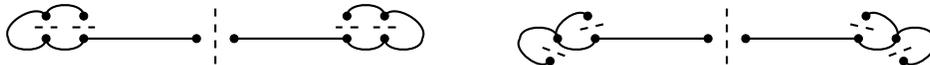

Finally, we consider rigid but non-exhaustive partitions. In this case, a topological argument shows that it is impossible to obtain two equal path graphs, since there is no way to cut through a degree-four vertex to obtain two degree-two vertices within a cluster support. Thus there is strict inequality between \eqref{eq:non-exhaustive-optimiser} and \eqref{eq:non-admissible-optimiser} in this case. However, if we take the partition in Figure~\ref{fig:double-dumbbell-partitions} (right), which is achievable as the limit of a sequence of non-exhaustive rigid partitions, then this partition is obtainable from the optimal one for $\noptenergy[2,\infty] (\Graph)$ by cutting through certain vertices of the latter. One may then show using the strictness statement in \cite[Theorem~3.4]{BeKeKuMu18} (in the form of Remark~3.5) that this partition has strictly lower energy than $\noptenergy[2,\infty] (\Graph)$; thus the infimum \eqref{eq:non-exhaustive-optimiser} is also lower. This completes the proof of the claimed chain of strict inequalities.
\end{example}

We shall finish with a different observation on non-exhaustive partitions. We saw in Example~\ref{ex:non-equi-3-star} that a partition $\parti^\ast$ achieving $\doptenergy[k,\infty]$ need not be an equipartition. However, if we allow non-exhaustive partitions, then we can always ``artificially generate'' a minimal equipartition by shrinking every cluster in $\parti^\ast$ whose first eigenvalue is too large. Similarly, by then discarding superfluous connected components in each cluster, we can guarantee that the resulting minimal equipartition consists only of internally connected clusters.

\begin{proposition}
\label{prop:exhaustive-shrinking}
Suppose $\parti = \{\Graph_1,\ldots,\Graph_k\}$ is an (exhaustive) rigid $k$-partition of $\Graph$ such that $\denergy[\infty](\parti) = \doptenergy[k,\infty](\Graph)$ for some $k\geq 2$. Then there exists another, possibly non-exhaustive but equilateral rigid $k$-partition $\parti' = \{\Graph_1',\ldots,\Graph_k'\}$ such that also $\denergy[\infty](\parti') = \doptenergy[k,\infty](\Graph)$. The partition $\parti'$ may additionally be chosen in such a way that the clusters $\Graph_1',\ldots,\Graph_k'$ are all internally connected.
\end{proposition}

\begin{proof}
It suffices to prove that if $\HGraph$ is any graph with a non-empty set of Dirichlet vertices $\VertexSet_D (\HGraph)$ and $\lambda \geq \lambda_1 (\HGraph;\VertexSet_D(\HGraph))$, then there exists a subgraph
 $\NewHGraph \subset \HGraph$ with boundary (equivalently, Dirichlet vertex set)
\begin{equation}
\label{eq:shrunk-subgraph}
	\partial \NewHGraph := \left(\NewHGraph \cap \overline{\HGraph \setminus \NewHGraph}\right) \cup \left(\NewHGraph \cap \VertexSet_D (\HGraph)
	\right)
\end{equation}
such that $\lambda_1 (\NewHGraph;\partial \NewHGraph) = \lambda$.\footnote{Throughout this proof, in accordance with Remark~\ref{rem:dirichlet-subset-graph-distinction} we will not distinguish between the clusters and the cluster supports of a partition.} Indeed, in this case, whenever $\lambda_1 (\Graph_i) < \denergy[\infty](\parti)$, we simply find some $\Graph_i' \subset \Graph_i$ such that $\lambda_1 (\Graph_i') = \denergy[\infty](\parti)$.

To this end, simply choose any fixed $z \in \VertexSet_D (\HGraph)$ and for $t\in [0,\max_{x \in \HGraph} \dist (x,z)]$ set
\begin{displaymath}
	\HGraph_t := \HGraph \setminus \{x \in \HGraph: \dist (x,z) < t \}
\end{displaymath}
(where $\dist$ denotes the Euclidean distance within $\HGraph$, defined such that paths may not pass through vertices in $\VertexSet_D (\HGraph)$, and $\partial \HGraph_t$ is defined as in \eqref{eq:shrunk-subgraph}). We claim that $t \mapsto \lambda_1 (\HGraph_t; \partial \NewHGraph_t)$ is continuous for $t\in [0,\max_{x \in \HGraph} \dist (x,z)]$ (although $\HGraph_t$ may not necessarily be connected). Assuming this claim, since also $|\HGraph_t| \to 0$ as $t \to \max_{x \in \HGraph} \dist (x,z)]$, we also have $\lambda_1 (\HGraph_t; \partial \NewHGraph_t) \to \infty)$ by \cite[Th\'eor\`eme~3.1]{Nic87}. The statement of the proposition thus follows.

To prove the claim, we simply observe that this is a special case of Lemma~\ref{lem:eig-convergence}, where the graphs need not be connected. Indeed, fix $t \in [0,\max_{x \in \HGraph} \dist (x,z)]$ and a sequence $t_n \to t$; then the conclusion follows with $\HGraph_t = \Graph_\infty$ and $\HGraph_{t_n} = \Graph_n$.

Finally, fix an arbitrary cluster $\Graph_i'$ and suppose that it is not internally connected. Without loss of generality, we regard all points in $\partial\Graph_i'$, which are equipped with a Dirichlet condition, as having degree one: then in particular $\Graph_i'$ is not connected. Let $\psi_i$ be any eigenfunction corresponding to $\lambda_1 (\Graph_i')$ (which may be multiple since $\Graph_i'$ is not connected; a minimal $2$-partition of an equilateral $3$-star gives an example), such that the support of $\psi_i$ is connected. Upon replacing $\Graph_i'$ by $\Graph_i^\dagger := \supp \psi_i$ and discarding $\Graph_i' \setminus \Graph_i^\dagger$ from the support of the partition, we have obtained an internally connected cluster whose energy is still equal to $\lambda_1 (\Graph_i') = \doptenergy[k,\infty] (\Graph)$. Repeating this process for all $i=1,\ldots,k$ yields an equilateral internally connected $k$-partition realising $\doptenergy[k,\infty] (\Graph)$.
\end{proof}

\begin{remark}
In the case of the Neumann problem, the situation is more complicated. Here we limit ourselves to the observation that an (exhaustive) Neumann minimal partition need not be an equipartition, even when $p=\infty$, and even \emph{generically}: i.e., there exist graphs for which the minimal partition is not an equipartition, and there is some $\varepsilon>0$ for which no perturbation of the edge lengths less than $\varepsilon$ will produce a graph whose (exhaustive) minimal partition is an equipartition.

As an example, take sufficiently small but fixed numbers $0 < \varepsilon_1 < \varepsilon_2 < \varepsilon_3 < \varepsilon_4$ and let $\Graph$ be the $4$-star having edges $e_1,\ldots,e_4$ of lengths $1+\varepsilon_1,\ldots,1+\varepsilon_4$, respectively. Then an easy but tedious case bash shows that the unique (up to relabelling) partition achieving $\noptenergy[2,\infty](\Graph)$ is given by $\Graph_1 = e_1 \cup e_4$, $\Graph_2 = e_2 \cup e_3$ (that is, whose only cut is through the central vertex). If the $\varepsilon_i$ are chosen in such a way that $\varepsilon_1 + \varepsilon_4 \neq \varepsilon_2 + \varepsilon_3$, then not only is the partition not equilateral, but no sufficiently small perturbation will change this.
\end{remark}

\bibliographystyle{alpha}

\end{document}